\documentclass[12pt]{article}
\usepackage{times}
\usepackage{graphicx}
\usepackage{amsthm}
\usepackage{amsmath}
\usepackage{amssymb}
\usepackage{epstopdf}
\usepackage{enumerate}
\usepackage{amscd}
\usepackage[cmtip,arrow]{xy}
\usepackage{pb-diagram,pb-xy}
\theoremstyle{plain}
\newtheorem{theorem}{Theorem}[section]
\newtheorem{lemma}[theorem]{Lemma}
\newtheorem{proposition}[theorem]{Proposition}
\newtheorem{corollary}[theorem]{Corollary}

\theoremstyle{definition}
\newtheorem{definition}[theorem]{Definition}
\newtheorem{example}[theorem]{Example}

\newtheorem{remark}[theorem]{Remark}

\def\ss{\smallskip}

\def\be{\begin{enumerate}}
\def\ee{\end{enumerate}}
\def\bi{\begin{itemize}}
\def\ei{\end{itemize}}

\newcommand{\Z}{\mathbb{Z}} 
\newcommand{\Q}{\mathbb{Q}} 
\newcommand{\R}{\mathbb{R}} 
\newcommand{\s}[1]{\mathbb{S}^{#1}} 

\newcommand{\cA}{\mathcal{A}} 
\newcommand{\cB}{\mathcal{B}} 
\newcommand{\cC}{\mathcal{C}}
\newcommand{\cD}{\mathcal{D}}
\newcommand{\cH}{\mathcal{H}}
\newcommand{\cI}{\mathcal{I}}
\newcommand{\cJ}{\mathcal{J}}
\newcommand{\cM}{\mathcal{M}}
\newcommand{\cN}{\mathcal{N}}
\newcommand{\cP}{\mathcal{P}}
\newcommand{\cQ}{\mathcal{Q}} 
\newcommand{\cR}{\mathcal{R}}
\newcommand{\cS}{\mathcal{S}}
\newcommand{\cU}{\mathcal{U}}
\newcommand{\cV}{\mathcal{V}}

\newcommand{\cZ}{\mathcal{Z}}
\renewcommand{\t}{^t} 
\newcommand{\F}{F} 
\renewcommand{\mod}{\mathrm{mod}\ } 
\newcommand{\ie}{{\it i.e. }} 
\newcommand{\eg}{{\it e.g. }} 
\newcommand{\id}{1}
\newcommand{\rk}{\mathrm{rank}\ } 
\newcommand{\tor}{\mathrm{Tor}\ } 
\newcommand{\End}{\mathrm{End}\ } 
\newcommand{\Aut}{\mathrm{Aut}\ } 
\newcommand{\Sp}{\mathrm{Sp}} 
\newcommand{\genus}{\mathrm{genus}\ } 
\newcommand{\Ybar}{\bar{Y}}
\newcommand{\Nbar}{\bar{N}}
\newcommand{\Fbar}{\bar{F}}
\newcommand{\Bbar}{\bar{B}}

\newcommand{\T}{\mathsf{T}} 
\newcommand{\U}{\mathsf{U}} 
\newcommand{\N}{\mathbf{N}} 
\newcommand{\tr}{\mathrm{Tr}} 
\newcommand{\diag}{\mathrm{Diag}} 
\newcommand{\Hom}{\mathrm{Hom}} 
\newcommand{\leg}[2]{{\left(\frac{#1}{#2}\right)}} 
\newcommand{\units}{\mathbb{U}} 
\newcommand{\vectortwoone}[2]{
\left(\begin{smallmatrix}
#1 \\ #2
\end{smallmatrix}\right)}
\newcommand{\SL}{\mathrm{SL}} 
\newcommand{\GL}{\mathrm{GL}} 
\newcommand{\mattwotwo}[4]{
\left(\begin{smallmatrix}
#1 & #2 \\ #3 & #4
\end{smallmatrix}\right)}
\newcommand{\bmattwotwo}[4]{
\left(\begin{matrix}
#1 & #2 \\ #3 & #4
\end{matrix}\right)}
\newcommand{\bvectortwoone}[2]{
\left(\begin{matrix}
#1 \\ #2
\end{matrix}\right)}

\newcommand\Ab{\mbox{ab}}
\newcommand\Th{\mbox{th}}

\begin{document}
\pagestyle{myheadings}
\author{Joan S. Birman\footnote{Supported in part by NSF Grant  DMS-040558}, \ Dennis Johnson,\ and Andrew Putman}

\title{Symplectic Heegaard splittings and linked abelian groups}
\date{04/14/08}

\maketitle  
\newpage
\tableofcontents

\flushleft
\setlength{\parskip}{2ex plus 0.5ex minus 0.2ex}
\section{Introduction}
\label{S:introduction}
\subsection{The Johnson-Morita filtration of the mapping class group}
\label{SS:the Johnson-Morita filtration of the mapping class group}
Let $M_g$ be a closed oriented 2-manifold of genus $g$ and $\tilde{\Gamma}_g$ be its mapping class group, that is, the group
of isotopy classes of orientation-preserving diffeomorphisms of $M_g$.  Also, let $\pi = \pi_1(M_g)$, and denote by $\pi^{(k)}$ the 
$k^{\Th}$ term in the lower central series of $\pi$, i.e.\ $\pi^{(1)} = \pi$ and $\pi^{(k+1)} = [\pi,\pi^{(k)}]$.  
Then $\tilde{\Gamma}_g$ acts on the quotient groups $\pi/\pi^{(k)}$, and that action yields a representation 
$\rho^k:\tilde{\Gamma}_g \to \Gamma_g^k$, where $\Gamma_g^k < \Aut(\pi/\pi^{(k)})$.  
With these conventions, $\rho^1$ is the trivial representation and 
$\rho^2$ is the symplectic representation.  The kernels of these representations make what has been called the `Johnson filtration' 
of $\tilde{\Gamma}_g$, because they were studied by Johnson in \cite{Johnson1983b,Johnson1985}. Subsequently they were developed by Morita 
in a series of papers \cite{Morita1993a,Morita1993b,Morita1996,Morita2001}.  In particular, Morita studied the extensions of Johnson's homomorphisms to $\tilde{\Gamma}$, and so we refer to our representations as the {\it Johnson-Morita} filtration of $\tilde{\Gamma}$. 

Our work in this article is motivated by the case when $M_g$ is a Heegaard surface in a 3-manifold $W$ and elements of 
$\tilde{\Gamma}_g$ are `gluing maps' for the Heegaard splitting.  One may then study $W$ by investigating the image
under the maps $\rho^k$ of the set of all possible gluing maps that yield $W$.  Among the many papers which relate to 
this approach to 3-manifold topology are those of Birman \cite{Birman1975}, 
Birman and Craggs \cite{BC}, Brendle and Farb \cite{BF}, Broaddus, Farb and Putman \cite{BFP}, Cochran, Gerges and Orr \cite{CGO}, 
Garoufalidis and Levine \cite{GL}, Johnson \cite{Johnson1980, Johnson1980b, Johnson1983, Johnson1985}, Montesinos and Safont \cite{MS}, Morita \cite{Morita1991}, 
Pitsch \cite{Pitsch}, \cite{Pitsch2}, and Reidemeister \cite{R2}.  The papers just referenced relate to the cases $k=$1-4 
in the infinite sequence of actions of $\tilde{\Gamma}_g$ on the quotient groups of the lower central series, but the 
possibility is there to study deeper invariants of $W$, obtainable in principle from deeper quotients of the lower central series.   
The foundations for such deeper studies have been laid in the work of  Morita\cite{Morita1993a, Morita1993b}, who introduced the 
idea of studying higher representations via crossed homomorphisms.  It was proved by Day \cite{Day} that the crossed product 
structure discovered by Morita in the cases $k=3$ and $4$ can be generalized to all $k$, enabling one in principle to 
separate out, at each level, the new contributions.

The invariants of 3-manifolds that can be obtained in this way are known to be closely related to finite type invariants of 
3-manifolds \cite{CM,GL}, although as yet this approach to finite type invariants opens up many more questions than answers.
For example, it is known that the Rochlin and Casson invariants of 3-manifolds appear in this setting at levels 3 and 4, respectively. 
It is also known that in general there are finitely many linearly independent finite-type invariants of 3-manifolds at each fixed order 
(or, in our setting, fixed level) $k$, yet at this moment no more than one topological invariant has been encountered
at any level.

The simplest non-trivial example of the program mentioned above is the case $k=2$.  Here $\tilde{\Gamma}_g$ acts on 
$H_1(M_g) = \pi/[\pi,\pi]$.  The information about $W$ that is encoded in $\rho^2(\phi)$, where $\phi \in {\rm Diff}^+(M_g)$ 
is the Heegaard gluing map for a Heegaard splitting of $W$ of minimum genus, together with the images under 
$\rho^2$ of the Heegaard gluing maps of all `stabilizations' of the given splitting, is what we have in mind when we refer 
to a `symplectic Heegaard splitting'.    

The purpose of this article is to review the literature on symplectic Heegaard splittings of 3-manifolds and the closely related 
literature on linked abelian groups, with the goal of describing what we know, as completely and explicitly and efficiently as 
possible, in a form in which we hope will be useful for future work.  At the same time, we will add a few new things that we 
learned in the process. That is the broad outline of what the reader can expect to find in the pages that follow.   

This article dates back to 1989.  At that time, the first two authors had discussed the first author's invariant of Heegaard splittings, 
in \cite{Birman1975}, and had succeeded in proving three new facts : first, that the invariant in \cite{Birman1975} could be 
improved in a small way; second, that the improved invariant was essentially the only invariant of Heegaard splittings that 
could be obtained from a symplectic Heegaard splitting; and third, that the index of stabilization of a symplectic Heegaard 
splitting is one.  That work was set aside, in partially completed form, to gather dust in a filing cabinet.  An early version 
of this paper had, however, been shared with the authors of \cite{MS} (and was referenced and acknowledged in \cite{MS}).  Alas, 
it took us 18 years to prepare our work for publication!  Our work was resurrected, tentatively,  at roughly the time of the conference on 
{\it Groups of Diffeomorphisms} that was held  in Tokyo September 11-15,  2006.  As it turned out, the subject still 
seemed to be relevant, and since a conference proceedings was planned, we decided to update it and complete it, in the hope 
that it might still be useful to current workers in the area.   When that decision was under discussion,  the manuscript was shared with the third author, who contributed many excellent suggestions, and also answered a question posed by the first author (see $\S$\ref{S:a remark on higher invariants}). Soon after that, he became a coauthor. 

\subsection{Heegaard splittings of 3-manifolds}
\label{SS:Heegaard splittings of 3-manifolds}

Let $W$ be a closed, orientable 3-dimensional manifold. A {\it Heegaard surface} in $W$ is a closed, orientable surface $M$ of genus $g \geqslant 0$ embedded in $W$ which divides $W$ into homeomorphic handlebodies $N \cup \bar{N}$, where $N \cap \bar{N} = \partial N = \partial \bar{N} = M$. For example, if $W$ is the 3-sphere 
$$\s{3} = \left\{ (x_1, x_2, x_3, x_4) \in \R^4 \ \vert \ x_1^2 +x_2^2 + x_3^2 +x_4^2 = 1 \right\},$$
then the torus
$$
M = \left\{ (x_1, x_2, x_3, x_4) \in \s{3} \ \vert \ x_1^2 +x_2^2 = x_3^2 +x_4^2 = \frac{1}{2} \right\}
$$
is a Heegaard surface. 

\begin{proposition}
\label{P:every W admits a Heegaard splitting}
Every closed orientable 3-manifold $W$ admits Heegaard splittings.
\end{proposition}

See \cite{Scharlemann}, for example, for a proof.  One will also find there related notions of Heegaard splittings of non-orientable 
3-manifolds, of open 3-manifolds such as knot complements, and of 3-manifolds with boundary, and also an excellent introduction to 
the topic and its many open problems from the viewpoint of geometric topology. 

Since any Heegaard splitting clearly gives rise to others under homeomorphisms of $W$, an equivalence relation is in order.  

\begin{definition} \label{D:equivalent Heegaard splittings}  \rm  Assume that $W$ is an oriented 3-manifold, and write $W = N \cup \bar{N} = N' \cup \bar{N}'$.  These two Heegaard 
splittings will be said to be {\it equivalent} if there is a homeomorphism $F: W \to W$ which restricts to homeomorphisms 
$f: N \to N'$ and $\bar{f} : \bar{N} \to \bar{N}'$.  Observe that our particular way of defining equivalent Heegaard splittings involve a choice of the initial handlebody $N$ and a choice of an orientation on $W$.   The {\it genus} of the splitting $W = N \cup \bar{N}$ is the genus of $N$.  $||$
\end{definition}
There are 3-manifolds and even prime 3-manifolds which admit more 
than one equivalence class of splittings (for example, see \cite{{Eng}, {BGM}}), there are also 3-manifolds which admit 
unique equivalence classes of splittings of minimal genus (\eg lens spaces and the 3-torus $\s{1}\times \s{1} \times \s{1}$),
and there are also 3-manifolds which admit unique equivalence classes of Heegaard splittings of every genus.  
A very fundamental example was studied by Waldhausen in \cite{W}, who proved:
\begin{theorem}[{\cite{W}}] \label{T:Waldhausen on S3}
Any two Heegaard splittings of the same, but arbitrary, genus of the 3-sphere $\s{3}$ are equivalent.
\end{theorem} 
After that important result became known, other manifolds were investigated. At the present writing, it seems correct to say that 
`most' 3-manifolds admit exactly one equivalence class of minimal genus Heegaard splittings.  On the other hand, many examples 
are known of manifolds that admit more than one equivalence class of splittings.  See, for example, \cite{M-S}, where all 
the minimal genus Heegaard splittings of certain Seifert fiber spaces are determined.  

If a three-manifold admits a Heegaard splitting of genus $g$, then it also admits of one genus $g'$ for every $g'>g$.  To see why 
this is the case, let $N_g \cup \bar{N}_g$ be a Heegaard splitting of $W$ of genus $g$, and let $T_1 \cup \bar{T}_1$ be a 
Heegaard splitting of the 3-sphere $\s{3}$ of genus 1. Remove a 3-ball from $W$ and a 3-ball from $\s{3}$, choosing these 3-balls 
so that they meet the respective Heegaard surfaces in discs.  Using these 3-balls to form the connected sum $W\# \s{3}$, 
we obtain a new Heegaard splitting $(N_g\# T_1) \cup (\bar{N}_g\# \bar{T}_1)$ of $W\cong W\# \s{3}$ of genus $g+1$. 
This process is called {\it stabilizing} a Heegaard splitting.  Note that Theorem \ref{T:Waldhausen on S3} implies 
that the equivalence class of the new genus $g+1$ splitting is independent of the choice of $T_1$ and of $\bar{T}_1$, 
as subsets of $\s{3}$, since all splittings of $\s{3}$ of genus 1, indeed of any genus, are equivalent.  Iterating 
the procedure, we obtain splittings $(N_g\# T_1\# \cdots \# T_1) \cup (\bar{N}_g\# \bar{T}_1\# \cdots \# \bar{T}_1)$ of 
$W$ of each genus $g+k, k > 0$. 

Heegaard splittings of genus $g$ and $g'$  of a 3-manifold $W$ are said to be {\it stably equivalent} if they have equivalent 
stabilizations of some genus $g+k = g'+k'$.  In this regard, we have a classical result, proved in 1933 by Reidemeister \cite{R1}
and (simultaneously and independently) by James Singer \cite{Si}:
\begin{theorem}[\cite{R1}, \cite{Si}] 
\label{T:RS}
Any two Heegaard splittings of any closed, orientable 3-manifold $W$ are stably equivalent.
\end{theorem}
\begin{remark} \label{R:inequivalent Heegaard splittings} \rm  We distinguish two types of candidates for inequivalent minimum genus Heegaard splittings of a 3-manifold. The first (we call it {\it ordinary}) is always present: two splittings which differ in the choice of ordering of the two handlebodies, i.e. $N\cup \bar{N}$ in one case and $\bar{N}\cup N$ in the other. Two ordinary Heegaard splittings  may or may not be equivalent.  The second are {\it all examples which are not ordinary}, e.g. the `horizontal' and `vertical' Heegaard splittings of certain Seifert fibered spaces \cite{M-S}.  In view of the fact that Theorem~\ref{T:RS}
was proved in 1933, it seems remarkable that the following situation exists, as we write in 2008: 
\begin{itemize}
\item The only examples of inequivalent minimal genus Heegaard 
splittings of genus $g$ of the same 3-manifold which can be proved to require more than one stabilization before they become equivalent are ordinary examples;  
\item The discovery of the first ordinary examples which can be proved to require more than one stabilization was made in 2008  \cite{HTT}.   
\item While non-ordinary examples have been known for some time, at this writing there is no known pair which do not not become equivalent after a single stabilization. For example, the inequivalent minimal genus Heegaard splittings of Seifert fiber spaces which were studied in \cite{M-S} were proved 
in \cite{Schultens1996} to become equivalent after a single stabilization.   
\end{itemize} Note that ordinary  examples can be ruled out by a small change in the definition of equivalence, although we have chosen not to do so, because the situation as regards ordinary examples is still far from being understood.  $||$
\end{remark}
Keeping Remark \ref{R:inequivalent Heegaard splittings} in mind, we have several classical problems 
about Heegaard splittings:
\bi
\item How many stabilizations are needed before two inequivalent Heegaard splittings of a 3-manifold become equivalent, 
as they must because of Theorem~\ref{T:RS}?  Is there a uniform bound, which is independent of the choice of $W$ and 
of the Heegaard surface $\partial N = \partial\bar{N}$ in $W$?
\item Can we use {\it stabilized} Heegaard splittings to find topological invariants of 3-manifolds?
\ei
An example of a 3-manifold invariant which was discovered with the help of Heegaard splittings is Casson's invariant \cite{AM}.  

A Heegaard splitting of a 3-manifold $W$ is said to have {\it minimal genus} (or simply to be {\it minimal}) if there do not exist 
splittings of $W$ which have smaller genus.  Our second problem involves Heegaard splittings which are not stabilized.  Since it 
can happen that a Heegaard splitting of a 3-manifold $W$ is non-minimal in genus, but is not the stabilization of a Heegaard splitting 
of smaller genus (a complication which we wish to avoid), we assume from now on that wherever we consider {\it unstabilized} 
Heegaard splittings, we assume the genus to be minimal over all Heegaard splittings of the particular manifold.  This brings us to 
another problem:
\bi
\item Can we find invariants of {\it unstable} Heegaard splittings, and so reach a better understanding of the 
classification of Heegaard splittings?  
\ei 
Surprisingly, such invariants are very hard to come by, and little is known. 
\subsection{Symplectic Heegaard splittings} 
 \label{SS:symplectic Heegaard splittings}

We begin by setting up notation that will be used throughout this paper.  We will use a standard model for a symplectic space and for 
the symplectic group $\Sp(2g,\Z)$.  Let $N_g$ be a handlebody.  Then $H_1(\partial N_g)$ is a free abelian group of rank $2g$. 
Thinking of it as a vector space, the free abelian group $H_1(N_g;\Z)$ is a subspace.  We choose as basis elements for the 
former the ordered array of homology classes of the loops $a_1,\dots,a_g, b_1,\dots,b_g$ which are depicted in 
Figure 1.  With our choices, the images of the $a_i$ under the inclusion map $\partial N_g \to N_g$ are a basis for 
$H_1(N_g;\Z)$.   
\begin{figure}[htpb!]
\centerline{\includegraphics[scale=1.2] {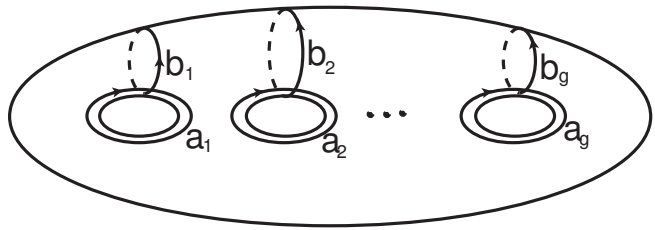}}
\caption{Curves representing a canonical basis for $H_1(\partial N_g)$}
\end{figure}
The algebraic intersection pairing $(\cdot,\cdot)$ defines a symplectic form on $H_1(\partial N_g;\Z)$, making it into a 
symplectic space.  The matrix of intersection numbers for our canonical basis is 
$\cJ = \bmattwotwo{0_g}{\cI_g}{-\cI_g}{0_g}$, where $0_g$ and $\cI_g$ are the $g \times g$ zero and identity matrices.  
\begin{definition} 
\label{D:symplectic constraints}
$\Sp(2g,\Z)$ is the group of all $2g\times 2g$ matrices $\cH = \bmattwotwo{\mathcal R}{\mathcal P}{\mathcal S}{\mathcal Q}$ over $\Z$ which satisfy
\begin{equation} 
\label{E:symplectic constraints-1}
\hat{\cH} \ \cJ \ \cH = \cJ
\end{equation}
where $\hat{\cH}$ denotes the transpose of $\cH$.
Hence
$\cH\in \Sp(2g,\Z)$ if and only if its $g \times g$ blocks $\cR, \cP, \cS, \cQ$ satisfy
\begin{equation}
\label{E:symplectic constraints-2}
 \hat{\cR} \cS, \hat{\cP} \cQ, \cR \hat{\cP} \text{ and $\cS \hat{\cQ}$ are symmetric,  and } 
\hat{\cR} \cQ - \hat{\cS} \cP =  \cR \hat{\cQ} - \cP \hat{\cS} = \mathcal I. 
\end{equation}   
Note that $\cH \in \Sp(2g,\Z)$ if and only if $\hat{\cH} \in \Sp(2g,\Z)$.  $\|$
\end{definition}
\begin{lemma}
\label{L:mapping class group maps onto Sp}
The group $\Gamma_g$ (i.e. the image of the mapping class group under $\rho^2$) coincides with $\Sp(2g,\Z)$.
\end{lemma}
\begin{proof} The fact that elements of $\Gamma_g$ satisfy the constraints in (\ref{E:symplectic constraints-2}) comes from the fact 
that topological mappings preserve algebraic intersection numbers.  The fact that {\it every} symplectic matrix is in the image of 
$\rho^2$ can be proven by combining the classical fact that $\Sp(2g,\Z)$ is generated by symplectic transvections with the
fact that every such symplectic transvection is the image of a Dehn twist.  This fact was used by Humphries, in his 
famous paper \cite{Humphries}, to find a lower bound on the number of Dehn twists needed to generate the mapping class group.  
He used the known fact that $\Gamma_g$ cannot be generated by fewer than $2g+1$ transvections.  
\end{proof}
\begin{lemma}
\label{L:the subgroup Lambda}  Let $\Lambda_g$ be the subgroup of matrices in $\Gamma_g$ which are induced by topological 
mappings of $\partial N_g$ which extend to homeomorphisms of $N_g$ (the so-called {\it handlebody subgroup}).
Then  $\Lambda_g$ coincides with the subgroup of all elements in $\Gamma_g$ with a $g \times g$ block of zeros in the upper right corner. 
 \end{lemma}
 \begin{proof}
By our choice of a basis for $H_1(\partial N_g;\Z)$, a topologically induced automorphism  of  $H_1(\partial N_g;\Z)$ extends to an automorphism of $H_1(N_g;\Z)$
only if it preserves the kernel of the inclusion-induced homomorphism $H_1(\partial N_g)\to H_1(N_g)$, i.e. the subgroup generated by $b_1,\dots,b_g$.  Sufficiency is proved by finding generators for the group $\Lambda_g$, given in \cite{Newman}, and showing that each comes from a topological mapping on $\partial N_g$ which extends to a homeomorphism of $N_g$. Explicit lifts are given in \cite{Birman1975}.
\end{proof}

In $\S$~\ref{SS:Heegaard splittings of 3-manifolds}, we saw that every closed orientable 3-manifold admits Heegaard splittings.  Let us now choose coordinates to make this more explicit.  Let $N = N_g$ be a standard model for an oriented handlebody of genus $g$, and 
let $\bar{N} = \phi(N)$ be a copy of $N$, where $\phi$ is a fixed orientation-reversing homeomorphism. 
(Note that representative diffeomorphisms are always required to be orientation-preserving.)  Choosing any element $\tilde{h} \in {\rm Diff}^+(\partial N_g)$, we may then construct  a 3-manifold $W$ as the disjoint union of $N_g$ and $\bar{N_g}$, glued together by the rule $\phi \circ \tilde{h} (x) = x, \ x \in \partial N_g$. To stress the role of $\tilde{h}$ we will write $W = N_g \cup_{\phi \circ \tilde{h}} \bar{N_g}$.  With these conventions, if we choose $\tilde{h}$ to be the identity map, the manifold $W$ will be the connect sum of $g$ copies of $\s{2} \times \s{1}$.  The mapping class group $\tilde{\Gamma}_g$ now means $\pi_0{\rm Diff}^+(\partial N_g)$.

Now let $\tilde{\Lambda} = \tilde{\Lambda}_g$ denote the subgroup of $\tilde{\Gamma}_g$ consisting of mapping classes which have a representative which extends to a homeomorphism of $N_g$.    Note that every map of $\partial N_g$ which is isotopic to the identity extends, hence if one representative extends then so does every other representative, so $\tilde{\Lambda}_g$ is well-defined.

\begin{proposition} 
\label{P:Heegaard splittings and double cosets}
Equivalence classes of genus $g$ Heegaard splittings of 3-manifolds are in 1-1 correspondence with double cosets in the sequence of groups $\tilde{\Gamma}_gÊ\ \mod  \tilde{\Lambda}_g$.
\end{proposition}

\begin{proof}
Each Heegaard splitting of a 3-manifold determines a (non-unique) $\tilde{h} \in \tilde{\Gamma}_g$ for some $g$, and each 
$\tilde{h} \in \tilde{\Gamma}_g$ determines a 3-manifold $W = N_g \cup_{\phi \circ \tilde{h}} \bar{N_g}$.  Suppose 
$N_g \cup_{\phi\circ \tilde{h}} \bar{N_g}$ and $N_g' \cup_{\phi\circ \tilde{h'}} \bar{N_g}'$ are equivalent splittings 
of a 3-manifold $W$.  Then there is an equivalence $F$ which restricts to equivalences $f, \bar{f}$ on $N_g, \bar{N_g}$ and 
then to $f_0 = f \vert_{\partial N_g}, \bar{f}_0 = \bar{f} \vert_{\partial \bar{N_g}}$.  There is thus a commutative diagram
$$
\begin{CD}
\partial N_g @>\tilde{h}>> \partial N_g @>\phi>> \partial \bar{N_g} \\
@VVf_0V & & @VV\bar{f}_0V \\
\partial N_g @>\tilde{h'}>> \partial N_g @>\phi>> \partial \bar{N_g}
\end{CD}
$$
Then $\tilde{h'} f_0 = \phi^{-1} \bar{f}_0 \phi \tilde{h}$, hence $\tilde{h'} \in \tilde{\Lambda} \tilde{h} \tilde{\Lambda}$. Conversely, if $\tilde{h'} \in \tilde{\Lambda}\tilde{h} \tilde{\Lambda}$ then $\tilde{h'} f_0 = \phi^{-1} \bar{f}_0 \phi \tilde{h}$ for some $f_0, \phi^{-1} \bar{f}_0 \phi \in \tilde{\Lambda}$. Let $f, \phi^{-1} \bar{f} \phi$ be an extension of $f_0, \phi^{-1} f_0 \phi$ to $N_g$. Define $F \vert_{N_g} = f$, $F\vert_{\bar{N_g}} = \bar{f}$. 
\end{proof}
For convenience, we will not distinguish between the diffeomorphism $\tilde{h}$ and the mapping class it determines in $\tilde{\Gamma}_g$. 

\begin{corollary} \label{C:alg-RS}
Let $W = N_g \cup_{\phi \circ \tilde{h}} \bar{N}_g$ and let $W' = N_{g'} \cup_{\phi \circ\tilde{h'}} \bar{N_{g'}}$.  Let $\tilde{s}$ 
be any choice of gluing map for a genus 1 splitting of $\s{3}$.  Then $W$ is homeomorphic to $W'$ if and only if there are 
integers $k, k'$ with $g+k = g'+k'$ so that $\tilde{h}\#_k \tilde{s}$ is in the same double coset of 
$\tilde{\Gamma}_{g+k} \ {\rm mod} \ \tilde{\Lambda}_{g+k}$ as $\tilde{h'}\#_{k'} \tilde{s}$.
\end{corollary}

\begin{proof}  This follows directly from Theorem~\ref{T:RS}.  \end{proof}

\begin{corollary}\label{C:topological invariants of 3-manifolds} 
Let $W$ be a closed, orientable 3-dimensional manifold which is defined by any Heegaard splitting of genus $g$ with Heegaard gluing 
map $\tilde{h}$.  Then invariants of the stable double coset of $\tilde{h}$ in $\tilde{\Gamma}_g$ are topological invariants of 
the 3-manifold $W$. 
\end{corollary}

\begin{proof} 
This is a direct consequence of Proposition~\ref{P:every W admits a Heegaard splitting},
Proposition~\ref{P:Heegaard splittings and double cosets}, and 
Corollary~\ref{C:alg-RS}.
\end{proof}

We pass to the action of $\tilde{\Gamma}_g$ on $\pi_1(\partial N_g)/[\pi_1(\partial N_g),\pi_1(\partial N_g)]$, 
i.e.\ to the representation $ \rho^2:\tilde{\Gamma}_g\to\Gamma_g$.  What information might we expect to detect 
about Heegaard splittings from the image $\rho^2(\tilde{h})$ of our gluing map $\tilde{h}$ in $\Gamma_g$?  

\begin{definition}
\label{D:stabilization of HHS}
A {\it stabilization of index $k$ of $\cH$} is the image of $\cH\in \Gamma_g$ under the embedding $\Gamma_g \to \Gamma_{g+k}$ defined by bordering $\cR, \cP, \cS, \cQ$ according to the rule 
$$ \cR \mapsto 0_k \oplus \cR, \quad 
\cP \mapsto \mathcal I_k \oplus \cP, \quad
\cS \mapsto - \mathcal I_k \oplus \cS, \quad
\cQ \mapsto 0_k \oplus \cQ. $$
This is a particular way of taking the direct sum of $\cH\in\Gamma_g$ with the matrix $\cJ\in\Gamma_1$, which is the image under 
$\rho^2$ of a Heegaard gluing map that defines $\s{2}$.

Define $\cH, \cH' \in \Gamma_g$ to be {\it equivalent} ($\cH \simeq \cH')$ if $\cH' \in \Lambda_g \cH\Lambda$ and 
{\it stably equivalent} $(\cH \simeq_s \cH')$ if $\cH$ and $\cH'$ have equivalent stabilizations for some index 
$k \geqslant 0$. Equivalence classes are then double cosets in $\Gamma_g \ \mod \Lambda_g$ and 
stable equivalence classes are double cosets in $\Gamma_{g+k}$ modulo $\Lambda_{g+k}$.

A {\it stabilized} symplectic Heegaard splitting is the union of all stabilizations of the double coset $\Lambda_g \cH \Lambda_g$.  $\|$
\end{definition}

This brings us to the main topic of this article.   Choose any $\tilde{h}\in\tilde{\Gamma}_g$ and use it to construct 
a 3-manifold $W$ as above.   Let $\cH$ be the symplectic matrix that is induced by the 
action of $h = \rho^2(\tilde{h})$. 
\begin{definition}
\label{D:HHS}
A {\it symplectic Heegaard splitting} of the 3-manifold $W = N_g \cup_{\phi\circ\tilde{h}} \bar{N_g}$ is the 
double coset $\Lambda_g \cH \Lambda_g \subset \Sp(2g,\Z)$, together with the double cosets of all 
stabilizations of $\cH$.  A symplectic Heegaard splitting
is {\it minimal} if it is not the stabilization of a symplectic Heegaard splitting of lower genus which is in the same double coset. $\|$
\end{definition}

\subsection{Survey of the literature} 
\label{SS:survey of the literature}

The earliest investigation of Heegaard splittings were the proofs, by Singer \cite{Si} and Reidemeister \cite{R1}  that all Heegaard splittings of an arbitrary 3-manifold are stably equivalent.  Shortly after the publication of \cite{R1} Reidemeister asked about invariants of 3-manifolds that can be determined from a Heegaard splittings. His invariants are given in the paper \cite{R2}. He proves by an example (the Lens spaces) that the invariants he discovered distinguish manifolds which have the same fundamental group $\pi_1(W)$, and so are independent of the rank and torsion coefficients of $W$.
Reidemeister's invariants are determined from the action of a Heegaard gluing map on $H_1(W;\mathbb Z)$.   We will explain exactly what he proved at the end of $\S$\ref{SS:Reidemeister's invariants}.

Essentially simultaneously and independently of Reidemeister's work, Seifert \cite{Sei} introduced the concept of a linking form on a 3-manifold whose homology group has a torsion subgroup $T$, and studied the special case when $T$ has no 2-torsion, obtaining a complete set of invariants for linked abelian groups in this special case.  His very new idea was that linking numbers could be defined not just in homology spheres, but also in 3-manifolds whose $\mathbb Z$-homology group has torsion. 
Let $W$ be a closed, oriented 3-manifold and suppose that the torsion subgroup $T$ of $H_1 (W; \Z)$ is non-trivial. Let $a, b$ be simple closed curves in $W$ which represent elements of $T$ of order $\alpha, \beta$ respectively. Since $\alpha a, \beta b$ are homologous to zero they bound surfaces $A, B \subset W$. Let $A\cdot b$ denote the algebraic intersection number of $A$ with $b$, similarly define $B\cdot a$. The {\it linking number} $\lambda(a,b)$ of $a$ with $b$ is the natural number 
$$
\lambda(a,b) = \frac{1}{\alpha} A\cdot b = \frac{1}{\beta}a\cdot B.
$$
Seifert's invariants are defined in terms of an array of integer determinants associated to  the $p$-primary cyclic summands of $T$. The invariant depends upon whether each determinant in the array is or is not a quadratic residue mod $p^k$.  
His work is, however, restricted to the case when there is no 2-torsion.  In the appendix to \cite{Sei}, and also at the end of \cite{R2}, both Seifert and Reidemeister noted that their invariants are in fact closely related, although neither makes that precise.   Both \cite{R2} and \cite{Sei}  are, at this writing very well known but it takes some work to pin down  the precise relationship so that one can move comfortably between them.  See  $\S$\ref{SS:Reidemeister's invariants}.

In \cite{Bu} Burger reduced the problem of classifying linked $p$-groups ($p\geqslant 2$) to the classification of symmetric bilinear forms over $\Z_{p^n}$. His procedure, together with Minkowski's work on quadratic forms \cite{Min} gives a complete set of invariants for the case $p=2$, but they are inconveniently cumbersome. Our contribution here is to reduce Burger's invariants to a simple and useful set. 
Most of what we do is probably obtainable from Burger's work together with the work of O'Meara \cite{OMA1}; however, our presentation is unified and part of a systematic study, hence it may be more useful than the two references \cite{Bu} and \cite{OMA1}. We note that Kawauchi and Kojima \cite{KK} {\it also} studied linked abelian groups with 2-torsion. They obtained a solution of the problem which is similar to ours, however, their goal was different and the intersection between their paper and ours is small.

Invariants of Heegaard splittings, rather than of the manifold itself, were first studied in the context of symplectic Heegaard splittings, in \cite{Birman1975}. Later, the work in \cite{Birman1975} was further investigated in \cite{MS}, from a slightly different perspective, with two motivations behind their work.  The first is that they thought that linking forms in 3-manifold might give more information than intersection forms on a Heegaard surface, but that is not the case.   Second, they thought that, because a finite abelian group can be decomposed as a direct sum of cyclic groups of prime power order, whereas in \cite{Birman1975} the decomposition was as a direct sum of a (in general smaller) set of cyclic groups which are not of prime power order, that perhaps there were invariants of unstabilized Heegaard splittings that were missed in \cite{Birman1975}. The main result in \cite{MS} is that, with one small exception in the case when there is 2-torsion, the Heegaard splitting invariants in \cite{Birman1975} cannot be improved.

See \cite{L-M} for an invariant of Heegaard splittings which is related in an interesting way to our work in this paper.  The relationship will be discussed in $\S$\ref{S:a remark on higher invariants}  of this paper. 

\subsection{Six problems about symplectic Heegaard splittings}
\label{SS:6 problems}

In this article we will consider six problems about symplectic Heegaard splittings, giving complete solutions for the first five and a partial solution for the sixth:

{\bf Problem 1:}
Find a complete set of invariants for stabilized symplectic Heegaard splittings.  \\
The full solution is in Corollary~\ref{C:stably isomorphic iff linked quotients}, which asserts the well-known result that a complete set of invariants are the rank of $H_1(W;\Z)$, its torsion coefficients, and the complete set of linking invariants.   

{\bf Problem 2:}
Knowing the invariants which are given in the solution to Problem 1 above, the next step is to learn how to compute them.  Problem 2 asks for a constructive procedure for computing the invariants in Problem 1 for particular $\cH \in \Gamma$.  
The easy part of this, i.e. the computation of invariants which determine $H_1(W;\Z)$,  is given in Theorem~\ref{T:partial normal form}.  The hard part is in the analysis of the linking invariants associated to the torsion subgroup of $H_1(W;\Z)$. See $\S$6.2 for the case when $p$ is odd and $\S$6.3 for the case where there is 2-torsion.  

{\bf Problem 3:}  Determine whether there is a bound on the stabilization index of a 
symplectic Heegaard splitting.    We will prove that there is a uniform bound, and it is 1.   
See  Corollary~\ref{C:problem on stabilization index}. 

{\bf Problem 4:}
Find a complete set of invariants which characterize minimal (unstabilized) symplectic 
Heegaard splittings and learn how to compute them.   In Theorem~\ref{T:unstabilizedHeegaardpairs} we will prove that the only invariant is a strengthened form of the invariant which was discovered in \cite{Birman1975}, using very different methods.   Example~\ref{Example:inequivalent unstabilized splittings} shows that we have, indeed, found an invariant which is stronger than the one in \cite{Birman1975}.

{\bf Problem 5:}
Count the number of equivalence classes of minimal (unstabilized) symplectic Heegaard 
splittings.  The answer is given in Theorem~\ref{theorem:isomorphismcount}.

{\bf Problem 6:} This problem asks for a normal form which allows one to choose a unique representative for the collection of matrices in an  unstabilized  double coset in $\Gamma_g$ (mod $\Lambda_g$).   We were only able to give a partial solution to this problem.   In $\S$\ref{SS:remarks on problem 6} we  explain the difficulty.

In $\S$\ref{S:a remark on higher invariants} we go a little bit beyond the main goal of this paper, and consider whether the work in $\S$\ref{S:presentation theory for finitely generated abelian groups} of this paper can be generalized to the higher order terms in the Johnson-Morita filtration.  
As we shall see, the approach generalizes, but it does not yield anything new.    

\section{Symplectic matrices : a partial normalization}
\label{S:symplectic matrices:a partial normalization}

Our task in this section is the proof of Theorem~\ref{T:partial normal form}, which gives a partial 
solution to Problem 2 and tells us how to recognize when a symplectic Heegaard splitting is stabilized.  

\subsection{Preliminaries}
\label{SS:preliminaries} 
We follow the notation that we set up $\S$\ref{S:introduction}.  Let $\tilde{h}$ be the gluing map for a Heegaard splitting of a 3-manifold $W$.  We wish to study the double coset $\Lambda_g h \Lambda_g \subset \Gamma$.  For that it will be helpful to learn  a little bit more about the subgroup $\Lambda_g$.  Recall that,  by Lemma ~\ref{L:the subgroup Lambda}, the group $\Lambda_g$ is the subgroup of elements in $\Gamma_g$ with a $g \times g$ block of zeros in the upper right corner. 
\begin{lemma} 
\label{L:structure of Lambda}
The group $\Lambda_g$ is the semi-direct product of its normal subgroup
$$
\Omega = \left\{ \left. \bmattwotwo{\mathcal I}{0}{\cZ}{\mathcal I} \ \right| \ \cZ \text{ symmetric } \right\}
\text{and its subgroup}  \ \ 
\Sigma = \left\{ \left. \bmattwotwo{\cA}{0}{0}{\hat{\cA}^{-1}} \ \right| \ \cA \text{ unimodular } \right\}.
$$
\end{lemma}
\begin{proof}
Since a general matrix  $\mattwotwo{\cA}{0}{\cC}{\cD} \in  \Lambda_g$ is symplectic, it follows from (\ref{E:symplectic constraints-2}) that $\hat{\cA} \cD = \mathcal I$, hence $\hat{\cA} = \cD^{-1}$, so $\cA \in \GL(g,\mathbb{Z})$.  Since $\mattwotwo{\cA}{0}{0}{\cA^{-1}} \in \Lambda_g$, it follows that the most general matrix in $\Lambda_g$ has the form:
$$
\bmattwotwo{\cA}{0}{\cC}{\cA^{-1}} = 
\bmattwotwo{\cA}{0}{0}{\cA^{-1}} \bmattwotwo{\mathcal I}{0}{\cZ}{0} =
\bmattwotwo{\mathcal I}{0}{\hat{\cA}^{-1} \cZ \cA^{-1}}{\mathcal I} \bmattwotwo{\cA}{0}{0}{\hat{\cA}^{-1}}
,$$  
with $\cZ = \cA \cC$.  But then (by (\ref{E:symplectic constraints-2}) again) $\cZ$ must be symmetric.  A simple calculation reveals that  the conjugate of any element in $\Omega$ by an element in $\Lambda$ is in $\Omega$. The semi-direct product structure follows from the fact that both $\Sigma$ and $\Omega$ embed naturally in $\Lambda$, and that they generate $\Lambda$. 
\end{proof}

In several places in this article it will be necessary to pass between the two canonical ways of decomposing a finite abelian group $T$ into cyclic summands.  We record here the following well-known theorem:
\begin{theorem}[The fundamental theorem for finitely generated abelian groups]
\label {T:fundamental theorem for f.g. abelian groups} 
Let $G$ be a finitely generated abelian group. Then the following hold:
\begin{enumerate}[{\rm (i)}]
\item $G$ is a direct sum of $r$ infinite cyclic groups and a finite abelian group $T$. The group $T$ is a direct sum of $t$ finite cyclic subgroups $T^{(1)} \oplus \cdots \oplus T^{(t)}$, where $T^{(i)}$ has order $\tau_i$.  Each $\tau_i$ divides $\tau_{i+1}, 1\leq i \leq t-1.$    The integers $r, t, \tau_1,\dots,\tau_t$  are a complete set of invariants of the isomorphism class of $G$.   

Let $p_1,\dots,p_k$ be the prime divisors of $\tau_t$.  Then  each integer  $\tau_i, \ 1\leq i\leq t$ has a decomposition as a product of primes:
\begin{equation} \label{E:tau-i as a product of primes}
\tau_i = p_1^{e_{i,1}} p_2^{e_{i,2}} \cdots p_k^{e_{i,k}}, \quad 0\leqslant e_{1,d} \leqslant e_{2,d} \leqslant \cdots \leqslant e_{t,d}, \ \  {\rm for \ \ each} \ \ 1\leq d\leq k.
\end{equation}
\item 
$T$ is also a direct sum of $p$-primary groups $T(p_1) \oplus \cdots \oplus T(p_k)$.  Here each $T(p_d)$ decomposes in a unique way as a direct sum of cyclic groups, each of which has order a power of $p_d$. 
Focusing on one such prime $p_d, \ 1\leq d\leq k$, the group $T(p_d)$  is a sum of cyclic groups of orders $p_d^{e_{1,d}}, p_d^{e_{2,d}},\dots,p_d^{e_{t,d}}$, where the powers $e_{i,d}$ that occur are not necessarily distinct. That is, we have: 
\begin{equation}\label{E:inequality on prime powers}
e_{1,d}=e_{2,d}=\cdots=e_{{t_1},d}< e_{{t_1}+1,d} =\cdots = e_{{t_2},d} < \cdots < e_{{t_r}+1,d} = e_{{t_r}+2,d} =\cdots = e_{{t_{r+1}},d}.
\end{equation}
\item  Let $y_i$ be a generator of the cyclic group of order $\tau_i$ in (i) above.   Let $g_{i,d}$ be a generator of the cyclic group of order $p_d^{e_{i,d}}$ in (ii) above.  Note that there may be more than one group with this order.  Then the generators $g_{i,d}$ and  $y_i$, where $1\leq i \leq t$ and $1\leq d \leq k$ are related by:
\begin{equation}
\label{E:gid from y_i}
g_{i,d} = \Bigg( {\tau_i \over (p_d^{e_{i,d}})} \Bigg) y_i = (p_1^{e_{i,1}} p_2^{e_{i,2}} \cdots 
p_{d-1}^{e_{i,d-1}}p_{d+1}^{e_{i, d+1}}\cdots p_k^{e_{i,k}})y_i.
\end{equation}
\end{enumerate}
\end{theorem}
The following corollary  to statement (i) of Theorem~\ref{T:fundamental theorem for f.g. abelian groups}   allows
us to transform a presentation matrix for a finitely generated abelian group into a particularly simple form.
\begin{corollary}[{Smith normal form, see, e.g., \cite[Theorem II.9]{Newman}}]
\label{C:Smith normal form}
Let $\cP$ be any $g \times g$ integer matrix.  Then there exist $\cU,\cV \in \GL(g,\Z)$ so that
$\cU \cP \cV = \diag(1,\dots,1,\tau_1,\dots,\tau_t,0,\dots,0)$, where the $\tau_i$ are nonnegative integers  which are different from 1 and satisfying $\tau_i | \tau_{i+1}$ for
all $1 \leq i < g$.  The diagonal matrix is called the \underline{Smith normal form} of $\cP$.  Additionally,
the Smith normal form of a matrix is unique, so that in particular the torsion free rank $r$ (the number of zeros in the diagonal) and the torsion rank $t$  are unique. The number of $1's$ is the index of stabilization of the symplectic Heegaard splitting, which can vary.
\end{corollary}

\subsection{A partial normal form}
\label{SS:a partial normal form}
 
%
\begin{theorem}
\label{T:partial normal form}
Let 
\begin{equation}
\cH = \rho^2(\tilde{h}) = \bmattwotwo{\cR}{\cP}{\cS}{\cQ}
\end{equation}
\label{E:symplectic matrix cH}
 be the symplectic matrix associated to a given Heegaard 
splitting of a 3-manifold $W$, where $\tilde{h}$ is the Heegaard gluing map.  Then:
\begin{enumerate}
\item [{\rm (i)}] The $g$-dimensional matrix $P$ is a relation matrix for $H = H_1(W;\mathbb Z)$.  This is true, independent of the choice of $\cH$ in its double coset modulo $\Lambda_g$. Different choices correspond to different choices of basis for $H$.
\item [{\rm (ii)}] The double coset  $\Lambda_g\cH\Lambda_g$ has a 
representative:
\begin{equation}
\label{E:partial normal form for cH}
\cH' =
\left(
\begin{array}{ccc|ccc}
0 & 0 & 0 & \mathcal I & 0 & 0 \\
0 & \cR^{(2)} & 0 & 0 & \cP^{(2)} & 0 \\
0 & 0 & \mathcal I & 0 & 0 & 0 \\
\hline
- \mathcal I & 0 & 0 & 0 & 0 & 0 \\
0 & \cS^{(2)} & 0 & 0 & \cQ^{(2)} & 0 \\
0 & 0 & 0 & 0 & 0 & \mathcal I
\end{array}
\right),
\end{equation}
where $\cP^{(2)} = Diag`(\tau_1,\ldots,\tau_t)$ 
with the $\tau_i$ positive integers satisfying $\tau_i | \tau_{i+1}$ for $1 \leq i < t$.  
In this representation the submatrix 
\begin{equation}
\label{E:R2,P2,S2,Q2}
\cH^{(2)} = \left(\begin{array}{c|c}
\cR^{(2)} & \cP^{(2)} \\
\hline
\cS^{(2)} & \cQ^{(2)}
\end{array}
\right)
\end{equation}
is symplectic.  
\item  [{\rm (iii)}] The $t\times t$ matrix $\cP^{(2)}$ is a relation matrix for the torsion subgroup $T$ of $H$, which is a direct sum of cyclic groups of orders $\tau_1,\dots,\tau_t$.
The number $r$ of zeros in the lower part of the diagonal of $P^{(1)} = {\rm Diag}(1,\dots,1,\tau_1,\dots,\tau_t, 0,\dots 0)$ is the free rank of $H$ and the number of $1's$ is the index of stabilization of the splitting. In particular, a symplectic Heegaard splitting with defining matrix $\cH\in {\rm Sp}(2g,\mathbb{Z})$ is unstabilized precisely when  the diagonal matrix $\cP^{(1)}$ contains no unit entries.
\item  [{\rm (iv)}]  We may further assume that every entry $q_{ij}\in \mathcal Q^{(2)}$ and every entry $r_{ij}\in \mathcal R^{(2)}$ is constrained as follows. Assume that $i\leq j$. Then:

$$  0 \leq q_{ji} < \tau_j, \ \ \ q_{ij} = (\tau_j/\tau_i)q_{ji}, \ \ \ {\rm and} \ \ \ 0 \leq r_{ij} < \tau_i, \ \ \ 
r_{ji} = (\tau_j/\tau_i) r_{ij}.$$    

\end{enumerate}
\end{theorem}

\begin{proof}

\underline{Proof of (i).}  Apply the Mayer-Vietoris sequence to the decomposition of the 3-manifold 
$W$ that arises through the Heegaard splitting $N_g \cup_{\phi\circ\tilde{h}} \Nbar_g$. 

\underline{Proof of (ii).}  The proof is a fun exercise in manipulating symplectic matrices, but without lots of care the proof will not be very efficient. 

In view of Lemma~\ref{L:structure of Lambda}, the most general element in the double coset of $\cH = \mattwotwo{\cR}{\cP}{\cS}{\cQ}$ has the form
\begin{equation}
\label{E:most general element in Lambda H Lambda}
\cM = 
\bmattwotwo{\mathcal I}{0}{\cZ_1}{\mathcal I} 
\bmattwotwo{\cU}{0}{0}{\hat{\cU}^{-1}} 
\bmattwotwo{\cR}{\cP}{\cS}{\cQ}
\bmattwotwo{\hat{\cV}^{-1}}{0}{0}{\cV} 
\bmattwotwo{\mathcal I}{0}{\cZ_2}{\mathcal I}
= 
\bmattwotwo{*}{\cU \cP \cV}{*}{*}
\end{equation}
where $\cU, \cV$ are arbitrary matrices in GL(2,$\mathbb{Z})$.

Choose $\cU, \cV \in \GL(t, \Z)$ so that 
$$\cP^{(1)} = (I\oplus\cU\oplus I)( \cP)(I\oplus\cV\oplus I)= \diag (1, 1, \dots, 1, \tau_1, \tau_2, \dots, \tau_t, 0, \dots, 0)\in \GL(g,\Z).$$    
By Corollary \ref{C:Smith normal form}, this is always possible.  Let $\cP^{(2)} = \diag(\tau_1, \tau_2, \dots, \tau_t) \in \GL(t,\Z).$ 
Using (\ref{E:most general element in Lambda H Lambda}). we have shown that $\cH$ is in the same double coset as 
\begin{equation} 
\label{E:The matrix H(1)}
\cH^{(1)}  =  \bmattwotwo{\cR^{(1)}}{\cP^{(1)}}{\cS^{(1)}}{\cQ^{(1)}} = 
\left(
\begin{array}{ccc|ccc}
\cR_{11} & \cR_{12} & \cR_{13} & \mathcal I & 0 & 0 \\
\cR_{21} & \cR_{22} & \cR_{23} & 0 & \cP^{(2)} & 0 \\
\cR_{31} & \cR_{32} & \cR_{33} & 0 & 0 & 0 \\
\hline
\cS_{11} & \cS_{12} & \cS_{13} & \cQ_{11} & \cQ_{12} & \cQ_{13} \\
\cS_{21} & \cS_{22} & \cS_{23} & \cQ_{21} & \cQ_{22} & \cQ_{23} \\
\cS_{31} & \cS_{32} & \cS_{33} & \cQ_{31} & \cQ_{32} & \cQ_{33} 
\end{array}
\right)
\end{equation}
This is the first step in our partial normal form. 

It will be convenient to write $\cH^{(1)}$ in several different ways in block form. The first one is the block decomposition in (\ref{E:The matrix H(1)}). In each of the other cases, given below, the main decomposition is into square $g \times g$ blocks, and these blocks will not be further decomposed (although much later they will be modified):
\begin{equation}
\label{E:alternative decompositions of H(1)}
\cH^{(1)} = \left(
\begin{array}{cc|cc}
\cA_{11} & \cA_{12} & \cB_{11} & 0 \\
\cA_{21} & \cA_{22} & 0 & 0 \\
\hline
\cC_{11} & \cC_{12} & \cD_{11} & \cD_{12} \\
\cC_{21} & \cC_{22} & \cD_{21} & \cD_{22}
\end{array}
\right)
\ = \
\left(
\begin{array}{c|c}
\cA & \cB \\
\hline
\cC & \cD
\end{array}
\right)
\end{equation}
where
$$
\cA_{11} = 
\left(
\begin{array}{cc}
\cR_{11} & \cR_{12} \\
\cR_{21} & \cR^{(2)}
\end{array}
\right),
\quad
\cB_{11} = 
\left(
\begin{array}{cc}
\mathcal I & 0 \\
0 & \cP^{(2)}
\end{array}
\right),
\quad 
\cC_{11} = 
\left(
\begin{array}{cc}
\cS_{11} & \cS_{12} \\
\cS_{21} & \cS^{(2)}
\end{array}
\right), \dots
$$
$$
\cA_{12} = 
\left(
\begin{array}{c}
\cR_{13} \\
\cR_{23}
\end{array}
\right),
\quad
\cA_{21} = 
\left(
\begin{array}{cc}
\cR_{31} & \cR_{32}
\end{array}
\right),
\quad
\cA_{22} = \left( \cR_{33} \right), \dots
$$
In general, the blocks $\cR_{ij}, \cA_{ij}, \dots$ are {\it not} square, however $\cR^{(2)}, \cS^{(2)}, \cP^{(2)}, \cQ^{(2)}$ are square $t \times t$ matrices.

Now $\cH^{(1)} \in \Gamma_g$, hence its $g \times g$ block satisfy the conditions (\ref{E:symplectic constraints-1}) and (\ref{E:symplectic constraints-2}). Working with the decomposition of $\cH^{(1)}$ into the block form  given in (\ref{E:The matrix H(1)}), one sees that because of the special form of $\cB = \mattwotwo{\cB_{11}}{0}{0}{0}$,  the $2(g - r) \times 2 (g - r)$ matrix  $\mattwotwo{\cA_{11}}{\cB_{11}}{\cC_{11}}{\cD_{11}}$
{\it also} satisfies (\ref{E:symplectic constraints-1}) and (\ref{E:symplectic constraints-2}), now with respect to its $(g - r) \times (g - r)$ block. From there it follows (using Definition~ \ref{D:symplectic constraints}) that the matrix given in (\ref{E:The matrix H(1)}) is in the group $\Gamma_{2(g-r)}$. One may then verify without difficulty that the augmented matrix
\begin{equation} 
\label{E:the matrix H(2)}
\cM^{(2)} = 
\left(
\begin{array}{cc|cc}
\cA_{11} & 0 & \cB_{11} & 0 \\
0 & \mathcal I & 0 & 0 \\
\hline
\cC_{11} & 0 & \cD_{11} & 0 \\
0 & 0 & 0 & \mathcal I
\end{array}
\right)
\ = \
\left(
\begin{array}{ccc|ccc}
\cR_{11} & \cR_{12} & 0 & \mathcal I & 0 & 0 \\
\cR_{21} & \cR^{(2)} & 0 & 0 & \cP^{(2)} & 0 \\
0 & 0 & \mathcal I & 0 & 0 & 0 \\
\hline
\cS_{11} & \cS_{12} & 0 & \cQ_{11} & \cQ_{12} & 0 \\
\cS_{21} & \cS^{(2)} & 0 & \cQ_{21} & \cQ^{(2)} & 0 \\
0 & 0 & 0 & 0 & 0 & \mathcal I 
\end{array}
\right),
\end{equation}
which has dimension $2g$ again, {\it also} satisfies the conditions (\ref{E:symplectic constraints-1}) and (\ref{E:symplectic constraints-2}), now with respect to its $g \times g$ block, and so $\cH^{(2)}$ is in $\Gamma_g$.

We will need further information about $\cH^{(1)}$ and $\cM^{(2)}$. Returning to (\ref{E:The matrix H(1)}), and using the right decomposition of $\cH^{(1)}$, we now verify that conditions (\ref{E:symplectic constraints-1}) and (\ref{E:symplectic constraints-2}) imply the following relations between the subblocks:
\begin{eqnarray}
& & \cQ_{13} = \cQ_{23} = \cR_{31} = \cR_{32} = 0 \label{equation:normalform11} \\
& & \cP^{(2)} \cQ_{21} = \hat{\cQ}_{12} \label{equation:normalform12} \\
& & \cR_{12} \cP^{(2)} = \hat{\cR}_{21} \label{equation:normalform13} \\
& & \cP^{(2)} \cQ^{(2)} \text{ symmetric} \label{equation:normalform14} \\
& & \cR^{(2)} \cP^{(2)} \text{ symmetric} \label{equation:normalform15} \\
& & \cQ_{11}, \cR_{11} \text{ symmetric} \label{equation:normalform16}
\end{eqnarray}
Now observe that if ${\mattwotwo{\cA}{\cB}{\cC}{\cD}} \in \Gamma_g$ then (\ref{E:symplectic constraints-1}) and (\ref{E:symplectic constraints-2}) imply that
\begin{equation}Ê
\label{E:inverse of a symplectic matrix}
{\bmattwotwo{\cA}{\cB}{\cC}{\cD}}^{-1} = {\bmattwotwo{\hat{\cD}}{-\hat{\cB}}{-\hat{\cC}}{\hat{\cA}}}.
\end{equation} 

Using equation (\ref{E:inverse of a symplectic matrix}) to compute $(\cM^{(2)})^{-1}$, and making use of the conditions in (\ref{equation:normalform11})-(\ref{equation:normalform16}), one may then verify that the product matrix $(\cM^{(2)})^{-1} \cH^{(1)}$ has a $g \times g$ block of zeros in the upper right corner. But then $(\cM^{(2)})^{-1} \cH^{(1)} \in \Lambda_g$, hence $\cM^{(2)}$ and $ \cH^{(1)}$ are in the same double coset.

Further normalizations are now possible. Since $\cR_{11}$ and $\cQ_{11}$ are symmetric (by the symplectic constraints (\ref{E:symplectic constraints-2})) the following matrices are in the subgroup $\Omega \subset \Lambda_g$ defined in Lemma \ref{L:structure of Lambda}:
$$
\cN_1 = 
\left(
\begin{array}{ccc|ccc}
\mathcal I & 0 & 0 & 0 & 0 & 0 \\
0 & \mathcal I & 0 & 0 & 0 & 0 \\
0 & 0 & \mathcal I & 0 & 0 & 0 \\
\hline
- \cQ_{11} & - \hat{\cQ}_{21} & 0 & \mathcal I & 0 & 0 \\
- \cQ_{21} & 0 & 0 & 0 & \mathcal I & 0 \\
0 & 0 & 0 & 0 & 0 & \mathcal I
\end{array}
\right)
\qquad \qquad
\cN_2 = 
\left(
\begin{array}{ccc|ccc}
\mathcal I & 0 & 0 & 0 & 0 & 0 \\
0 & \mathcal I & 0 & 0 & 0 & 0 \\
0 & 0 & \mathcal I & 0 & 0 & 0 \\
\hline
- \cR_{11} & - \cR_{21} & 0 & \mathcal I & 0 & 0 \\
- \hat{\cR}_{12} & 0 & 0 & 0 & \mathcal I & 0 \\
0 & 0 & 0 & 0 & 0 & \mathcal I
\end{array}
\right)
$$
Computing, we find that
$$
\cN_1 \cM^{(2)} \cN_2 = 
\left(
\begin{array}{ccc|ccc}
0 & 0 & 0 & \mathcal I & 0 & 0 \\
0 & \cR^{(2)} & 0 & 0 & \cP^{(2)} & 0 \\
0 & 0 & \mathcal I & 0 & 0 & 0 \\
\hline
* & * & 0 & 0 & 0 & 0 \\
* & \cS^{(2)} & 0 & 0 & \cQ^{(2)} & 0 \\
0 & 0 & 0 & 0 & 0 & 0
\end{array}
\right).
$$
Since this matrix is in $\Gamma_g$, its entries satisfy the conditions (\ref{E:symplectic constraints-1}) and (\ref{E:symplectic constraints-2}). An easy check shows that the lower left $g \times g$ box necessarily agrees with the entries in the matrix defined in the statement of Theorem \ref{T:partial normal form}. Thus $\cH^{(2)} = \cN_1 \cM^{(2)} \cN_2$ is in the same double coset as $\cM^{(2)}$, $\cH^{(1)}$ and $\cH$.   This completes the proof of (ii).

\underline {Proof of (iii).}  In (i) we saw that in the partial normal form the matrix $\mathcal I_{g-r-t} \oplus \mathcal P^{(2)} \oplus 0_r$ is a relation matrix for $H$.  Since $H$ is a finitely generated abelian group, it is a direct sum of  $t$ cyclic groups of order $\tau_1,\dots,\tau_t$ and $r$ infinite cyclic groups and $g-r-t$ trivial groups. The $g-r-t$ trivial groups indicate that the symplectic Heegaard splitting has been stabilized $g-r-t$ times.  That is, (iii) is true.

\underline{Proof of (iv).}  We consider additional changes in the submatrix $\cH^{(2)}$ which leave the $\mathcal P^{(2)}$-block unchanged.  Note that by Lemma~\ref{L:structure of Lambda}, any
 changes in the double coset of $\cH^{(2)}$ in $\Gamma_t$ can be lifted canonically to corresponding changes in the double coset of $\cH'$ in $\Gamma_g$, and therefore it suffices to consider modifications to the double coset of $\cH^{(2)}$ in $\Gamma_t$.   To simplify notation for the remainder of this proof, we 
 set $\cH^{(2)} = \mattwotwo{\mathcal R}{\mathcal P}{\mathcal S}{\mathcal Q}$ 

Choose any element $\mattwotwo{\cI_t}{0_t}{\cZ}{\cI_t} $ in the subgroup $\Omega_t$ of ${\rm Sp}(2t,\mathbb Z)$. Then:

$$  \bmattwotwo{\cI_t}{0_t}{\cZ}{\cI_t} \bmattwotwo{\cR}{\cP}{\cS}{\cQ}  = 
 \bmattwotwo{\cR}{\cP}{\star}{\cQ + \cZ P}, \ \ \ 
 \bmattwotwo{\cR}{\cP}{\cS}{\cQ} \bmattwotwo{\mathcal I_t}{0_t}{\cZ}{\mathcal I_t} =  \bmattwotwo{\cR+ \cP\cZ}{\cP}{\star}{\cQ} $$
 
Let $\cQ = (q_{ij}), \ \cR = (r_{ij}), \cZ = (z_{ij})$.  Then $\cZ \cP = (\tau_j z_{ij})$ and $\cP\cZ = (\tau_i z_{ij})$.
Therefore we may perform multiplications as above so that if $i \leq j$ then $0 \leq q_{ji} < \tau_j$ and $0 \leq r_{ij} < \tau_i$.   
The fact that $\cH^{(2)}$ is symplectic shows that $\cP \cQ$ and $\cR \cP$ are symmetric.   Therefore $q_{ij} = (\tau_j/\tau_i)q_{ji}$, $r_{ji} = (\tau_j/\tau_i) r_{ij}$.   Thus the matrices $\cQ$ and $\cR$ are completely determined once we fix the entries $q_{ji}$ and $r_{ij}$ which satisfy  $i\leq j$. This completes the proof of (iv), and so of Theorem~\ref{T:partial normal form}.
\end{proof}  

\begin{remark}
\label{R:4 double cosets}
As noted in $\S$\ref{SS:Heegaard splittings of 3-manifolds},  we made two choices when we defined equivalence of Heegaard splittings: the choice of one of the two handlebodies as the preferred one, and the choice of a preferred orientation on the 3-manifold $W$.   When we allow for all possible choices, we see that the symplectic matrix $\cH$ of 
Theorem~\ref{T:partial normal form}  is replaced by 4 possible symplectic matrices, related by the operations of taking the transpose and the inverse and the inverse of the transpose: 
$$
\bmattwotwo{\cR^{(2)}}{\cP^{(2)}}{\cS^{(2)}}{\cQ^{(2)}},\quad
\bmattwotwo{\hat{\cR}^{(2)}}{\hat{\cS}^{(2)}}{\hat{\cP}^{(2)}}{\hat{\cQ}^{(2)}},\quad
 \bmattwotwo{\hat{\cQ}^{(2)}}{-\hat{\cP}^{(2)}}{-\hat{\cS}^{(2)}}{\hat{\cR}^{(2)}},\quad
  \bmattwotwo{\cQ^{(2)}}{-\cS^{(2)}}{-\cP^{(2)}}{\cR^{(2)}}.
$$ 
Any one of the four could equally well have been chosen as a representative of the Heegaard splitting.  These four matrices may or may not be in the same double coset.  $\|$
\end{remark}

\subsection{Uniqueness questions}
\label{SS:uniqueness questions}

There is a source of non-uniqueness  in the partial normal form of Theorem~\ref{T:partial normal form}. It lies in the fact that further normalizations are possible after those in (iv) of Theorem~\ref{T:partial normal form}, but they are difficult to understand.   By Lemma~\ref{L:structure of Lambda}, we know that $\Lambda_t$ is the semi-direct product of the normal subgroup $\Omega_t$ and the subgroup $\Sigma_t$ that were defined there. We already determined how left and right multiplication by elements in $\Omega_t$ change $\cH'$ in the proof of part (iv) of Theorem~\ref{T:partial normal form}.  We now investigate further changes, using left (resp right) multiplication by matrices in $\Sigma_t.$  

\begin{lemma}
\label{L:improved normal form}
Assume that $\cP^{(2)}= {\rm Diag}(\tau_1,\dots,\tau_t)$ is fixed, and that $\cP^{(2)}\cQ^{(2)} = \hat{\cQ}^{(2)}\cP^{(2)}$. \ Then there is a well-defined subgroup $G$ of  $ \Sigma_t\times \Sigma_t$,  determined  by the condition that  there exist matrices $\cU,\cV\in {\rm GL}(t,\mathbb Z)$ such that $\cU\cP^{(2)} = \cP^{(2)} \cV.$   Equivalently, there exist symplectic matrices $U,V \in \Sigma_t$ such that $(U, V)\in G   \Longleftrightarrow $
\begin{equation}
\label{E:fiddling with P}
U \bmattwotwo{\cR^{(2)}}{\cP^{(2)}}{\cS^{(2)}}{\cQ^{(2)}}V =  \bmattwotwo{\cU}{0}{0}{\hat{\cU}^{-1}}  \bmattwotwo{\cR^{(2)}}{\cP^{(2)}}{\cS^{(2)}}{\cQ^{(2)}}\bmattwotwo{\cV}{0}{0}{\hat{\cV}^{-1}} =   
\bmattwotwo{\star}{\cP^{(2)}}{\star}{\hat{\cU}^{-1}\cQ^{(2)}\hat{\cV}^{-1}}.
\end{equation} 

For later use, we also have that if $(\cP^{(2)})^{-1}$ is the diagonal matrix whose $i^{\rm th}$ entry is the rational number 
$1/\tau_i$, then $\cQ^{(2)}(\cP^{(2)})^{-1}$ will be replaced by $\hat{\cU}^{-1}(\cQ^{(2)}(\cP^{(2)})^{-1})\cU^{-1}.$ 
\end{lemma}
 
\begin{proof}  The statement in (\ref{E:fiddling with P}) is a simple calculation.  We need to prove that it determines a group.  Suppose that $(U_1,V_1), \ (U_2,V_2) \in G$.  Then $\cU_i\cP^{(2)} = \cP^{(2)}\cV_i$ for $i=1,2$, so $\cU_1\cU_2 \cP^{(2)} = \cU_1\cP^{(2)}\hat{\cV}_2 = \cP^{(2)}\hat{\cV}_1\hat{\cV}_2$. Therefore  
$(U_1U_2, V_1V_2) \in G$.  Also, $(\cP^{(2)})^{-1} \cU_1^{-1} = \hat{\cV_1}^{-1}(\cP^{(2)})^{-1} $, which implies that $(U_1^{-1}, V_1^{-1})\in G$, so $G$ is a group. 
It is immediate that $\cP^{(2)}$ remains unchanged and that $\cQ^{(2)}(\cP^{(2)})^{-1}$ changes in the stated way.
\end{proof} 
\begin{remark} The condition $\cU\cP^{(2)} = \cP^{(2)}\hat{\cV}$ means that $\cU$ is restricted to $t \times t$ unimodular matrices which satisfy the condition: $\cU = (u_{ij})$, where $u_{ji}$ is divisible by 
$\tau_j/\tau_i$ whenever $j < i$.  There are no restrictions on $u_{ij}$ when $j\geq i$ other than that the determinant $|u_{ij}| = \pm 1$.   $\|$
\end{remark}

\begin{remark} \label{R:no unique double coset rep}  We were unable to find a nice way to choose $U$ and $V$ so as to obtain a unique representative of the double coset of a symplectic Heegaard splitting, in the case when $H$ is not torsion-free.   The reason will become clear in $\S$\ref{S:the classification problem for linked abelian groups}: invariants of the matrix $\cQ^{(2)}(\cP^{(2)})^{-1}$, and so also a normal form, depend crucially on whether or not there is 2-torsion in the torsion subgroup $T$ of $H$, and so a general rule cannot be easily stated.  See also the discussion in $\S$\ref{SS:remarks on problem 6}.      $\|$
\end{remark}

\newpage
 \section{Presentation theory for finitely generated abelian groups}
 \label{S:presentation theory for finitely generated abelian groups}
 
We are ready to begin the main work in this article.  In Section~\ref{S:introduction} we described the topological motivation that underlies the work in this paper, namely we were interested in understanding all topological invariants of a 3-manifold $W$ and of its Heegaard splittings that might arise through symplectic Heegaard splittings.  In Theorem~\ref{T:partial normal form} we saw that the matrix associated to a symplectic Heegaard splitting gives a natural presentation of  $H_1(W;\Z)$. Therefore it is natural to begin our work by investigating the theory of presentations of abelian groups.  Our goal is to understand, fully, that part of the obstruction to stable equivalence coming from a symplectic Heegaard splitting.    

We begin by introducing the two concepts of isomorphism and equivalence of presentations. The rank of a presentation is defined and a concept of stabilizing a presentation (thereby increasing the rank) is introduced. Most of this is aimed at Theorem \ref{T:h lifts to the presentation level}, which gives necessary and sufficient conditions for an automorphism of $H$ to lift to an automorphism of the free group of a presentation. Theorem \ref{T:h lifts to the presentation level} implies Corollaries \ref{C:cor 1 of T:h lifts to the presentation level}, \ref{C:cor 2 of T:h lifts to the presentation level} and \ref{C:cor 3 of T:h lifts to the presentation level}, which assert (in various ways) that any two presentations of the same finitely generated abelian group which are of non-minimal equal rank are equivalent. From this it follows that at most a single stabilization is required to remove any obstruction to equivalence between two presentations.  This does not
solve Problem 3, but it is a first step in the direction of this problem's solution.
We remark that, by contrast, the usual proof of Tietze's Theorem on equivalence of two particular presentations of an arbitrary finitely 
generated but in general non-abelian group shows that presentations of rank $r, r'$ become equivalent after 
stabilizations of index $r', r$, respectively \cite{MKS}.

 The latter half of the section focuses on presentations of minimal rank of a finitely generated abelian group. An ``orientation" and a ``volume" on $H$ are defined. The determinant of an endomorphism $h$ of abelian groups, and hence of a presentation $\pi$ of an abelian group, is introduced. The key result is Theorem \ref{T:equivalence of minimal presentations}, which gives necessary and sufficient conditions for an isomorphism between two {\it minimal} presentations to lift to the presentation level. Corollary \ref{C:complete invariant of equivalence of minimal presentations} follows: two minimal presentations of $H$ are equivalent if and only if they have the same volume on $H$. The section closes with two examples which illustrate the application of Theorem \ref{T:equivalence of minimal presentations} and Corollary \ref{C:complete invariant of equivalence of minimal presentations} to explicit group presentations.

In \S \ref{S:the classification problem for minimal symplectic Heegaard splittings}, we will apply the notion of ``volume'' to obtain invariants of Heegaard splittings.  It
turns out that associated to a symplectic Heegaard splitting is a natural presentation of the first homology group, and
thus an induced volume.  We will use the interplay between this volume and a linking form on the first homology group
to find invariants of Heegaard splittings. 
\subsection{Equivalence classes of not-necessarily minimal presentations} 
\label{SS:equivalence classes of not-necessarily minimal presentations}
We begin our work with several definitions which may seem unnecessary and even pedantic; however, extra care now will help to make 
what follows later seem natural and appropriate.

\begin{definition} \label{D:free pair}
 A \emph{free pair} is a pair of groups $(\F, R)$ with $R \subset \F$ and $\F$ free abelian and finitely generated; its 
\emph{quotient} is $\F/R$. If $H$ is a finitely generated abelian group, a \emph{presentation} of $H$ is a surjection 
$\pi: \F \to H$, with $\F$ again a finitely generated free abelian group. The \emph{rank} of a free pair and of a 
presentation is the rank of $\F$.  Direct sums of these objects are defined in the obvious way.
The {\it index k stabilization of a free pair} $(F,R)$ is the free pair $(F\oplus \Z^k, R\oplus \Z^k)$, and the 
{\it index $k$ stabilization of a presentation} $\pi:F\to H$ is the presentation $\pi\oplus 0:F\oplus \Z^k \to H$.   $\|$
\end{definition} 
  
\begin{definition}\label{D:isomorphism and equivalence of presentations}
An \emph{isomorphism of free pairs} $(\F, R), (\F', R')$ is an isomorphism $f : \F \to \F'$ such that $f(R) = R'$.  
An \emph{isomorphism of presentations} is a commutative diagram
 $$
 \begin{CD}
 \F @>\pi>> H \\
 @VfVV @VVhV \\
 \F' @>\pi'>> H'
 \end{CD}
 $$
with $f, h$ isomorphisms.  If $H = H'$, then we have the stronger notion of an \emph{equivalence of presentations}, which is a 
commutative diagram
 $$
 \begin{diagram}
 \node{\F} \arrow[2]{s,l}{f} \arrow{see,t}{\pi} \\
 \node[3]{H.} \\
 \node{\F'} \arrow{nee,b}{\pi'}
 \end{diagram}
 $$
with $f$ an isomorphism.  Two free pairs (resp. two presentations) are \emph{stably isomorphic} (resp. 
\emph{stably equivalent}) if they have isomorphic 
(resp. equivalent) stabilizations.  If $\pi: \F \to H$, $\pi': \F' \to H$ are both of minimal rank and stably equivalent, then we define 
the \emph{stabilization index} of $\pi, \pi'$ to be the smallest index $k$ such that $\pi, \pi'$ have equivalent stabilizations of 
index $k$.   $\|$
\end{definition}  

\begin{example}
\label{Ex:inequivalent minimal presentations-1} 
To see that equivalence and isomorphism of presentations are distinct concepts, let $F = \Z$ and let $H = \Z_5$ with 
$\pi : \Z \to \Z_5$ defined by $\pi(1) = 1$ and $\pi' : \Z \to \Z_5$ defined by $\pi'(1) = 2$.  Then $\pi$ and $\pi'$ are 
isomorphic because the automorphism $h: \Z_5 \to \Z_5$ defined by $h(1)=2$ lifts to the identity automorphism of $\Z$. 
However, it is easy to see that $\pi$ and $\pi'$ are not equivalent. $\|$
\end{example}

The standard `elementary divisor theorem' concerning presentations of abelian groups may be phrased as follows:
 
\begin{proposition} \label{P:structure of abelian groups}
Two free pairs are isomorphic if and only if they have the same rank and isomorphic quotients.  For any pair $(\F, R)$ 
there is a basis $f_i$ of $\F$ and integers $m_i$ so that $\{ m_i f_i \ | \ m_i \not= 0\}$ is a basis for $R$. 
The $f_i$ and $m_i$ may be chosen so that $m_i \vert m_{i+1}$ for all $i$.
\end{proposition} 
 
Since we can always stabilize two pairs to the same rank, two pairs are stably isomorphic if and only if their quotients are isomorphic. 

We now investigate equivalence classes of presentations of a finitely generated abelian group $H$.
\begin{definition}
If $H$ is a finitely generated abelian group, its {\it rank} is, equivalently,
\begin{itemize}
\item the minimal number of infinite and finite cyclic direct summands required to construct $H$
\item the minimal rank of a presentation of $H$
\item the number of torsion coefficients of the torsion subgroup $T$ of $H$ plus the rank of $H/T$.
\end{itemize}
A presentation of minimal rank is simply called a minimal presentation.   $\|$
\end{definition}
 
\begin{lemma} 
\label{L:standardizing non-minimal presentations-1}
Every non-minimal presentation $\F\xrightarrow{\pi} H$ is equivalent to a presentation of the form 
$\F' \oplus \Z \xrightarrow{\pi' \oplus 0} H$, where $\pi'$ is a presentation of $H$.  Every presentation of 
$H$ is a stabilization of a minimal one.
\end{lemma}
 
\begin{proof}
Clearly, by stabilizing, $H$ has a presentation of every rank $\geqslant$ rank $H$.  Let $\F' \xrightarrow{j} H$ be a 
presentation of rank $(\rk \F - 1)$.  By Proposition \ref{P:structure of abelian groups}, the stabilization of $j$
is isomorphic to $\F\xrightarrow{\pi} H$, say by a diagram of the form
$$
\begin{CD}
\F @>\pi>> H \\
@VfVV @VVhV \\
\F' \oplus \Z @>{j \oplus 0}>> H
\end{CD}
$$
Hence
$$
\begin{diagram}
\node{\F} \arrow[2]{s,l}{f} \arrow{see,t}{\pi} \\
\node[3]{H} \\
\node{\F'Ê\oplus \Z} \arrow{nee,b}{h^{-1} \circ (j \oplus 0)}
\end{diagram}
$$
is an equivalence, as desired.  By induction, we conclude that every presentation of $H$ is a stabilization of a minimal one.
\end{proof}
 
In the next few lemmas, we will show that all presentations of $H$ are stably equivalent and that the index of stabilization 
required is at most one.
 
\begin{lemma} 
\label {L:standardizing non-minimal presentations-2}
Let $n$ be the rank of  $\rk(H/T)$ and let $T = \tor(H)$. Then any presentation of $H$ is equivalent to one of the form $\F \oplus \Z^n \xrightarrow{\pi} H$, where $\pi |_{F}$ is a presentation of $T$  and $\pi |_{\Z^n}$ is injective.
\end{lemma}
 
\begin{proof}
A presentation of the required type certainly exists and may be of any rank $\geqslant \rk H$.  The proof that any presentation
is equivalent to one of this form is similar to the proof of Lemma~\ref {L:standardizing non-minimal presentations-1}.
\end{proof}
 
\begin{lemma} \label{L:standardizing non-minimal presentations-3}
Let $h$ be an automorphism of $H$ which acts trivially on $T$.  Then for any presentation $\pi: \F \to H$ there is an automorphism 
$f$ of $\F$ so that 
$$
\begin{CD}
\F @>{\pi}>> H \\ 
@VfVV @VVhV \\
\F @>{\pi}>> H
\end{CD}
$$
commutes.
\end{lemma}
\begin{proof}
By Lemma~\ref{L:standardizing non-minimal presentations-2}, our presentation is equivalent to the direct sum of presentations 
$\F_0 \to T$ and $\Z^n \xrightarrow{id} \Z^n$, where $H$ has been decomposed as $T \oplus \Z^n$.  Representing elements of $H$ by
column vectors $\vectortwoone{t}{z}$ with $t \in T$, $z \in \Z^n$, any automorphism $h$ of $H$ must be of the form
$$\vectortwoone{t}{z} \mapsto \mattwotwo{A}{B}{0}{C} \vectortwoone{t}{z},$$
where $A : T \to T$ and $C : \Z^n \to \Z^n$ are automorphisms and $B : \Z^n \to T$ is a homomorphism; by hypothesis, $A=1$.  If 
we lift $B$ to a homomorphism $\bar{B} : \Z^n \to \F_0$, then the endomorphism $\mattwotwo{1}{\bar{B}}{0}{C}$ of $\F_0 \oplus \Z^n$ 
is an automorphism which clearly induces $\mattwotwo{1}{B}{0}{C}$ on $T\oplus \Z^n = H$, as desired.
\end{proof}
  
\begin{definition}
In the situation of Lemma \ref{L:standardizing non-minimal presentations-3}, we say that $f$ {\it lifts} $h$. $\|$
\end{definition}
 
\begin{corollary} \label{C:lifting-1}
Using the notation of the proof of Lemma \ref{L:standardizing non-minimal presentations-3}, if $h$ is any 
automorphism of $H$ so that $h|_T$ lifts to $\F_0$, then $h$ lifts to $\F$.
\end{corollary}
\begin{proof}
Let $g = (h |_T) \oplus 1_{\Z^n}$.  Clearly $g$ lifts to $\F_0 \oplus \Z^n$.  Hence $h \circ g^{-1} = 1$ on $T$, so
$h g^{-1}$ also lifts.  We conclude $(h g^{-1}) \circ g$ lifts, as desired.
\end{proof}

\begin{definition} Let $f : \F \to \F$ be an endomorphism. Since $\F$ is free abelian, we may represent $f$ by a matrix with respect to any basis for $\F$. We define the {\it determinant} of $\F$ to be the determinant of any such matrix. Clearly $\det f$ is well-defined, up to sign, independent of the choice of basis.  $\|$
\end{definition}

\begin{lemma} \label{L:lifting-1}
Let $T$ be an abelian $p$-group for some prime $p$.  Let $\pi : \F \to T$ be a presentation and let $h : T \to T$ be an automorphism. 
Then there is an endomorphism (which we are not claiming is an automorphism) $f : \F \to \F$ lifting $h$ so that $p$ does not 
divide $\det f$.
\end{lemma}
\begin{proof}
Since $F$ is a {\em free} abelian group, it is easy to construct an endomorphism $f$ of $F$ that lifts $h$, so the key point 
is to construct one so that $p$ does not divide $\det (f)$.   

By Lemma \ref{L:standardizing non-minimal presentations-1}, the presentation $\pi$ is equivalent to
$\pi_0 \oplus 0^k : \F_0 \oplus \Z^k \to T$ with $\pi_0$ minimal (here possibly $k=0$).  Choose $f_0 \in \End \F_0$
so that $\pi_0 f_0 = h \pi_0$, and let $f = f_0 \oplus \id_{\mathbb Z^k}$.  By construction $f$ lifts
$h$, and we claim that $p$ does not divide $\det f$.

Consider the canonical map $T \to T/pT$ and the composite $\pi_p : \F_0 \to T \to T/pT$.  Since $T$ is an abelian $p$-group and 
since $\pi_0$ is minimal, we have $\rk T/pT = \rk T = \rk \F_0$.  Hence $\ker \pi_p = p \F_0$.  Now $h$ induces an automorphism 
$h_p$ of $T/pT$ and we have
$$
\begin{CD}
\F_0 / p \F_0 @>{\cong}>> T/pT \\
@V{f_p}VV                 @V{h_p}VV \\
\F_0 / p \F_0 @>{\cong}>> T/pT
\end{CD}
$$
Hence $f_p$ must be an isomorphism.  This implies that $p$ does not divide $\det f = \det f_0$, for $p$ divides $\det f_0$ 
if and only if the induced map on $\F_0 / p \F_0$ is \emph{not} an isomorphism.
\end{proof}

\begin{lemma} \label{L:lifting-2}
Let $T$ be any finite abelian group, $\pi : \F \to T$ be a presentation of $T$, and $h : T \to T$ be any automorphism. 
Then there is an endomorphism (which we are again not claiming is an automorphism) $f : \F \to \F$ lifting $h$ so 
that $(\det f, | T |) = 1$.
\end{lemma}
\begin{proof}
Let $p$ be a prime divisor of $|T|$.  Suppose that $p^k$ is the highest power of $p$ which divides $|T|$.  Then $T_p := T/p^k T$ 
is an abelian $p$-group isomorphic to the $p$-component of $T$, and $h$ induces an automorphism $h_p$ of $T_p$.  By the previous lemma, 
there is an endomorphism $f_p$ of $\F$ with $p \not | \det f_p$ so that 
$$
\begin{CD}
\F @>{\pi_p}>> T_p \\
@V{f_p}VV @VV{h_p}V \\
\F @>{\pi_p}>> T_p
\end{CD}
$$
commutes. Fixing a basis of $\F$ and representing $f_p$ as a matrix, we note that any matrix congruent to $f_p \ \mod p^k$ also 
induces $h_p$ on $T_p$.  By the Chinese remainder theorem, there is a single matrix $f$ so that $f \equiv f_p \ \mod p^k$ for 
all primes $p$ which divide $|T|$, that is to say, a single endomorphism $f$ of $\F$ inducing $h_p$ on $T_p$ for each such prime.  Since 
$\det f \equiv \det f_p \ \mod p^k$, we have $p \not | \det f$ for all such $p$, \ie $(\det f, |T|) = 1$. 
 
It remains to prove that $f$ induces $h$ on $T$.  Since $T_p = T/p^k T$ is a $p$-group isomorphic to the $p$-component of $T$, it 
follows that the kernel of $\pi : \F \to T$ is precisely
$$ \ker\pi =  \bigcap_{p \text{ divides } |T|} \ker (\pi_p : \F \to T_p).$$
Hence
$$
f(\ker \pi) = f \left( \bigcap_p \ker \pi_pÊ\right) \subset \bigcap_p f(\ker \pi_p) = \bigcap_p \ker \pi_p = \ker \pi.
$$
This shows that $f$ induces {\it some} automorphism of $T$.  But this automorphism induces $h_p$ on $T_p$ for every $p$, so it must be $h$.
\end{proof}
 
\begin{lemma} \label{L:lifting-3}
Let $T$ be a torsion group and let $h \in \Aut (T)$.  Then for any non-minimal presentation $\pi : \F \to T$ 
there is some $f \in \Aut(\F)$ which lifts $h$.
\end{lemma}
\begin{proof}  
Since $\pi$ is non-minimal, by Lemma~\ref{L:standardizing non-minimal presentations-1} we may assume that 
$\pi = \pi' \oplus 0 : \F' \oplus \Z \to T$.  By Lemma \ref{L:lifting-2}, we may lift $h$ to $f \in \End \F'$ with $(\det f, m) = 1$, 
where $m = |T|$. Choose an integer $\delta$ such that $\delta \cdot \det f \equiv 1 \ \mod m$.  Let $d$ denote the 
endomorphism of $\Z$ defined by $d(1) = \delta$.  Then $f \oplus \delta \in \End (\F' \oplus \Z)$ and 
$\det (f \oplus \delta) \equiv 1 \ \mod m$.  Since the canonical homomorphism $\SL(r,\Z) \to \SL(r, \Z_m)$ is surjective, we may lift 
$f \oplus \delta$ to $f' \in \Aut (\F' \oplus \Z)$.  More precisely, we can find some $f' \in {\rm Aut}(F' \oplus \Z)$ which
(when considered as a matrix over $\Z$) is equal to $f \oplus \delta \ \mod m$.  Since $f \oplus \delta$ lifts $h$, the diagram
$$
\begin{CD}
\F' \oplus \Z @>{\pi \oplus 0}>> T \\
@V{f \oplus \delta}VV @VV{h}V \\
\F' \oplus \Z @>{\pi \oplus 0}>> T
\end{CD}
$$
commutes.  Now $m (\F' \oplus \Z) \subset \ker (\pi \oplus 0)$, since $m x = 0$ for all $x \in T$.  Note that by construction, 
$f' \equiv f \oplus \delta \ \mod m$, \ie for each $x \in \F' \oplus \Z$ there is a 
$y \in \F' \oplus \Z$ such that $f'(x) = (f \oplus \delta)(x) + m y$. Hence $f'$ also lifts $h$.
\end{proof}
 
\begin{theorem} \label{T:h lifts to the presentation level}
If $h$ is any automorphism of $H$ and if $\pi : \F \to H$ is any non-minimal presentation of $H$, then $h$ lifts to $F$.
\end{theorem}
\begin{proof}
By Lemma~\ref{L:standardizing non-minimal presentations-2}, we may decompose $\F$ as $\F = \F_0 \oplus \Z^n$, where 
$\pi |_{\F_0} : \F_0 \to T$ is a presentation and $\pi |_{\Z^n} = 1 : \Z^n \to \Z^n$.  The non-minimality of $\F$ 
then implies the non-minimality of $\F_0$.  Hence by Lemma \ref{L:lifting-3}, we may lift $h |_T$ to $\F_0$.  Corollary~\ref{C:lifting-1}
then implies that $h$ lifts to $\F$, as desired.
\end{proof}
 
The following three results will be important later. They are immediate consequences of Theorem \ref{T:h lifts to the presentation level}.
 
\begin{corollary} \label{C:cor 1 of T:h lifts to the presentation level}
If $h : H \to H'$ is an isomorphism and $\pi : \F \to H$, $\pi' : \F' \to H'$ are non-minimal of equal rank, then $h$ lifts to a presentation isomorphism.
\end{corollary}
 
\begin{corollary} \label{C:cor 2 of T:h lifts to the presentation level}
All presentations of $H$ are stably equivalent, and any of two presentations of non-minimal, equal rank are equivalent.
\end{corollary}
 
\begin{corollary} \label{C:cor 3 of T:h lifts to the presentation level}
If $\pi : \F \to H$, $\pi' : \F \to H$ are two minimal presentations, then $\pi, \pi'$ have stabilization index 0 or 1.
\end{corollary}
 
\begin{example} \label{Ex:inequivalent minimal presentations-2} 
To illustrate Corollary~\ref{C:cor 3 of T:h lifts to the presentation level}, recall Example~\ref{Ex:inequivalent minimal presentations-1}.
Two rank 1 presentations $\pi, \pi'$ of $H = \Z_5$ were defined by $\pi (1) = 1$ and $\pi' (1) = 2$.  These are obviously inequivalent. 
We claim that they have equivalent index 1 stabilizations $\pi \oplus 0, \pi' \oplus 0$, \ie there exists some $f$
so that the following diagram commutes:
  $$
 \begin{diagram}
 \node{\Z \oplus \Z} \arrow[2]{s,l}{f} \arrow{see,t}{\pi \oplus 0} \\
 \node[3]{H.} \\
 \node{\Z \oplus \Z} \arrow{nee,b}{\pi' \oplus 0}
 \end{diagram}
 $$
For example, we may define $f$ by 
$$\vectortwoone{z_1}{z_2} \mapsto \mattwotwo{3}{5}{1}{2} \vectortwoone{z_1}{z_2 } = \vectortwoone{3z_1+5z_2}{z_1+2z_2 }$$ 
We then have
$$(\pi'\oplus 0) \circ f(z_1,z_2) = 6z_1 + 10z_2 \equiv z_1\ (\mod 5) = (\pi\oplus 0)(z_1,z_2).\ \ \|$$
\end{example}

\subsection{Equivalence classes of minimal presentations}
\label{SS:equivalence classes of minimal presentations}
We continue our study of presentations of finitely generated abelian groups by investigating equivalence classes of 
{\it minimal} presentations of finitely generated abelian groups. The main results are 
Theorem~\ref{T:equivalence of minimal presentations} and Corollary~\ref{C:complete invariant of equivalence of minimal presentations}, 
which give a complete invariant of equivalence of minimal presentations of $H$.

First we recall the definition of the exterior powers of an abelian group $H$.  From the $k^{\Th}$ tensor power 
$H^k = H\otimes \cdots \otimes H$ we form a quotient by dividing out by the subgroup generated by all 
$x_1 \otimes \cdots \otimes x_k$ in which two $x_i$'s are equal. This quotient is the {\it $k^{\Th}$ exterior power of $H$}, 
denoted by $\Lambda^k H$.  The image of an arbitrary tensor product $x_1 \otimes \cdots \otimes x_k$ in $\Lambda^k H$ is denoted 
by $x_1 \wedge \cdots \wedge x_k$, and we have the usual law 
$$x_1 \wedge \cdots x_i \wedge x_{i+1} \cdots \wedge x_k = - x_1 \wedge \cdots x_{i+1} \wedge x_i \cdots \wedge x_k.$$ 
Also as usual, $x_1 \wedge \cdots \wedge x_k = 0$ if any $x_i$ is a linear combination of the other terms, which implies 
that $\Lambda^k H = 0$ if $k> \rk H$.

\begin{lemma} \label{L:wedge_r-1}
Let $r = \rk H$ and $\tau$ be the smallest elementary divisor of the torsion subgroup $T$ of $H$.  If $T=0$, we put 
$\tau = 0$.  Then $\Lambda^r H$ is cyclic of order $\tau$ if $T \not= 0$, whereas if $T = 0$, then $\Lambda^r H$ is infinite cyclic. 
If $x_1, \dots, x_r$ generate $H$, then $x_1 \wedge \cdots \wedge x_r$ generates $\Lambda^r H$.
\end{lemma}
\begin{proof}
The case when $T=0$ is well known, so we assume that $\tau > 0$.  We prove the last statement first.  Now, $\Lambda^r H$ is 
generated by all $y_1 \wedge \cdots \wedge y_r$ as the $y_i$ range over $H$.  But let $y_i = \sum_{j=1}^r \alpha_{ij} x_j$ 
for integers $\alpha_{ij}$.  A straightforward check shows that
$y_1 \wedge \cdots \wedge y_r = \det (\alpha_{ij}) (x_1 \wedge \cdots \wedge x_r)$. 

Thus $\Lambda^r H$ is cyclic. Elementary divisor theory tells us that $H$ is the direct sum of $r$ cyclic groups $\Z_{\tau_i}$, 
where the $\tau_i$ are the elementary divisors ($\Z_0$ means $\Z$ here).  If these cyclic summands have generators 
$x_i$, with $x_1$, say, of order $\tau = \tau_1$, then $\theta = x_1 \wedge \cdots \wedge x_r$ generates $\Lambda^r H$, and 
$\tau \theta = (\tau x_1) \wedge x_2 \wedge \cdots \wedge x_r = 0$ since $\tau x_1 = 0$.  Thus $| \Lambda^r H | \leqslant \tau$. 

We must therefore prove that $| \Lambda^r H| \geqslant \tau$.  Consider the map $d : H^r \to \Z_\tau$ defined by 
$$d(y_1\otimes \cdots \otimes y_r) = \det (\alpha_{ij}) \ \mod \tau,$$ 
where $y_i = \sum_j \alpha_{ij} x_j$.  This is well defined, since if for some $i$ we had $\sum_j \alpha_{ij} x_j = 0$ in $H$, 
then we must have $\alpha_{ij} \equiv 0 \ \mod \tau$ for all $j$, and hence $\det( \alpha_{ij}) \equiv 0 \ \mod \tau$.  
The map $d$ is also clearly onto (let $y_i = x_i$).  Finally, $d$ kills all terms having two $y_i$'s equal, 
so it induces a map of $\Lambda^r H$ onto $\Z_\tau$.
\end{proof}

\begin{definition}\label{D:orientation and volume of H}
If $H$ has rank $r$, an {\it orientation} of $H$ is a selection of a generator $\theta$ of $\Lambda^r H$. A {\it volume} of $H$ is 
a pair $\pm \theta$ of orientations of $H$. \ie an orientation of $H$, determined up to sign.  Observe that a free abelian 
group of rank $r$ has $\Lambda^r \simeq \Z$ and hence two orientations and only one volume, but if $H$ has torsion, 
then it will in general have many volumes. $\|$
\end{definition}

If $f: H \to H'$ is a homomorphism between groups of the same rank $r$, then $f$ induces a homomorphism $\Lambda^r f : \Lambda^r H \to  \Lambda^r H'$ in the standard way; we will write simply $f$ for $ \Lambda^r f$.

\begin{lemma} \label{L:orientation-1}
Assume that $H$ and $H'$ both have rank $r$ and that $f : H \to H'$ is surjective.  Let $\tau$ and $\tau'$ be the smallest
elementary divisors of $H$ and $H'$, respectively.  Then $\tau' \mid \tau$ and if $\theta$ is any orientation of $H$, then
$f(\theta)$ is an orientation of $H'$.
\end{lemma}

\begin{proof}
Let $x_1, \dots, x_r$ generate $H$, so $\varphi = x_1 \wedge \cdots \wedge x_r$ generates $\Lambda^r H$.  There is thus some generator
$m$ of $\Z_\tau$ so that $\theta = m \varphi$.  But since $f$ is onto, $H'$ is generated by $f(x_1),\ldots,f(x_r)$ and 
$\varphi' = f(x_1) \wedge \cdots \wedge f(x_r) = f(\varphi)$ generates $\Lambda^r H'$.  This shows that 
$f : \Lambda^r H \to \Lambda^r H'$ is also surjective and hence that $\tau' \mid \tau$.  We conclude that $m$ is also 
a generator of $\Z_{\tau'}$, and hence that $\theta' = f(\theta) = m f (\varphi) = m \varphi'$ generates $\Lambda^r H'$.
\end{proof}

If $H, H'$ have specific orientations $\theta, \theta'$ and $f: H \to H'$ is a homomorphism, then since 
$\theta'$ generates $\Lambda^r H'$ we have $f(\theta) = m \theta'$ for a unique $m \in \Z_{\tau'}$.  We call $m$ 
the {\it determinant} of $f$ (with respect to the orientations $\theta, \theta'$) and write $f(\theta) = \det f \cdot \theta'$. 
If $H, H'$ have only {\it volumes} specified, then $\det f$ is determined up to sign.
If, however, $H = H'$ and $\theta = \theta'$, then $\det f$ is independent of $\theta$; in fact, 
$f: \Lambda^r H \to \Lambda^r H'$ is just multiplication by $\det f \in \Z_\tau$.
Thus endomorphisms of $H$ have a well-defined determinant, and it is easy to see that this definition 
is the classical one when $H$ is free.  More generally, we have:

\begin{lemma} \label{L:orientation-2}
Suppose 
 $$
 \begin{CD}
 \F @>\pi>> H \\
 @VfVV @VVhV \\
 \F' @>\pi'>> H'
 \end{CD}
 $$
commutes, where $\pi, \pi'$ are presentations and all groups have the same rank. Then if $\varphi, \varphi'$ are orientations of $\F, \F'$ 
inducing orientations $\theta, \theta'$ of $H, H'$, we have $\det h \equiv \det f \ \mod \tau'$.  If no orientations are specified, 
then we measure $\det h$ with respect to the canonical induced volumes, and the above congruence holds up to sign.
\end{lemma}

\begin{proof}
Observe that $\det f \cdot \varphi' = f(\varphi)$ and that
$$\det h \cdot \theta' = h(\theta) = h \pi(\varphi) = \pi' (\det f\cdot \varphi') \equiv_{\mod \tau'} \det f \cdot \pi' (\varphi') 
= \det f \cdot \theta'.$$ 
The final statement is obvious.
\end{proof}

Note that if $H = H'$, $\F = \F'$ and $\pi = \pi'$, then $\det f$, $\det h$ and the congruence are independent of the orientations.

The following two lemmas show that $\det$ behaves like the classical determinant.

\begin{lemma} \label{L:properties of det(f)-1}
If $f : (H_1,\theta_1) \rightarrow (H_2,\theta_2)$ and $g : (H_2, \theta_2) \rightarrow (H_3, \theta_3)$ 
are homomorphisms of oriented groups so that all the $H_i$ have the same rank, then if $\tau_3$ is the smallest elementary
divisor of $H_3$ we have $\det (g f) \equiv (\det g) \cdot (\det f)\ \mod \tau_3$.
\end{lemma}

\begin{proof}
We calculate:
$$\det(g f) \cdot \theta_3 = g f (Ê\theta_1) = g (\det f \cdot \theta_2) = \det f \cdot g (\theta_2) = \det f \cdot \det g \cdot \theta_3.$$
\end{proof}

\begin{lemma} \label{L:properties of det(f)-2}
Let $\F$ and $G$ be abelian groups with $\F$ free, and let $h$ be an endomorphism of $\F \oplus G$ so that $h(G) < G$.  Let
$g = h|_G$ and let $f$ be the map on $\F = \frac{\F \oplus G}{G}$ induced by $h$.  Then $\det h = \det f \cdot \det g \ \mod \tau$, 
where $\tau$ is the smallest elementary divisor of $G$ (and hence of $\F \oplus G$).  In particular, if $h$ is an automorphism, 
then $\det h = \pm \det g$.
\end{lemma}

\begin{proof}
Let $m = \rk \F$ and $n = \rk G$; then $m+n = \rk \F \oplus G$ holds because $\F$ is free.  Let $x_1, \dots, x_m$ and 
$y_1, \dots, y_n$ be minimal sets of generators of $\F$ and $G$; their union is then a minimal set of generators of $\F \oplus G$. 
By hypothesis, $h(y_i) = g(y_i)$; also, $h(x_i) = f(x_i) + e(x_i)$ is the direct sum decomposition of $h(x_i)$, where 
$e$ is some homomorphism $\F \to G$. Hence 
$$
\begin{array}{l}
 \det h \cdot (x_1 \wedge \cdots \wedge x_m) \wedge (y_1 \wedge \cdots \wedge y_n) \\
 \qquad \qquad = \quad \left(f(x_1) + e(x_1)\right) \wedge \cdots \wedge \left(f(x_m)+e(x_m)\right) \wedge \left(g(y_1) \wedge \cdots \wedge g(y_n)\right) \\
 \qquad \qquad = \quad  \det g \cdot \left(f(x_1) + e(x_1)\right) \wedge \cdots \wedge \left(f(x_m)+e(x_m)\right) \wedge \left(y_1 \wedge \cdots \wedge y_n\right). 
\end{array}
$$
But since $e(x_i)$ is a linear combination of the $y_i$'s, the above reduces to just 
 $$
 \det g \cdot (f(x_1) \wedge \cdots \wedge f(x_m)) \wedge (y_1 \wedge \cdots \wedge y_n) = \det f \cdot \det g \cdot (x_1 \wedge \cdots \wedge y_n)
 $$
as desired. If $h$ is an automorphism, then $\det f$ must be $\pm 1$, proving the last statement.
\end{proof}

Suppose now that $\F \xrightarrow{\pi} H$ is a minimal presentation, so that $\rk \F = \rk H = r$.  Let $\pm \varphi$ be
the unique volume on $\F$, and let $\pm \theta \in \wedge^r H$ be $\pm \pi(\varphi')$; we call the volume $\pm \theta$ the 
{\it volume of} (or {\it induced by}) {\it the presentation $\pi$}.  Suppose now that
$\F \xrightarrow{\pi'} H$ is an equivalent presentation, \ie there exists a diagram
$$
 \begin{diagram}
 \node{\F} \arrow[2]{s,l}{f} \arrow{see,t}{\pi} \\
 \node[3]{H} \\
 \node{\F'} \arrow{nee,b}{\pi'}
 \end{diagram}
 $$
with $f$ an isomorphism.  If $\pm \theta'$ is the volume induced by $\pi'$, we then have
$$\pm \theta = \pi( \pm \varphi) = \pi' f (\pm \varphi) = \pi' (\pm \det f \cdot \varphi') = \pm \det f \cdot \theta'.$$ 
But the fact that $f$ is an isomorphism implies that $\det f = \pm 1$, and hence, {\it equivalent presentations have the same volume}. 
This argument generalizes in the obvious way to prove the necessity in the following:
 
\begin{theorem} \label{T:equivalence of minimal presentations}
Let $\pi: \F \to H$, $\pi':\F' \to H'$ be two minimal presentations with volumes $\theta, \theta'$ and let 
$h : H \to H'$ be an isomorphism. Then $h$ lifts to an isomorphism $f : \F \to \F'$ if and only if 
$h(\pm \theta) =  \pm \theta'$; that is, if and only if $\det h = \pm 1 \ \mod \tau$ where $\tau$ is the
smallest elementary divisor of $H \cong H'$.
 \end{theorem}
 
\begin{proof}
We first claim that it suffices to consider the special case when the two presentations are identical.  Indeed,
by Proposition \ref{P:structure of abelian groups} the two presentations are isomorphic; \ie there exists
a commutative diagram
$$\begin{CD}
\F       @>{\pi}>>  H\\
@V{f'}VV            @VV{h'}V\\
\F'      @>{\pi'}>> H'
\end{CD}$$
with both $f'$ and $h'$ isomorphisms.  Since $f'$ is an isomorphism, the map $h'$ has determinant $1$ and hence so does
$h \circ h'^{-1}$.  If we could lift the automorphism $h \circ h'^{-1}$ to an automorphism $f''$ of $F'$, then
$f:=f'' \circ f':\F \rightarrow F'$ would be the desired lift of $h$.
     
Hence let $h$ be an automorphism of $H$ with $\det h = \pm 1$ and let $\F \xrightarrow{\pi} H$ be any minimal presentation. 
The proof proceeds just as the proof of Theorem \ref{T:h lifts to the presentation level}: if $T$ is the torsion subgroup of $H$ and 
$\F_0 = \pi^{-1} (T)$ and $h_0 = h |_T$, then it still suffices to lift $h_0$ to $\F_0$.  Since $H$ is the direct sum of $T$
and a free abelian group, the presentation $\F_0 \to T$ is also minimal. Furthermore, the conditions of Lemma 
\ref{L:properties of det(f)-2} hold here, so $\det h_0 = \pm 1$ also.  Thus it suffices to prove the theorem when 
$H = T$ is a torsion group.
 
Let $|H| = m$.  By Lemma \ref{L:lifting-1} we may lift $h$ to an endomorphism $f_0$ of $\F$ such that $(\det f_0, m) = 1$; by 
Lemma \ref{L:lifting-3}, it follows that $\det f_0 \equiv \det h \equiv \pm 1 \ \mod \tau$.  Choose $k$ such that 
$k \cdot \det f_0 \equiv \pm 1 \ \mod m$, where the sign here is to be the same as the one above, so that $k \equiv 1 \ \mod \tau$. 
We choose a basis $e_1, \dots, e_r$ of $\F$ (as in Proposition \ref{P:structure of abelian groups}) so that $H$ is the direct sum 
of the cyclic subgroups generated by $x_i = \pi (e_i)$ and $x_1$ has order $\tau = m_1$.  The endomorphism $f_1$ of $\F$ defined by 
$e_1 \mapsto ke_1$, $e_i \mapsto e_i$ for all $i>1$ clearly induces the identity map on $H$, since $k \equiv 1 \ \mod \tau$. 
Hence $f_0 f_1$ still induces $h$ on $H$, and its determinant is now $\det f_0 \cdot \det f_1 \equiv k \det f_0 \equiv \pm 1 \ \mod m$. 
Just as in the proof of Lemma \ref{L:lifting-3}, we conclude that there exists an {\it isomorphism} $f$ of $\F$, with determinant $\pm 1$ 
(same sign!) such that $f \equiv f_0 f_1 \ \mod m$ (we are here using the fact that $\GL (r, \Z)$ maps onto all elements of 
$\GL (r, \Z_m)$ with determinant $\pm 1$).  As in Lemma \ref{L:lifting-3}, $f$ still induces $h$ on $H$, and we are done.
 \end{proof}
 
 Lifting the identity automorphism gives:
 
 \begin{corollary} \label{C:complete invariant of equivalence of minimal presentations}
 Two minimal presentations of $H$ are equivalent if and only if they induce the same volume on $H$.
 \end{corollary}
 
 Here are some examples to show that calculations can actually be done with this machinery.
 
 \begin{example} \label{Ex:inequivalent minimal presentations-3}
In Example~\ref{Ex:inequivalent minimal presentations-1} we gave an example of inequivalent minimal presentations, namely $\Z \xrightarrow{1\mapsto 1} \Z_5$ and $\Z \xrightarrow{1 \mapsto 2} \Z_5$. Now observe that $r=1$, $\tau = 5$, $\Lambda^1 \Z_5 = \Z_5$; $\theta = \pm1$, $\theta' = 2 (\pm 1) = \pm 2 \not\equiv \pm 1 \ \mod 5$. $\|$ 
 \end{example}

 \begin{example} \label{Ex:inequivalent minimal presentations-4}
 Let $H = \Z_{2^{n-1}} \oplus \Z_{2^n}^2$ with standard generators $e_1, e_2, e_3$ and standard presentation $\Z^3 \to H$ taking $(1,0,0) \mapsto e_1$, etc. Let the second presentation be given by
 $$
(1,0,0) \mapsto e_1 + 2 e_2 - 2 e_3 
\quad ; \quad
(0,1,0) \mapsto e_1 + e_2
\quad ; \quad
(0,0,1) \mapsto e_1 - e_3 
 $$
 (it is easily seen that this map is onto $H$). The former volume is $\pm e_1 \wedge e_2 \wedge e_3$, the latter 
 $$
 \pm (e_1 + 2 e_2 - 2 e_3) \wedge (e_1 + e_2) \wedge (e_1 - e_3) = \pm \theta \cdot \det 
\left(\begin{array}{ccr}
 1 & 2 & -2 \\
 1 & 1 & 0 \\
 1 & 0 & -1 
 \end{array}\right)
 = \pm 3 \theta.
 $$
 Since $\tau = 2^{n-1}$ here, the presentations are equivalent if and only if $\pm 3 \equiv \pm 1 \ \mod 2^{n-1}$, \ie if and only if $n \leqslant 3$ (the signs on 3 and 1 are independent). $\|$
 \end{example}

\newpage
  \section{Symplectic spaces, Heegaard pairs and symplectic Heegaard splittings} 
  \label{S:symplectic spaces, Heegaard pairs and symplectic Heegaard splittings}

As we noted at the start of the previous section, when a 3-manifold $W$ is defined by a Heegaard splitting, then we have, in a natural way, a presentation of 
$H_1(W;\Z)$. In fact we have more, because there is also a natural symplectic form associated to the presentation.   
In this section, our goal is to begin to broaden the concept of a presentation by placing additional structure on the free group of the presentation, and then to extend the results of Section~\ref{S:presentation theory for finitely generated abelian groups} to include the symplectic structure.
With that goal in mind we introduce symplectic spaces and their lagrangian subspaces, leading to the concept of a {\it Heegaard pair}.  There are equivalence relations on Heegaard pairs analogous to those on free pairs $(F,R)$. Just as we stabilized free pairs by taking their direct sums with $(\Z^k, \Z^k)$, we will see that there is an analogous concept of stabilization of Heegaard pairs, only now we need direct sums with a {\it standard} Heegaard pair.   At the end of the section (see Theorem~\ref{T:double cosets and Heegaard pairs}) we will relate our Heegaard pairs to the symplectic Heegaard splittings that were introduced in $\S$\ref{S:introduction}.

\subsection{Symplectic spaces and Heegaard pairs}
\label{SS:symplectic spaces and Heegaard pairs}

To begin, we reinterpret the free group $F$ of Definition~\ref{D:free pair}, introducing new notation, ideas and structure in the process.
\begin{definition}\label{D:symplectic space}
A {\it symplectic space} is a finitely generated free abelian group $V$ which is endowed with a non-singular antisymmetric bilinear pairing, written here as a dot product. {\it Non-singular} means that for each homomorphism $\alpha : V \to \Z$ there is an $x_\alpha \in V$ (necessarily unique) such that $\alpha (y) = x_\alpha \cdot y \ (\forall y \in V$). A {\it symplectic} or {\it $\Sp$-basis} for $V$ is a basis $\{Êa_i, b_i ; 1 \leqslant i \leqslant g \}$ such that $a_i \cdot a_j = b_i \cdot b_j = 0, a_i \cdot b_j = \delta_{ij}, 1 \leqslant i,j \leqslant g$. Every symplectic space has such a basis, and so is of even rank, say $2g$. As our {\it standard model} of a rank $2g$ symplectic space we have $X_g = \Z^{2g}$ with basis $\{a_1, \dots, a_g, b_1, \dots, b_g\}$ the $2g$ unit vectors, given in order. An isomorphism $V \to V$ which is form-preserving is a {\it symplectic isomorphism}. The group of all symplectic isomorphisms of $V$ is denoted $\Sp(V)$.   $\|$
\end{definition}

\begin{definition}\label{D:subspace of symplectic space}
Let $B \subset V$ be a subset of a symplectic space $V$ and define $B^\perp = \{Êv \in V \ | \ v \cdot b = 0 \  (\forall b \in B)\}$. 
\begin{itemize}
\item A subspace $B \subset V$ is {\it symplectic} if, equivalently,
\begin{enumerate}[a)]
\item the symplectic form restricted to $B$ is non-singular, or
\item $V = B \oplus B^\perp$.
\end{enumerate}
\item A subspace $B \subset V$ is {\it isotropic} if, equivalently,
\begin{enumerate}[a)]
\item $x \cdot y = 0$ for all $x, y \in B$, or
\item $B \subset B^\perp$.
\end{enumerate}
\item A subspace $B \subset V$ is {\it lagrangian} if, equivalently,
\begin{enumerate}[a)]
\item $B$ is maximal isotropic, or
\item $B = B^\perp$, or
\item $B$ is isotropic, a direct summand of $V$, and $\rk B = \frac{1}{2} \rk V$.
\end{enumerate}
\end{itemize}
We shall omit the proof that these various conditions are indeed equivalent.   $\|$
\end{definition}

Our next definition is motivated by the material in $\S$\ref{SS:Heegaard splittings of 3-manifolds}, where we defined symplectic Heegaard splittings.  We will see very soon that our current definitions lead to the identical concept.
\begin{definition}
A {\it Heegaard pair} is a triplet $(V ; B, \Bbar)$ consisting of a symplectic space $V$ and an ordered pair $B, \Bbar$ of lagrangian subspaces. The {\it genus} of the pair is $\rk B = \rk \Bbar = \frac{1}{2} \rk V$. An {\it isomorphism} of Heegaard pairs $(V_i; B_i, \Bbar_i), i = 1,2$ is a symplectic isomorphism $f : V_1 \to V_2$ such that $f(B_1) = B_2, f(\Bbar_1) = \Bbar_2$.  $\|$
\end{definition}

We now want to define a concept of ``stabilization" for Heegaard pairs. If $V$ has $\Sp$-basis $\{a_i, b_i \ | \ i = 1, \dots, g \}$, then the $a_i$'s (and also the $b_i$'s) generate a lagrangian subspace. These two subspaces $A, B$ have the following properties:
\begin{enumerate}[(a)]
\item $AÊ\oplus B = V$
\item the symplectic form induces a dual pairing of $A$ and $B$, \ie $a_i \cdot a_j = b_i \cdot b_j = 0$, $a_i \cdot b_j = \delta_{ij}$, $1 \leqslant i, j \leqslant g$.
\end{enumerate}
Any pair of lagrangian subspaces of $V$ satisfying these two properties with respect to some basis 
will be called a {\it dual pair}, and either space will be called the {\it dual complement} of the other.

If $X_g$ is the standard model for a symplectic space, then the lagrangian subspaces $E_g$ spanned by $a_1, \dots, a_g$ and $F_g$ spanned by $f_1, \dots, f_g$ are a dual pair. We will refer to $(X_g ; E_g, F_g)$ as the {\it standard Heegaard pair}. Note that in an arbitrary Heegaard pair $(V ; B, \Bbar)$ the lagrangian subspaces $B, \Bbar$ need not be dual complements.

If $V_1$ and $V_2$ are symplectic spaces, then $V_1 \oplus V_2$ has an obvious symplectic structure, and $V_1$ and $V_2$ are 
$\Sp$-subspaces of $V_1 \oplus V_2$ with $V_1 = V_2^\perp$ and $V_2 = V_1^\perp$. This induces a natural direct sum construction for Heegaard pairs, with $(V_1 ; B_1, \Bbar_1) \oplus (V_2 ; B_2, \Bbar_2) = (V_1 \oplus V_2 ; B_1 \oplus B_2, \Bbar_1 \oplus \Bbar_2)$. The {\it stabilization of index $k$} of a Heegaard pair is its direct sum with the standard Heegaard pair $(X_k ; E_k, F_k)$ of genus $k$. Two Heegaard pairs $(V_i; B_i, \Bbar_i),$ $i = 1,2$, are then {\it stably isomorphic} if they have isomorphic stabilizations.

These concepts will soon be related to topological ideas. First, however, we will show that stable isomorphism classes and isomorphism classes of Heegaard pairs are in 1-1 correspondence with stable double cosets and double cosets in the symplectic modular group $\Gamma$, with respect to its subgroup $\Lambda$.

Note that if $A, B$ is a dual pair of $V$ and $\cU : B \to B$ is a linear automorphism, then the adjoint map $(\cU ^*)^{-1}$ is an isomorphism of $A = B^*$. Moreover $(\cU^*)^{-1} \oplus \cU$ is a symplectic automorphism of $V = A \oplus B$.

\begin{lemma} \label{lemma:3.1}
If $B \subset V$ is lagrangian, $A \subset V$ is isotropic and $A \oplus B = V$, then $A$ is lagrangian and $A, B$ is a dual pair of $V$. Every lagrangian subspace has a dual complement.
\end{lemma}

\begin{proof}
Since $A$ is an isotropic direct summand of $V$ and $\rk A = \rk V - \rk B = \frac{1}{2} \rk V$, it follows that $A$ is lagrangian. 
Let now $f : B \to \Z$ be linear, and extend it to $\alpha : V \to \Z$ by setting $\alpha (A) = 0$.  Then $\alpha (v) = x \cdot v$ 
for some $x \in V$.  Since $x \cdot A = 0$ and $A$ is maximal isotropic, we must have $x \in A$, showing that $A, B$ are dually paired 
and hence a dual pair of $V$.  To prove the second statement, let $b_i$ be a basis of $B$.  Since $B$ is lagrangian it 
is a direct summand of $V$, so we may choose a homomorphism $f_1 : V \to \Z$ such that $f_1(b_1) = 1$ and $f_1(b_i) = 0$ for $i>1$. 
Let $a_1 \in V$ be such that $a_1 \cdot v = f_1 (v)$ for all $v$.  Clearly the subgroup generated by $a_1$ and $B$ is still a 
direct summand of $V$, so choose $f_2 : V \to \Z$ such that $f_2 (b_2) = 1$, $f_2(a_1) = f_2(b_i) = 0 \ (i \not= 2)$ and 
$a_2 \in V$ realizing this map.  Continuing in this way, we get finally $a_1, \dots, a_g$ such that $a_i, b_i$ satisfy the laws of a 
symplectic basis and generate a direct summand of $V$.  This direct summand has the same rank as $V$, so it equals $V$, and the 
group $A$ generated by the $a_i$'s is then a dual complement of $B$.
\end{proof}

\begin{proposition} \label{proposition:3.2}
Let $A, B$ be a dual splitting and $b_i$ a basis of $B$. If $a_i$ is the dual basis of $A$ defined by $a_i \cdot b_j = \delta_{ij}$, 
then $a_i, b_i$ is a symplectic basis of $V$.
\end{proposition}

\begin{corollary} \label{corollary:3.3}
If $A, B$ and $A', B'$ are two dual pairs of $V$, then there is an $f \in \Sp(V)$ such that $f(A) = A'$, $f(B) = B'$.
\end{corollary}

\begin{proof}
Choose symplectic bases $a_i, b_i$ adapted to $A, B$ and $a_i', b_i'$ adapted to $A', B'$.  
Then the map defined by $a_i \mapsto a_i'$ and $b_i \mapsto b_i'$ is symplectic.
\end{proof}

\begin{corollary} \label{corollary:3.4}
If $B, B'$ are lagrangian, there is an $f \in \Sp(V)$ such that $f(B) = B'$.
\end{corollary}

\begin{proof}
By Lemma \ref{lemma:3.1}, $B$ and $B'$ have dual complements $A$ and $A'$. By Corollary \ref{corollary:3.3} we may find $f \in \Sp(V)$ such that $f(B) = B'$.
\end{proof}

\subsection{Heegaard pairs and symplectic Heegaard splittings}
\label{SS:Heegaard pairs and symplectic Heegaard splittings}

We are now ready to relate our work on Heegaard pairs to the double cosets introduced in $\S$\ref{S:introduction}.  We follow notation used there.    

\begin{theorem} \label{T:double cosets and Heegaard pairs}   The following hold:
\begin{enumerate}
\item [{\rm (1)}] Isomorphism classes of Heegaard pairs are in 1-1 correspondence with double cosets in $\Gamma \ \mod \Lambda$. Stable isomorphism classes of Heegaard pairs are in 1-1 correspondence with stable double cosets in $\Gamma \ \mod \Lambda$.
\item  [{\rm (2)}] Let $j:\partial N_g\to N_g, \ \bar{j}: \partial N_g\to \bar{N_g}$, and let $j _*, \bar{j}_*$ be the induced actions on homology. \\ Then the triplet $(H_1 (M; \Z) ; \ker j_*, \ker \bar{j}_*)$ is a Heegaard pair.
\item [{\rm (3)}] Every Heegaard pair is topologically induced as the Heegaard pair associated to a topological Heegaard splitting of some 3-manifold. Moreover, equivalence classes and stable equivalence classes of Heegaard pairs are topologically induced by equivalence classes and stable equivalence classes of Heegaard splittings.  
\end{enumerate}
\end{theorem}

\begin{proof}  We begin with assertion (1). 
Let $(V ; B, \Bbar)$ be a Heegaard pair of genus $g$. Then by Corollary \ref{corollary:3.3} we may find a symplectic isomorphism $f:V \to X_g$ such that $f(B) = F_g$. Putting $\Fbar_g = f( \Bbar)$, we then have $(V ; B, \Bbar)$ isomorphic to $(X_g ; F_g, \Fbar_g)$. If $f' : (V, B) \to (X_g, F_g)$ is another choice, with $f'(\Bbar) = \Fbar_g'$, then $f' f^{-1} (F_g) = F_g$, hence $f' f^{-1} \in \Lambda$. Then we see that the isomorphism classes of genus $g$ Heegaard pairs correspond to equivalence classes of lagrangian subspaces $\Fbar_g \subset X_g$, with $\Fbar_g, \Fbar_g'$ equivalent if there is a map $m \in \Lambda$ such that $m(\Fbar_g) = \Fbar_g'$. Now, we have seen that there is a map $h \in \Gamma$ such that $\Fbar_g = h(F_g)$, and $h_1 (F_g) = h_2(F_g)$ if and only if $h_2 = h_1 f$ for some $f \in \Lambda$. Then each $\Fbar_g$ can be represented by an element $h \in \Gamma$ and $h, h'$ give equivalent subspaces $\Fbar_g = h(F_g)$, $\Fbar_g' = h' (F_g)$ if and only if there are $f_1, f_2 \in \Lambda$ such that $h' = f_1 h f_2$. The set of all $f_1 h f_2$, $f_i \in \Lambda$, is a double coset of $\Gamma \ \mod \Lambda$. Then the isomorphism classes of Heegaard pairs of genus $g$ are in 1-1 correspondence with the double cosets of $\Gamma \ \mod \Lambda$.

Direct sums and stabilizations of Heegaard pairs corresponds to a topological construction. If $(W_i ; N_i, \Nbar_i)$ ($i=1,2$) are 
Heegaard splittings, their connected sum $(W_1 \# W_2 ; N_1 \# N_2, \Nbar_1 \# \Nbar_2)$ is a Heegaard splitting whose abelianization to a 
Heegaard pair is readily identifiable as the direct sum of the Heegaard pairs associated to the summands.  Moreover, if $(\s{3} ; Y_k, \Ybar_k)$ is a standard Heegaard splitting of genus $k$ for $\s{3}$, its Heegaard pair may be identified with the standard 
Heegaard pair $(X_k ; E_k, F_k)$ of index $k$.  This stabilization of Heegaard pairs is induced by the topological construction $(W ; N, \Nbar) \to (W \# \s{3} ; N \# Y_k, \Nbar \# \Ybar_k)$. 

In an entirely analogous manner to the proof just given for (1), stable isomorphism classes of 
Heegaard pairs correspond to stable double cosets in $\Gamma \ \mod \Lambda$.

Proof of (2): There is a natural symplectic structure on the free abelian group $H_1 (M; \Z)$, with the bilinear pairing defined by intersection numbers of closed curves which represent elements of $H_1 (M; \Z)$ on $M$. 
In fact, as claimed in (2) above, the triplet $(H_1 (M; \Z) ; \ker j_*, \ker \bar{j}_*)$ is a Heegaard pair.
To see this, let $B = \ker j_*$. Since $x \cdot y = 0$ $\forall x, y \in B$, the subspace $B$ is isotropic. Also, $\rk B = \frac{1}{2} \rk H_1 (M ; \Z) = \genus N = \genus M$. Hence $B$ is lagrangian. Similarly, $\Bbar$ is lagrangian.  Therefore the assertion is true.

Proof of (3):  It remains to show that every Heegaard pair is topologically induced as the Heegaard pair associated to a topological 
Heegaard splitting of some 3-manifold, and also that  equivalence classes and stable equivalence classes of Heegaard pairs are 
topologically induced by equivalence classes and stable equivalence classes of Heegaard splittings.  To see this, let $(V; B, \Bbar)$ be a Heegaard pair. By Theorem \ref{T:double cosets and Heegaard pairs} we may without loss of generality assume that $(V; B, \Bbar)$ is $(X_g ; F_g, \Fbar_g)$. Choose a standard basis for $H_1 (M; \Z)$, with representative curves as illustrated in Figure 1.   We may without loss of generality take one of these (say $w_1, \dots, w_g $) to be standard and cut $M$ open along $w_1, \dots, w_g$ to a sphere with $2g$ boundary components $w_i, \bar{w}_i \ (i = 1, \dots, g)$. Choose $2g$ additional curves $V_1, \dots, V_g, W_1, \dots, W_g$ on $M$ such that each pair $w_i, \ W_i$ is a canceling pair of handles, \ie $w_i \cdot W_i = 1$ point, $w_i \cdot W_j = W_i \cdot W_j = \emptyset$ if $i \not= j$, and similarly for the $V_i$'s. Then the matrices of algebraic intersection numbers 
$$
\big\| | v_i \cdot w_j | \big\| \ , \quad \big\| | v_i \cdot W_j | \big\|
$$
uniquely determine a symplectic Heegaard splitting. 
This gives a natural symplectic isomorphism from $H_1(M; \Z)$ to $X_g$.  Also, since $M$ is pictured in Figure 1 as the boundary of a handlebody $N$, our map sends $H_1 (N;\Z)$ to $F_g$.   By Corollary \ref{corollary:3.4} we may find $h \in \Gamma$ such that $h(F_g) = \Fbar_g$. By \cite{Burkhardt1890} each $h \in \Gamma$ is topologically induced by a homeomorphism $\tilde{h} : M \to M$. Let $\Nbar$ be a copy of $N$, and let $W$ be the disjoint union of $N$ and $\Nbar$, identified along $\partial N = M$ and $\partial \Nbar = M$ by the map $\tilde{h}$. Then $(W ; N, \Nbar)$ is a Heegaard splitting of $W$ which induces the Heegaard pair $(V; B, \Bbar)$.

In an entirely analogous manner, the correspondence between (stable) isomorphism classes of Heegaard pairs and symplectic Heegaard splittings may be established, using the method of proof of Theorem \ref{T:double cosets and Heegaard pairs} and the essential fact that each $h \in \Lambda$ is topologically induced by a homeomorphism $\tilde{h} : M \to M$.  
\end{proof}

\newpage
  \section{Heegaard pairs and their linked abelian groups} 
  \label{S:Heegaard pairs and their linked abelian groups} 

In this section we meet linked groups for the first time in our investigations of Heegaard pairs. 
We show that the problem of classifying stable isomorphism classes of Heegaard pairs reduces to the problem of classifying linked abelian groups. 
This is accomplished in Theorem \ref{T:linking isomorphism lifts to Heegaard isomorphism} and Corollary \ref{C:stably isomorphic iff linked quotients}.  In Theorem \ref{T:a linked group of rank r has a Heegaard presentation of genus r} and Corollary \ref{C:stabilize} we consider the question: how many stabilizations are needed to obtain equivalence of minimal, stably equivalent Heegaard pairs? Corollary \ref{C:problem on stabilization index} asserts that a single stabilization suffices, generalizing the results of Theorem \ref{T:h lifts to the presentation level} and Corollary \ref{C:cor 3 of T:h lifts to the presentation level}.  This solves Problem 3. 

The final part of the section contains partial results about classifying Heegaard pairs of minimal rank.  Theorem \ref{T:the volume again} 
is a first step.  The complete solution to that problem will be given, later, in 
Theorem~\ref{T:unstabilizedHeegaardpairs}. 

We will not be able to address the issue of computing the linking invariants in this section.  Later, after we have learned more, we will develop a set of computable invariants for both stable and unstable double cosets.
\subsection{The quotient group of a Heegaard pair and its natural linking form.  
Solution to Problem 1}
\label{SS:the quotient group of a Heegaard pair}
We now introduce the concept of the quotient group of a Heegaard pair.  We will prove (see 
Theorem~\ref{T:quotient group of a Heegaard pair has a linking form}) that the quotient group of a Heegaard pair has a natural non-singular linking form.  This leads us to the concept of a `linked abelian group'.  In Corollary~\ref{C:stably isomorphic implies linked quotients} we show that, as a consequence of Theorem~\ref{T:quotient group of a Heegaard pair has a linking form}, the linked abelian group that is associated to a Heegaard pair is an invariant of its stable isomorphism class.  Corollary~\ref{C:stably isomorphic iff linked quotients} solves Problem 1.

\begin{lemma} \label{L:quotient-1}
 Let $(V; B, \Bbar)$ be a Heegaard pair and let $C = \{ x \in V \ | \ x \cdot (B + \Bbar) = 0\}$. Then $C = B \cap \Bbar$.
\end{lemma}

\begin{proof}
We have $x \in C$ if and only if $x \cdot (B + \Bbar) = 0$, which is true if and only if $x \cdot B = 0$ and $x \cdot \Bbar = 0$. 
Since $B, \Bbar$ are maximally isotropic, this is true if and only if $x \in B$ and $x \in \Bbar$.
\end{proof}

This lemma implies that, for lagrangian subspaces $B, \Bbar, B', \Bbar'$, if $B + \Bbar = B' + \Bbar'$, then 
$B \cap \Bbar = B' \cap \Bbar'$.

\begin{lemma} \label{L:quotient-2}
Every Heegaard pair $(V ; B, \Bbar)$ is a direct sum of Heegaard pairs of the form $(V_1; C, C) \oplus (V_2; D, \bar{D})$ where $C = B \cap \Bbar$ and $D \cap \bar{D} = 0$.
\end{lemma}

\begin{proof}
Since $B/C = B / (B \cap \Bbar) \cong (B + \Bbar) / \Bbar \subset V/ \Bbar$, the group $B/C$ is free and thus $C$ is a direct summand of 
$B$.  Thus $B = C \oplus D$ for some subgroup $D$ of $B$.  Let now $B^*$ be a dual complement of $B$ in $V$; the splitting 
$C \oplus D$ of $B$ then induces, dually, a splitting $C^* \oplus D^*$ of $B^*$, where $C \perp D^*$ and $D \perp C^*$ and where 
$C$ and $C^*$ are dually paired (by the symplectic form) and likewise $D, D^*$.  Thus $V_1 = C \oplus C^*$ and $V_2 = D \oplus D^*$ are 
{\it symplectic} subspaces of $V$, with $V_1 = V_2^\perp$ and $V = V_1 \oplus V_2 = C \oplus C^* \oplus D \oplus D^*$.  We claim 
that $\Bbar \subset C \oplus D \oplus D^*$.  Indeed, express $\bar{b} \in \Bbar$ as $\bar{b} = c + c^* + d + d^*$ with $c \in C$, etc. 
Since $\Bbar \supset C$ and $\Bbar$ is isotropic, we have $\bar{b} \cdot c' = 0$ for all $c' \in C$, \ie 
$(c+c^*+d+d^*)\cdot c' = c^* \cdot c'= 0$ for all $c' \in C$. But $C, C^*$ are dually paired, so $c^*$ must be zero.
 
Hence we have $C \oplus D \oplus D^* \supset \Bbar \supset C$.  But this implies that $\Bbar = C \oplus \bar{D}$, where 
$\bar{D} = \Bbar \cap (D \oplus D^*) = \Bbar \cap V_2= V_2$.  We have now shown that:
\begin{enumerate}[a)]
\item $V = V_1 \oplus V_2$
\item $B = C \oplus D$ with $C = B \cap \Bbar \subset V_1$, $D \subset V_2$
\item $\Bbar = C \oplus \bar{D}$ with $\bar{D} \subset V_2$
\end{enumerate}

Note that $C = B \cap \Bbar = (C \oplus D) \cap (C \cap \bar{D}) = C \oplus (D \cap \bar{D})$, so $D \cap \bar{D} = 0$.  To 
finish the proof, it suffices then to show that $(V_1; C, C)$ and $(V_2 ; D, \bar{D})$ are Heegaard pairs.  The former is trivially 
so since $V_1 = C \oplus C^*$, and for the same reason, $D$ is lagrangian in $V_2$.  We must then show that $\bar{D}$ is lagrangian 
in $V_2$.  It is certainly isotropic, since $\Bbar = C \oplus \bar{D}$ is so.  But let $x \in V_2$ be such that $x \cdot \bar{D} = 0$. 
We also have $x \cdot V_1 = 0$ since $V_1 \perp V_2$ and hence $x\cdot \Bbar = x \cdot (C \oplus \bar{D}) = 0$.  Since $\Bbar$ is 
maximally isotropic, $x \in \Bbar$ and hence $\Bbar \cap V_2 = \bar{D}$, showing $\bar{D}$ to be maximally isotropic in $V_2$.
\end{proof}

\begin{lemma} \label{L:quotient-3}
Let $(V; B, \Bbar)$ and $(V; B, \Bbar')$ be two Heegaard pairs such that $B + \Bbar = B + \Bbar'$. If the first is split as in the previous lemma, then the second has a splitting of the form $(V_1; C, C) \oplus (V_2; D, \bar{D}')$, where $D \cap \bar{D}' = 0$ and $D + \bar{D}' = D + \bar{D}$.
\end{lemma}

\begin{proof}
By Lemma \ref{L:quotient-1}, the fact that $B + \Bbar = B + \Bbar'$ implies that $C = B \cap \Bbar = B \cap \Bbar^\prime$. Examining the 
construction of Lemma \ref{L:quotient-2}, we see that since $B$ and $C$ are the same for both pairs, we may choose $B^*$ and $D$ the same, 
and hence $C^*$, $D^*$ and $V_1 = C \oplus C^*$, $V_2 = D \oplus D^*$ will also be the same.  Thus the second pair has a splitting 
satisfying all the requirements except possibly the last.  But we have 
$D + \bar{D}' = (B + \Bbar') \cap V_2 = (B + \Bbar) \cap V_2 = D + \bar{D}$.
\end{proof}

We define the {\it quotient} of a Heegaard pair $(V; B, \Bbar)$ to be the group $H = V / (B + \Bbar)$.  Clearly, isomorphic Heegaard 
pairs have isomorphic quotients.  Furthermore, stabilization does not change this quotient either, since 
$$(V \oplus X_k)/ [(B \oplus E_k) + (\Bbar \oplus F_k)] = (V \oplus X_k) / [(B + \Bbar) \oplus X_k] \cong V / (B + \Bbar).$$ 
Thus the isomorphism class of the quotient is an invariant of stable isomorphism classes of pairs.  We cannot conclude, however, that 
two Heegaard pairs are stably isomorphic if they have isomorphic quotients; there are further invariants.  To pursue these, we need 
the following concepts.

\begin{definition}
\label{D:linking form on T}
If $T$ is a finite abelian group, a {\it linking form} on $T$  is a symmetric bilinear $\Q/\Z$-valued form on $T$, where $\Q/\Z$ is the group of rationals $\mod ~ 1$. More generally, a linking form on any finitely generated abelian group $H$ means a linking form on its torsion subgroup $T$. A linking form $\lambda$ is {\it non-singular} if for every homomorphism $\varphi : T \to \Q/\Z$, there is a (necessarily unique) $x \in T$ such that $\varphi (y) = \lambda (x, y)$ for all $y \in T$. A group $H$ will be called a {\it linked group} if its torsion subgroup is endowed with a non-singular linking form.   $\|$
\end{definition}

\begin{theorem} \label{T:quotient group of a Heegaard pair has a linking form}
The quotient group of a Heegaard pair has a natural non-singular linking form.
\end{theorem}
\begin{proof}
Let the pair be $(V; B, \Bbar)$, the quotient be $H$, and its torsion subgroup be $T$.  Consider $x, y \in T$ and suppose that $m x = 0$. 
Lift $x, y$ to $u, v \in V$; then $m x = 0$ implies that $m u \in B + \Bbar$, say $m u = b + \bar{b}$.  Define $\lambda (x, y)$ to be 
$\frac{1}{m}(b \cdot v) \ \mod ~ 1$.  Note that if $n y = 0$ and hence $n v \in B + \Bbar$, say $n v = c + \bar{c}$ ($c \in B, \bar{c} \in \Bbar$) then we have 
$$\lambda (x, y) \equiv \frac{1}{m} (b \cdot v) \equiv \frac{1}{mn} (b \cdot n v) \equiv \frac{1}{mn} b \cdot (c + \bar{c}) \equiv \frac{1}{mn} b \cdot \bar{c} \ \mod ~ 1,$$ 
which gives a more symmetric definition of $\lambda (x, y)$.  We now verify the necessary facts about $\lambda$.
\begin{enumerate}[a)]
\item Independent of the choice of $b, \bar{b}$: if $b + \bar{b} = b' + \bar{b'}$, then we have $b' = b + \delta$, $\bar{b}' = \bar{b} - \delta$, with $\delta \in B \cap \Bbar$. But then $\frac{1}{mn} b' \cdot \bar{c} = \frac{1}{mn} ( b \cdot \bar{c} + \delta \cdot \bar{c}) = \frac{1}{mn} b \cdot \bar{c}$, since $\delta \cdot \bar{c} = 0$. Similarly, $\lambda(x,y)$ is independent of the choice of $c, \bar{c}$.
\item Independent of the lifting $u, v$: a different lifting $u'$ satisfies 
$u' = u + b_1 + \bar{b}_1$, so 
$$m u' = m u + m (b_1 + b_1') = (b + m b_1) + (\bar{b} + m \bar{b}_1),$$ 
and hence 
$$\lambda (x, y) \equiv \frac{1}{m} b \cdot v \equiv \frac{1}{m} (b + m b_1) \cdot v \ \mod ~ 1.$$ 
Similarly, $\lambda (x, y)$ does not depend on the lifting $v$ of $y$.
\item Independent of the choice of $m$: if $m' x =0$ also, with $m' u = b' + \bar{b}'$, then $m m' u = m' b + m' \bar{b} = m b' + m \bar{b}'$ and 
$$\frac{1}{m'} (b' \cdot v) \equiv \frac{1}{m m'} ( m b' \cdot v) \equiv_\text{ by a) } \frac{1}{m m'} (m' b \cdot v) \equiv \frac{1}{m} (b \cdot v) \ \mod ~ 1.$$
\item Bilinearity and symmetry: the first follows immediately from a). Then for symmetry, we have $\lambda (x, y) = \frac{1}{m n} (b \cdot \bar{c})$ and $\lambda(y,x) = \frac{1}{mn} (c \cdot \bar{b})$. But 
$$mn (u \cdot v) = (m u) \cdot (n v) = (b + \bar{b}) \cdot (c + \bar{c}) = b \cdot \bar{c} + \bar{b} \cdot c = b \cdot \bar{c} - c \cdot \bar{b} \equiv 0 \ \mod mn,$$ 
so $\frac{1}{m n} (b \cdot \bar{c}) \equiv \frac{1}{m n} (c \cdot \bar{b}) \ \mod ~ 1$.
\item Non-singularity: this is equivalent to the statement that $\lambda (x,y) \equiv 0 \ \mod ~ 1$ for all $y \in T$ implies that $x = 0$ in $T$.

Suppose then $x \in T$ and $\lambda (x,y) \equiv 0$ for all $y \in T$, and let $u$ be a lifting of $x$ to $V$. Now by Lemma \ref{L:quotient-2}, our Heegaard pair is a direct sum $(V_1; C, C) \oplus (V_2; D, \bar{D})$ with $C = B \cap \Bbar$ and $D \cap \bar{D} = 0$. Since the quotient $V_1 / (C + C)$ is {\it free}, $V_2$ projects onto $T$ and so the lifting of any torsion element may always be chosen in $V_2$. If $E$ is a dual complement of $D$ in $V_2$, then in fact the projection $V \to H$ will take $E$ onto $T$. Thus we may assume that $u \in E$. If $m x = 0$ in $T$ then $m u = d + \bar{d}$ for some $d \in D$, $\bar{d} \in \bar{D}$. The hypothesis that $\lambda (x, y) \equiv 0 \ \mod ~ 1$, all $y \in T$ is equivalent to $d \cdot v \equiv 0 \ \mod m$ for all $v \in V_2$. Since $V_2$ is symplectic, this implies that $d$ is {\it divisible by $m$} in $V_2$, that is, $d = m d'$ for some $d' \in V_2$. Now clearly $d' \cdot D = 0$, and hence $d' \in D$ since $D$ is maximally isotropic. Thus we have $m u = m d' + \bar{d}$, $\bar{d} = m (u - d')$ and we conclude similarly that $u - d' \in \bar{D}$, say $u-d' = \bar{d}'$. Thus $u = d' + \bar{d}' \in D + \bar{D}$, which implies that $x = 0$.  
\end{enumerate} 
 \end{proof}
\begin{remark} 
Note that the {\it maximal} isotropic nature of $B, \Bbar$ was used only in proving e); the weaker assumption 
that they are only isotropic still suffices to prove a)--d) and thus construct a natural linking on $H$.  $\|$
\end{remark}

\begin{lemma} \label{L:linking form-1}
Let $B \subset V$ be lagrangian, let $\Bbar$ be isotropic of rank $ \frac{1}{2} \rk V$, and suppose that $B \cap \Bbar = 0$. Then $\Bbar$ is lagrangian if and only if the induced linking form $\lambda$ on $H$ is non-singular.
\end{lemma}

\begin{proof}
We have already proved the necessity, so suppose that $\lambda$ is non-singular. By Definition~\ref{D:subspace of symplectic space}, we need 
only show that $\Bbar$ is a direct summand of $V$.  This is equivalent to showing that $V/\Bbar$ is torsion free, \ie that for $u \in V$, if $m u \in \Bbar$ for some nonzero $m$, then $u \in \Bbar$.  If then $mu \in \Bbar$ and $x$ is the image of $u$ in $H$, then $mx = 0$ in $H$, so $x$ is in the torsion group $T$.  But then $u$ lifts $x$ and $mu$ decomposes in $B + \Bbar$ as $0 + mu$. Hence if $y \in T$ is lifted to $v \in V$, we get $\lambda (x,y) = \frac{1}{m} (0 \cdot v) \equiv 0\ \mod ~ 1$, \ie $\lambda (x,y) \equiv 0$ for all 
$y \in T$.  By hypothesis, we have $x = 0$ in $T$, and hence $u \in B + \Bbar$, say $u = b + \bar{b}$.  Since $m u = mb + m \bar{b}$ is 
in $\Bbar$, $m b$ is also in $\Bbar$. But it is also in $B$, so must be zero, \ie $b=0$. Thus $u = \bar{b} \in \Bbar$.
\end{proof}

We have shown that the quotient of a Heegaard pair has the structure of a linked group in a natural way; clearly, isomorphic Heegaard pairs have isomorphic {\it linked} quotients: the Heegaard isomorphism induces an isomorphism on the quotients which preserves the linking. Let us see how the linked quotient behaves under stabilization.

\begin{lemma} \label{L:linking form-2}
The linked quotient of a stabilization of $(V; B, \Bbar)$ is canonically isomorphic to the unstabilized quotient.
\end{lemma}

\begin{proof}
The canonical isomorphism of the quotients is induced by the inclusion $V \hookrightarrow V \oplus X_k$, and we identify the two quotients in this way. To see that the linking defined by the two pairs are equal, let $x, y \in T$. Their liftings $u, v$ in $V \oplus X_k$ may be chosen to lie in $V \oplus 0$, since $X_k$ projects to $0$ in the quotient, and the splitting of $m u$ may then be chosen to be $(b,0) + (\bar{b}, 0)$. The stabilized linking number, defined thus, is then obviously the same as the unstabilized one.
\end{proof}

\begin{corollary} \label{C:stably isomorphic implies linked quotients}
The linked abelian group is an invariant of the stable isomorphism class of a Heegaard pair.
\end{corollary}

The remainder of this section is devoted to strengthening Corollary~\ref{C:stably isomorphic implies linked quotients} by showing that, 
in fact, two Heegaard pairs are stably isomorphic if and only if their linked quotients are isomorphic (see 
Corollary~\ref{C:stably isomorphic iff linked quotients}). It is easily verified that two linked groups are link-isomorphic if and only if they have link-isomorphic torsion groups and, mod their torsion groups, the same (free) rank.  

\begin{lemma} \label{L:linking form-3}
Let $(V; B, \Bbar)$ be a Heegaard pair with $B \cap \Bbar = 0$, and let $A$ be a dual complement of $B$. If $\bar{A}$ is the direct projection of $\Bbar$ into $A$, then there is a symplectic basis $a_i, b_i$ of $V$ ($a_i \in A$, $b_i \in B$) such that:
\begin{enumerate}[a)]
\item $m_i a_i$ is a basis for $\bar{A}$, for some integers $m_i \not= 0$;
\item $m_i a_i + \sum_j n_{ij} b_j$ is a basis for $\Bbar$, for some integers $n_{ij}$ such that $n_{ij}/m_i = n_{ji}/m_j$.
\end{enumerate}
\end{lemma}

\begin{remark}
Note that the hypothesis $B \cap \Bbar = 0$ is equivalent to the fact that the rank of $B+\Bbar$ is equal to the rank of $V$, \ie that 
the quotient is {\it finite}.  $\|$
\end{remark}

\begin{proof}
Since $B \cap \Bbar = 0$, the projection of $\Bbar$ into $A$ is 1-1, \ie 
$$\rk \bar{A} = \rk \Bbar = \frac{1}{2} \rk V = \rk A.$$ 
By Proposition \ref{P:structure of abelian groups}, there is a basis $a_i$ of $A$ such that $m_i a_i$ is a basis of $\bar{A}$, and $m_i \not= 0$ because $\rk \bar{A} = \rk A$. Let $b_i$ be the dual basis of $B = A^*$; then $a_i, b_i$ is a symplectic basis of $V$. The inverse of the projection $\Bbar \to \bar{A}$ takes $m_i a_i$ into a basis of $\Bbar$, which must then be of the form $\bar{b}_i = m_i a_i + \sum_k n_{ik} b_k$ for $n_{ij} \in \Z$. But $\bar{b}_i \cdot \bar{b}_j = 0$ for all $j$, \ie $m_i n_{ji} - m_j n_{ij} = 0$ for all $i,j$.
\end{proof}

\begin{lemma} \label{L:linking form-4}
Let $(V; B, \Bbar)$ and $(V; B, \Bbar')$ be two Heegaard pairs such that $B + \Bbar = B + \Bbar'$ and $B \cap \Bbar = B \cap \Bbar' = 0$. Then the linkings $\lambda, \lambda'$ induced on the common quotient $H$ are identical if and only if there is an $f \in \Sp(V)$ such that $f(B) = B$, $f(\Bbar) = \Bbar'$, and such that the automorphism $h$ of $H$ induced by $f$ is the identity map.
\end{lemma}

\begin{proof}
Certainly the condition is sufficient. To prove the necessity, let $A$ be a dual complement of $B$ and let $\bar{A}$, $\bar{A}'$ be
the projections of $\Bbar$, $\Bbar'$ into $A$.  Note that $\bar{A} = A \cap (B + \Bbar)$ and $\bar{A}' = A \cap (B + \Bbar')$. 
Hence $\bar{A} = \bar{A}'$.  As in the previous lemma we choose an $\Sp$-basis $a_i, b_i$ with $m_i a_i$ a basis of $\bar{A} = \bar{A}'$, and corresponding bases 
$$
\bar{b}_i = m_i a_i + \sum_j n_{ij} b_j \text{ of } \Bbar 
\qquad \text{ and } \qquad 
\bar{b}'_i = m_i a_i + \sum_j n'_{ij} b_j \text{ of } \Bbar'.
$$
Let $\beta_{ij} = \frac{n'_{ij}-n_{ij}}{m_i}$; the ``symmetry" conditions of the previous lemma on the $n_{ij}, n'_{ij}$ imply that $\beta_{ij} = \beta_{ji}$. Let $x_i$ be the image of $a_i$ in $H$. Since $m_i a_i \in B + \Bbar$, when calculating $\lambda(x_i,x_k)$ we may
choose the $B$-part of the lift of $x_i$ to be $- \sum_j n_{ij} b_j$, and we get 
$$
\lambda (x_i, x_k) = \frac{1}{m_i} (- \sum n_{ij} b_j) \cdot a_k = \frac{n_{ik}}{m_i}.
$$
Likewise we get $\lambda' (x_i, x_k) = \frac{n'_{ik}}{m_i}$.  By hypothesis, for all $i,k$ we have $\frac{n'_{ik}}{m_i} \equiv \frac{n_{ik}}{m_i} \ \mod ~ 1$, \ie $\beta_{ik} =  \frac{n'_{ik}-n_{ik}}{m_i} \equiv 0 \ \mod ~ 1$; that is, $\beta_{ik}$ is {\it integral} as well as symmetric.  Thus the transformation $f:V \rightarrow V$ which fixes the $b_i$ and takes $a_i$ to $a_i + \sum_j \beta _{ij} b_j$ is easily seen
to be {\it symplectic}.  We have $f(B) = B$ obviously, and 
$$
f(\bar{b}_i) = m_i ( a_i + \sum_j \beta_{ij} b_j) + \sum n_{ij} b_j = m_i a_i + \sum_j (m_i  \beta_{ij} +  n_{ij}) b_j = m_i a_i + \sum_j n'_{ij} b_j = \bar{b}'_i,
$$
so $f(\Bbar) = \Bbar'$. Finally, $h(x_i)$ is the image of $f(a_i) = a_i + \sum_j  \beta_{ij} b_j$, which is just $x_i$; this shows that $h = 1$.
\end{proof}

\begin{lemma} \label{L:linking form-5}
With hypotheses as in the preceding lemma, but omitting the assumption that $B \cap \Bbar = B \cap \Bbar' = 0$, the conclusion remains valid.
\end{lemma}

\begin{proof}
Again we need only prove the necessity. By Lemma \ref{L:quotient-3}, we split $(V; B, \Bbar)$ as $(V_1; C, C) \oplus (V_2; D, \bar{D})$, where $C = B \cap \Bbar = B \cap \Bbar'$, and $(V; B, \Bbar')$ as $(V_1; C, C) \oplus (V_2; D, \bar{D}')$. Since $D + \bar{D} = (B + \Bbar) \cap V_2 = (B + \Bbar') \cap V_2 = D + \bar{D}'$, both the $V_2$ pairs have the same quotient, namely the torsion subgroup $T$, and both these pairs define the same linking form. By the preceding lemma we have a map $f_2 \in \Sp(V_2)$ such that $f_2 (D) = D$, $f_2 (\bar{D}) = \bar{D}'$ and so that 
$$
\begin{diagram}
\node{V_2} \arrow[2]{s,l}{f_2} \arrow{see} \\
\node[3]{T} \\
\node{V_2} \arrow{nee}
\end{diagram}
$$
commutes. Let $f = \id_{V_1} \oplus f_2$; then $f$ satisfies the requirements.
\end{proof}

\begin{lemma} \label{L:linking form-6}
Let $(V; B, \Bbar)$ and $(V'; B', \Bbar')$ be Heegaard pairs of the same genus $g$, and let $h : H \to H'$ be an isomorphism of the quotients (not necessarily linking-preserving).  If $g > \rk H$, then there is a symplectic isomorphism $f : V \to V'$ lifting $h$
such that $f(B) = B'$ and $f(B + \Bbar) = B' + \Bbar'$.
\end{lemma}

\begin{proof}
Let $A, A'$ be dual complements of $B, B'$.  Then the projections $\pi, \pi'$ map $A, A'$ onto $H, H'$. Since $A, A'$ are free of rank $g > \rk H$, we have non-minimal presentations $A \to H$, $A' \to H'$, and hence by Theorem \ref{T:h lifts to the presentation level} 
there is an isomorphism $p : A \to A'$ lifting $h$. Let $q$ be the adjoint map of $p$ on $B = A^*$ (\ie $q = (p^*)^{-1}$); then $f = p \oplus q$ is a symplectic isomorphism of $A \oplus B = V$ to $A' \oplus B' = V'$. It clearly still lifts $h$, which implies that $f(B + \Bbar) = f (\ker \pi) = \ker \pi' = B' + \Bbar'$. Finally, $f(B) = B'$ by the construction of $f$.
\end{proof}

\begin{theorem} \label{T:linking isomorphism lifts to Heegaard isomorphism}
Let $(V; B, \Bbar)$ and $(V'; B', \Bbar')$ be two Heegaard pairs of genus $g$ and $h : H \to H'$ a link-isomorphism of their linked quotient groups. If $g > \rk H$, then $h$ lifts to a Heegaard isomorphism $j : (V; B, \Bbar) \to (V'; B', \Bbar')$.
\end{theorem}

\begin{proof}
Lift $h$ to $f$ as in the previous lemma, and put $\Bbar_1 = f^{-1} (\Bbar')$. Then $(V; B, \Bbar_1)$ is a Heegaard pair and $f$ maps it isomorphically to $(V'; B', \Bbar')$. Note that 
$$B + \Bbar_1 = f^{-1}(B' + \Bbar') = f^{-1} f (B + \Bbar) = B + \Bbar,$$ 
so the quotient of $(V; B, \Bbar_1)$ is also $H$.  Moreover, by construction the linking form on $H$ induced
by $(V,B,\Bbar_1)$ is identical to the linking form induced by $(V,B,\Bbar)$.
By Lemma \ref{L:linking form-5}, there is a map $g \in \Sp(V)$ such that $g(B) = B$, $g(\Bbar) = \Bbar_1$, and $g$ induces the identity map on $H$. Hence the map $f g$ also induces $h : H \to H'$, and $f g (B) = B'$, $f g (\Bbar) = f (\Bbar_1) = \Bbar'$.
\end{proof}

We are now ready to give our solution to Problem 1 of the introduction to this paper. 

\begin{corollary}
 \label{C:stably isomorphic iff linked quotients}
Two Heegaard pairs are stably isomorphic if and only if they have the same torsion free ranks, and their torsion groups are link-isomorphic; that is, if and only if they have isomorphic linked quotients. 
\end{corollary}

\subsection{The stabilization index.  Solution to Problem 3}
\label{SS:the stabilization index}

In this subsection we introduce the notion of a Heegaard presentation and define the genus of a Heegaard presentation.  
We return to the concept of the volume of a presentation of an abelian group, relating it now to Heegaard presentations.   
See Theorem~\ref{T:the volume again}.  At the end of this section we give the solution to Problem 3 of the Introduction to this article.
See Corollary~\ref{C:problem on stabilization index}.

\begin{definition}
\label{D:Heegaard presentation and its genus}
Let $H$ be a linked group. A {\it Heegaard presentation} of $H$ consists of a Heegaard pair $(V; B, \Bbar)$ and a surjection $\pi : V \to H$ such that:
\begin{enumerate}[a)]
\item $\ker \pi = B + \Bbar$
\item the linking induced on $H$ by means of $\pi$ is the given linking on $H$. 
\end{enumerate}
The {\it genus} of the presentation is the given genus of the pair.   
We will use the symbol $(V; B, \Bbar; \pi)$ to denote a Heegaard presentation.  $\|$
\end{definition}

\begin{theorem} \label{T:a linked group of rank r has a Heegaard presentation of genus r}
Every linked group $H$ has a Heegaard presentation of genus equal to the rank of $H$.
\end{theorem}

\begin{proof}
Let $H = F_k \oplus T$, where $F_k$ is free of rank $k$ and $T$ is torsion.  If $(V; B, \Bbar; \pi)$ is a Heegaard presentation of 
$T$ with genus equal to $\rk T$, then taking the direct sum with $(X_k; E_k, E_k; \rho)$ where $\rho: X_k \to F_k$ is a surjection with
kernel $E_k$ gives the required presentation for $H$.  Thus we need only prove the theorem for torsion groups $T$.  Let $V$ be symplectic 
rank of $2 \rk T$ and let $A, B$ be a dual pair in $V$. Let $\pi_A : A \to T$ be a presentation of $T$, which is possible since 
$\rk A = \rk T$. We may, by Proposition \ref{P:structure of abelian groups}, choose a basis $a_i$ of $A$ such that $m_i a_i$ is a basis of $\ker \pi_A$, and $m_i \not= 0$ since $T$ is a torsion group. If $b_i$ is the dual basis of $B$, then $a_i, b_i$ is a symplectic basis of $V$. Let now $x_i = \pi_A(a_i)$; the $x_i$'s generate $T$, and the order of $x_i$ in $T$ is $m_i$.

If now $\lambda$ is the linking form on $T$, choose rational numbers $q_{ij}$ representing $\lambda(x_i, x_j) \ \mod ~ 1$, which, since $\lambda(x_i, x_j) \equiv \lambda (x_j, x_i)$, may be assumed to satisfy $q_{ij} = q_{ji}$. Note that
$$
m_i q_{ij} \equiv m_i \lambda (x_i, x_j) \equiv \lambda (m_ix_i, x_j) \equiv \lambda(0, x_j) \equiv 0 \ \mod ~ 1,
$$
that is, $m_i q_{ij} = n_{ij}$ is {\it integral}.

We now define $\Bbar \subset V$ to be generated by $\bar{b}_i = m_i a_i + \sum_j n_{ij} b_j$. Clearly the map $\pi : a_i \mapsto x_i, b_i \mapsto 0$ is a surjection of $V$ onto $T$, and its kernel is generated by $m_i a_i$ and $b_i$, or just as well by $\bar{b}_i$ and $b_i$. In other words, $\ker \pi = B + \Bbar$.  Observe that $\Bbar$ is isotropic since 
$$
\bar{b}_i \cdot \bar{b}_k = m_i n_{ki} - m_k n_{ik} = m_i m_k q_{ki} - m_k m_i q_{ik} = 0.
$$
Hence we have a linking $\lambda'$ induced on $T$, as in Lemma \ref{L:linking form-1}, by the isotropic pair $B, \Bbar$. An easy calculation shows that $\lambda' = \lambda$, and hence is non-singular by hypothesis. Now $\rk \Bbar$ is obviously $ = \rk A = \frac{1}{2} \rk V$, and $B \cap \Bbar = 0$: for $\sum_i r_i \bar{b}_i \in B$ if and only if $\sum_i r_i m_i a_i = 0$, \ie if and only if $r_i = 0$ all $i$ (since $m_i \not= 0$). We now apply Lemma \ref{L:linking form-1} to conclude that $\Bbar$ is lagrangian and so $(V; B, \Bbar; \pi)$ is a Heegaard presentation of $T$ with genus equal to the rank of $T$.
\end{proof}

Our next goal is to show that a minimal Heegaard pair has a natural volume in the sense of 
\S \ref{SS:equivalence classes of minimal presentations}.

\begin{lemma} \label{L:the volume again}
Let $(V; B, \Bbar)$ be a minimal Heegaard pair of genus $g$ with quotient $H$. Then for any two dual complements $A_i$ of $B$ ($i = 1, 2$) the two presentations $A_i \to H$ induce the same volumes.  We shall call this volume the volume induced by the Heegaard pair.
\end{lemma}

\begin{proof}
The direct sum projection of $V = A_2 \oplus B$ onto $A_2$ gives a map $j : A_1 \to A_2$, and $j$ is an {\it isomorphism}. Clearly,
$$
\begin{diagram}
\node{A_1} \arrow[2]{s,l}{j} \arrow{see} \\
\node[3]{H} \\
\node{A_2} \arrow{nee}
\end{diagram}
$$
commutes, so the presentations are equivalent and have the same volume.
\end{proof}

We can strengthen Theorem~\ref{T:a linked group of rank r has a Heegaard presentation of genus r} in the
minimal volume case.

\begin{theorem} \label{T:the volume again}
Let $(V_i; B_i, \Bbar_i)$, ($i= 1,2$) be minimal Heegaard pairs of genus $g$ with quotients $H_i$ and induced volumes $\pm \theta_i$,
and let $h : H_1 \to H_2$ be a linking isomorphism.  Then $h$ lifts to a Heegaard isomorphism if and only if $h$ is also volume preserving.
\end{theorem}

\begin{proof}  Assume that $g = \rk H$. The proof of Theorem \ref{T:linking isomorphism lifts to Heegaard isomorphism} goes through exactly as is whenever we can lift $h$ to $f$ as in Lemma \ref{L:linking form-6}, and examining the proof of this lemma, we see that it also goes through as is if we can only lift $h$ to an isomorphism $p : A \to A'$ such that
$$
\begin{CD}
A @>>> H \\
@VpVV @VVhV \\
A' @>>> H'
\end{CD}
$$
commutes. Since $\rk A = \rk A' = g = \rk H$, the abelian groups $H$ and $H'$ have volumes $\theta, \theta'$ induced by these presentations, and Theorem \ref{T:equivalence of minimal presentations} tells us that $h$ lifts if and only if $h(\pm \theta) = \pm \theta'$, as desired.
\end{proof}

\begin{corollary} \label{C:stabilize}
Every Heegaard pair is isomorphic to a stabilization of a Heegaard pair whose genus is equal to the rank of the quotient.
\end{corollary}

\begin{proof}
Let $(V; B, \Bbar)$ be of genus $g$ and let its quotient be $H$ of rank $r$. If $g=r$ we are done. If $g > r$, then by Theorem 
\ref{T:a linked group of rank r has a Heegaard presentation of genus r} there is a Heegaard presentation of $H$ of genus $r$, and then by Corollary \ref{C:stably isomorphic iff linked quotients}, its stabilization of index $k = g-r > 0$ is isomorphic to $(V; B, \Bbar)$.
\end{proof} 

Problem 3 asked whether there is a bound, or even more a uniform bound on the stabilization index for arbitrary minimal inequivalent but stably equivalent pairs $\cH_1, \cH_2 \in \Gamma_g$. 

\begin{corollary}
\label{C:problem on stabilization index}
  If two Heegaard splittings of the same 3-manifold $W$ have the same genus, then their associated symplectic Heegaard splittings are either isomorphic or become isomorphic  after at most single stabilization.   In particular, if the genus of the Heegaard splitting is greater than the rank of $H_1(W;\Z)$, then the symplectic Heegaard splittings are always isomorphic. 
  \end{corollary}

\begin{proof}  Let $\tilde{h}, \tilde{h'}$ be Heegaard gluing maps of genus $g$ for the same 3-manifold. Let $h,h'$ be their images in 
$\Sp(2g,\Z)$.  By Theorem~\ref{T:double cosets and Heegaard pairs}  we know that the stable double cosets which characterize their stabilized symplectic Heegaard splittings  are in 1-1 correspondence with 
isomorphism classes of associated stabilized Heegaard pairs.  Let $(V;B,\bar{B}), (V'; B', \bar{B'})$ be the Heegaard pairs determined by $h =\rho^2(\tilde{h}), h' =\rho^2(\tilde{h'})$.   The fact that $\tilde{h}, \tilde{h'}$ determine the same 3-manifold $W$ shows that there is a linking isomorphism $H \to H'$ of their linked quotient groups.  Theorem~\ref{T:linking isomorphism lifts to Heegaard isomorphism} then asserts that, if $g > {\rm rank}H_1(W;\Z)$, then there is a Heegaard isomorphism  $(V;B,\bar{B}) \to (V'; B', \bar{B'})$.  In particular, the Heegaard splittings are equivalent.  By Theorem~\ref {T:a linked group of rank r has a Heegaard presentation of genus r}, every linked group $H$ has a Heegaard presentation of genus equal to the rank of $H$.  Thus, at most a single stabilization is required, and that only  if the genus is minimal and the Heegaard pairs are not isomorphic. 
\end{proof}


\newpage
\section{The classification problem for linked abelian groups}
\label{S:the classification problem for linked abelian groups}

We have reduced the problem of classifying symplectic Heegaard splittings to the problem of the classification of linked abelian groups. It remains to find a system of invariants that will do the job, and that is our goal in this section.
\subsection{Direct sum decompositions} 
\label{SS:direct sum decompositions}

We begin by showing that the problem of finding a complete system of invariants for a linked group $(H, \lambda)$ reduces to the problem of studying the invariants on the $p$-primary summands of the torsion subgroup $T$ of $H$.  
 \begin{theorem}[\cite{Sei}]  
 \label{T:linkings always split as direct sums}
 Every linking form on $T$ splits as a direct sum of linkings associated to the $p$-primary summands of $T$,
 and two linking forms are equivalent if and only if the linkings on the summands are equivalent.
 \end{theorem}
 \begin{proof}
 Let $x, y \in T$ where $x$ has order $m$ and $y$ has order $n$. Then
 $\lambda(x,y) = \lambda (x, y + 0) =  \lambda (x, y) + \lambda (x,0),$ hence  $\lambda (x,0) = 0.$
 From this it follows that
 $
 n \lambda (x, y) = \lambda(x, ny) = \lambda (x, 0) = 0 (mod~1)$ because $ny=0$.
 By the symmetry of linking numbers, we also have $m \lambda (x, y) = 0 (mod ~1).$  Therefore
 $
 \lambda (x,y) = \frac{r}{m} = \frac{s}{n} (mod ~ 1)$ for some integers $r,s$. This implies that $ \lambda(x,y) = \frac{t}{(m,n)} (mod ~ 1)$ for some integer $t$, where $(m,n)$ is the greatest common divisor of $m$ and $n$.
 Thus if $x, y$ have order $p^a, q^b$ where $p, q$ are distinct primes, then $\lambda(x,y) = 0$.
 Thus the linking on $T$ splits into a direct sum of linkings on the $p$-primary summands, as claimed.
 \end{proof}
  
In view of Theorem \ref{T:linkings always split as direct sums}, we may restrict our attention to a summand $T(p_j)$ of $T$ of 
prime power order $p_j^t$, where $p_j \in \{p_1,p_2,\dots,p_k\}$, the set of prime divisors of the largest torsion coefficient
$\tau_t$.  This brings us, immediately, to a very simple question: how do we find the linking form on $T(p_j)$ from a symplectic Heegaard splitting?  Our next result addresses this issue.
\begin{corollary}
\label{C:Seifert and Reidemeister}
Let $T$ be the torsion subgroup of $H_1(W;\Z)$.  Let  $\cQ^{(2)} = ||q_{ij}||$ and $\cP^{(2)} = {\rm Diag}(\tau_1,\tau_2,\dots,\tau_t)$ be the matrices that are given in 
Theorem~\ref{T:partial normal form}.  
\begin{enumerate}
\item [(1)] The $t\times t$ matrix 
$\cQ^{(2)}(\cP^{(2)})^{-1} = \| \lambda (y_r, y_j) \| = ||{q_{ij} \over \tau_j} ||$
 determines a linking on $H$.  
 \item [(2)] The linking matrices that were studied by Seifert in {\rm \cite{Sei}}  are the direct sum of $k$ distinct $t\times t$ matrices, one for each prime divisor $p_d$ of $\tau_t$. Each  summand represents the restriction of the linking in (1) to the cyclic summands of $T$ whose order is a fixed power of $p_d$. The one that is associated to the prime $p_d$ is a matrix of dimension at most  $t\times t$:
 \begin{equation}
\label{E:Seifert's linking matrix}
\lambda(g_{id}, g_{jd}) = ||{\tau_i\tau_jq_{ij} \over (p_d^{e_{i,d}})(p_d^{e_{jd}})\tau_j}|| = ||{\tau_iq_{ij} \over (p_d^{e_{id}})(p_d^{e_{jd}})}|| 
\end{equation} 
\end{enumerate}
 \end{corollary}
\begin{proof}   
(1) The easiest way to see that $\cQ^{(2)}(\cP^{(2)})^{-1}$ is a linking on $H$ is from the geometry. The matrix $\cQ^{(2)}$ records the number of appearances (algebraically) of each homology basis element $b_i$ in $h(b_j)$. The curves which represent the $b_i$'s bound discs $D_i$ in $N_g$, and the curves which represent the $h(b_j)$'s bounds discs $\bar{D}_j$ in $\Nbar_g$. The $g_{ij}$'s are intersection numbers of $h(b_j)$ with $D_i$ or of $b_i$ with $\bar{D}_j$. Dividing by $\tau_j$, intersection numbers go over to linking numbers.  Note that the submatrix $\cP^{(2)} \cQ^{(2)}$ is symmetric by (\ref{E:symplectic constraints-2}), hence $\cQ^{(2)} (\cP^{(2)})^{-1}$  is likewise symmetric, as it must be because $\cQ^{(2)} (\cP^{(2)})^{-1}$ is a linking. 

(2) By Theorem~\ref{T:linkings always split as direct sums}, $T$ is a direct sum of $p$-primary groups $T(p_1) \oplus \cdots \oplus T(p_k)$.  From this it follows that the linking on $T$ also splits as a direct sum of the linkings associated to $T(p_1),\dots,T(p_k)$.  

We focus on one such prime $p = p_d$. By Theorem~\ref{T:fundamental theorem for f.g. abelian groups}  the $p$-primary group $T(p)$ splits in a unique way as a direct sum of cyclic groups whose orders are powers of $p$, moreover the powers of $p$ which are involved occur in a non-decreasing sequence, as in (\ref{E:inequality on prime powers}) of Theorem~\ref{T:fundamental theorem for f.g. abelian groups}.
The generators of these groups are ordered in a corresponding way, as $g_{1,d},g_{2,d},\dots,g_{t,d}$, where distinct generators $g_{i,d}, g_{{i+1},d}$ may generate cyclic groups of the same order, and where it is possible that the first $q$ of these groups are trivial.  
As in the statement of Theorem~\ref{T:fundamental theorem for f.g. abelian groups},  the generators $g_{i,d}$ and  $y_i$ are related by (\ref{E:gid from y_i}).  The expression on the right in (\ref{E:Seifert's linking matrix}) follows immediately.   There are $k$ such matrices in all, where $k$ is the number of distinct prime divisors of $\tau_t$.     \end{proof}

\subsection{Classifying linked $p$-groups when $p$ is odd.  Solution to 
Problem 2, odd $p$.} 
\label{SS:classifying linked p-groups, p odd}  
In this section we describe Seifert's classification theorem for the case when all of the torsion coefficients are odd.  
Seifert studied the $t\times t$ matrix $\lambda(g_i, g_j)$ in (\ref{E:Seifert's linking matrix}) belonging to a fixed prime $p= p_d$. 
To explain what he did, we start with the inequalities in (\ref{E:inequality on prime powers}), but restrict to a subsequence of 
cyclic groups all of which have the same prime power order.  We simplify the notation, using the symbols
$$\varepsilon_1=\cdots=\varepsilon_1<\varepsilon_2=\cdots = \varepsilon_2 < \cdots <\varepsilon_r=\varepsilon_r=\cdots\varepsilon_r$$ 
in place of the powers $e_{i,d}$ which appear in 
(\ref{E:Seifert's linking matrix}).  The linking matrix then divides into blocks whose size is determined by the number of times, 
denoted $t_i$, that a given power, say $\varepsilon_i$, is repeated.  Among these, the blocks that interest us are the square blocks 
whose diagonals are along  the main diagonal of the linking matrix.  There will be $r$ such blocks of dimension $t_1,\dots,t_r$  
if $r$ distinct powers $p^{\varepsilon_i}$ occur in the subgroups of $T$ that are cyclic with order a power of $p$:
\begin{equation} \label{E:Seifert's matrix}
\big\| \lambda (g_i, g_j) \big\|  \quad = \quad 
\left(\begin{array}{cccc}
\frac{\cA_1}{p^{{\varepsilon_1}}} & * & \cdots & * \\
* & \frac{\cA_{2}}{p^{\varepsilon_{2}}} & \cdots & * \\
\vdots & \vdots & \ddots & \vdots \\
* & \cdots & \cdots   &  \frac{\cA_{r}}{p^{\varepsilon_{r}}}
\end{array} \right)
\end{equation}
The stars relate to linking numbers that we shall not consider further.   

\begin{theorem}[\cite{Sei}]
\label{T:Seifert's theorem on the linking form}
Two linkings of $T(p)$ are equivalent if and only if the corresponding box determinants $| \cA_1 |, | \cA_{2} |, \dots, | \cA_{r} |$
have the same quadratic residue characters $\mod p$.
\end{theorem}

Summarizing, we can now give the promised solution to Problem 2 in the case when $T$ has no 2-torsion.
\begin{theorem}
\label{T:compute invariants for stable equivalence, p odd}
\begin{enumerate}
\item  We are given the gluing map $\tilde{h}\in\tilde{\Gamma}_g$ for a Heegaard splitting of genus $g$ of a 3-manifold $W$. Let 
$$\cH = \rho^2(\tilde{h}) = \mattwotwo{\cR}{\cP}{\cS}{\cQ} \in \Sp(2g,\Z).$$
\item Use the methods described in the proof of Theorem II.9 of \cite{Newman} to find matrices $\cU,\cV$ such that $\cU\cP\cV = \cP^{(1)} = \diag (1, 1, \dots, 1, \tau_1, \tau_2, \dots, \tau_t, 0, \dots, 0)$.    
\item Use the methods described in the proof of Theorem~\ref{T:partial normal form} to find the equivalent matrix $\cH'$.   
\item  Let $y_i$ be a generator of the subgroup of order $\tau_i$ of $T$.  Let $p_1< p_2<\dots < p_k$ be the primes divisors of $\tau_1, \tau_2, \dots, \tau_t$. Compute the elements $g_{ij}$, using {\rm (\ref{E:gid from y_i})}.  Then compute the $k$ symmetric matrices $\big\| \lambda (g_i, g_j) \big\|$, using (\ref{E:Seifert's linking matrix}) above.  
\item  Using the matrices $\big\| \lambda (g_i, g_j) \big\|$, determine the submatrices $\cA_1,\dots,\cA_r$ that are shown in (\ref{E:Seifert's matrix}) above.  Here each block matrix $\cA_q$ belongs to a sequence (possibly of length 1) of cyclic subgroups of $T(p)$ of like prime power order. The matrix $\cA_q$ might be the identity matrix. Compute the quadratic residue characters $\mod p$ of the determinants $| \cA_1 |, | \cA_{2} |, \dots, | \cA_{r} |$.  Repeat this for each $p$.
\item By Corollary~\ref{C:stably isomorphic iff linked quotients} and Theorem~\ref {T:Seifert's theorem on the linking form}, the rank $r$, the torsion coefficients $\tau_1,\dots,\tau_t$ and the complete array of quadratic reside characters mod $p$ are the complete set of topological invariants of the associated symplectic Heegaard splitting of $W$ that is determined by $\cH$.  
\end{enumerate} 
\end{theorem}

\newpage

\subsection{Classifying linked $p$-groups, $p=2$.  Solution to 
Problem 2, $p=2$.} 
\label{SS:classifying linked p-groups, p=2} 

If $(H, \lambda)$ is a linked $p$-group, we have seen that $H$ may be decomposed into orthogonal summands $B_j$, each 
of which is a free $\Z_{p^j}$-module, that is, a direct sum of some number $r_j$ of copies of $\Z_{p^j}$; $r_j$ is the 
{\it rank} of $B_j$, and $B_j$ is the $j^{\Th}$ {\it block}.  The $r_j$'s are as usual invariants of $H$ alone. 
For odd $p$, the linking type of $\lambda |_{B_j}$ is an invariant of $\lambda$, and these types (which are determined 
by a Legendre symbol) give a complete set of invariants.  For $p=2$, however, no such decomposition is possible; 
this is due to the fact that the linking type of $B_j$ {\it is no longer an invariant}.  Here is a typical example 
of the kind of thing that can happen.  

\begin{definition}  An element $x \in H$ is said to be {\it primitive} if it generates a direct summand of $H$, or, equivalently, if $x \notin 2H$.  $\|$
\end{definition}

\begin{example}
\label{example:p2problem}
Let $H = \Z_{2^{n-1}} \oplus \Z_{2^n} \oplus \Z_{2^n}$, so $H$ has rank 3.  Let 
$$\cC= \bmattwotwo{0}{1}{1}{0} \ \ \   {\rm and} \ \ \  \cD=\bmattwotwo{2}{1}{1}{2},$$
and consider the two linking forms $\lambda$ and $\lambda'$ on $H$ whose matrices with respect to the standard basis are
\begin{equation}
\label{E:basic forms} 
(\frac{1}{2^{n-1}}) \oplus (\frac{1}{2^n}\cC) \ \ \  
  {\rm and} \ \ \ (\frac{-3}{2^{n-1}}) \oplus (\frac{1}{2^n}\cD),
  \end{equation}
respectively, where $(\frac{1}{2^{n-1}})$ and $(\frac{-3}{2^{n-1}})$ denote $1\times 1$ matrices.  We claim 
that $\lambda \simeq \lambda'$.  Indeed, let $\vec{e} = (e_1, e_2, e_3)$ be the standard basis of $H$ and put
$$
e_1' = e_1 + 2 e_2 - 2 e_3,
\qquad \qquad
e_2' = e_1 + e_2,
\qquad \qquad
e_3' = e_1 - e_3.
$$
Now, we have:
$$
3e_1 = -e_1' + 2e_2' + 2e_3', \qquad 3e_2 = e_1' + e_2' - 2 e_3',
\qquad  3e_3 = -e_1' + 2 e_2' - e_3'.
$$
Since 3 is a unit in $\Z_{2^n}$, it follows that $\vec{e'} = (e_1',e_2',e_3')$ is also a basis.  Now, observe that
\begin{align*}
\lambda(e_1',e_1') &= \lambda(e_1,e_1)-8\lambda(e_2,e_3) = \frac{-3}{2^{n-1}} 
& \lambda(e_2',e_2') &= \lambda(e_1,e_1) = \frac{2}{2^n}\\
\lambda(e_3',e_3') &= \lambda(e_1,e_1) = \frac{2}{2^n}                        
& \lambda(e_1',e_2') &= \lambda(e_1,e_1)-2\lambda(e_3,e_2)=0\\
\lambda(e_1',e_3') &= \lambda(e_1,e_1)-2\lambda(e_3,e_2)=0                    
& \lambda(e_2',e_3') &= \lambda(e_1,e_1)-\lambda(e_2,e_3)=\frac{1}{2^n}
\end{align*}
Thus  $\lambda$ with respect to the basis $\vec{e'}$ is equal to $\lambda'$ with respect to the basis $\vec{e}$. 
But $\frac{1}{2^n}\cC$ is {\it not} isomorphic to $\frac{1}{2^n}\cD$ as a linking
on $\Z^2_{2^n}$ 
when $n \geqslant 2$.  For example, $\lambda(x,x)=0$ has no primitive solutions for the second linking, whereas it does for the first. $\|$
\end{example}

Our example shows that we must approach the case $p = 2$ in a different manner. Burger (see \cite{Bu}) reduced the classification of linked $p$-groups to the classification of quadratic forms (more precisely, 
symmetric bilinear forms) over $\Z_{p^n}$, and his procedure gives in theory at least a complete set of invariants, but for 
$p=2$ they are inconveniently cumbersome.  Our goal in this section is to reduce his invariants to a simple, manageable set and
to demonstrate how to calculate them.  We will proceed  as follows:
\begin{itemize}
\item In $\S$\ref{SSS:Decomposing a linked abelian $2$-group} we show how to decompose a linking form (in a non-canonical way) into a direct sum of 3 types of basic forms.
\item In $\S$\ref{SSS:Burger's numerical invariants}  we discuss a variant of the Burger's numerical invariants.
\item Our variant of Burger's invariants cannot achieve arbitrary values.  To determine the values they can achieve, we
calculate the values of our numerical invariants on the basic forms.  This calculation will be the basis of the next step, given in $\S$\ref{SSS:Calculating the numerical invariants for the basic forms}.
\item Finally, in $\S$\ref{SSS:Reduction to the phase invariants} we show that all the information in Burger's numerical invariants is contained in
a simpler set of mod $8$ {\em phase invariants} together with the ranks of the various blocks.
\end{itemize}
The section ends with several examples.

\subsubsection{Decomposing a linked abelian $2$-group}
\label{SSS:Decomposing a linked abelian $2$-group}

Strengthening a result of Burger \cite{Bu}, Wall \cite{W1} showed how to decompose a linked abelian $2$-group
into an orthogonal direct sum of certain basic linking forms.  In this section, we review this result and
give a proof of Wall's theorem which builds on Burger's original ideas and which is better suited for
computations.  

\begin{definition} \label{D:basic linking forms}
The {\em basic linking forms} on a finite abelian $2$-group are the following
\begin{itemize}
\item The {\em unary forms}.  These are the forms on $\Z_{2^j}$ for $j \geq 1$ whose matrices are $\left(\frac{a}{2^j}\right)$
for odd integers $a$.
\item The two {\em binary forms}.  These are the forms on $(\Z_{2^j})^2$ for $j \geq 1$ whose matrices are
either $\frac{1}{2^j} \cC$ or $\frac{1}{2^j} \cD$, where $\cC$ and $\cD$ were defined in Example~\ref{example:p2problem}.  $\|$
\end{itemize}
\end{definition}

We can now state Wall's theorem.

\begin{theorem}[{Wall, \cite{W1}}]
\label{theorem:p2decomp}
Let $(H,\lambda)$ be a linked abelian $2$-group.  Then $(H,\lambda)$ is isomorphic to an orthogonal 
direct sum $(H_1,\lambda_1) \oplus \cdots \oplus (H_n,\lambda_n)$, where the $(H_i,\lambda_i)$ are basic linking
forms.
\end{theorem}

\begin{proof}  We will use the following well-known result about solving congruences mod $p^k$. 
\begin{lemma}[{Hensel's Lemma \cite[Theorem 2.23]{NivenZuckermanNumbers}}]  
Let $f(x)$ be a polynomial with integer coefficients, let $p$ be a prime and let $k \geq 2$ be an integer.  
Assume that the integer $r$ is a solution to the equation  $f(x) = 0 \ ({\rm mod} \ p^{k-1})$, 
and moreover assume that $f'(r) \not= 0 \ {\rm mod}\  p$.  Then $f(r+tp^{k-1}) = 0 \ ({\rm mod} \ p^k)$ for some (unique) integer 
$t$ with $0\leq t \leq p-1$.
\end{lemma}

The proof of Theorem~\ref{theorem:p2decomp} begins with a lemma:
\begin{lemma}
\label{lemma:binaryclassification}
Let $H = \Z_{p^j}\oplus \Z_{p^j}$ and let $\lambda$ be the linking form on $H$ whose matrix with respect to the standard basis $(e_1,e_2)$ for $H$ 
is $\frac{1}{2^j}\mattwotwo{2m}{1}{1}{2n}$.  Then $(H,\lambda)$ is isomorphic to the linking form
whose matrix is $\frac{1}{2^j} \cC$ if both $m$ and $n$ are odd and to the linking form whose
matrix is $\frac{1}{2^j} \cD$ otherwise.
\end{lemma}
\begin{proof}
Assume first that one of $m$ and $n$ (say $m$) is even.  Observe that if $a$ is an integer,
then since $\lambda$ is defined by the matrix $\frac{1}{2^j}\mattwotwo{2m}{1}{1}{2n}$, we have:
$$\lambda((1,a),(1,a)) = 2m+2a^2 n+2a = 2(n a^2 + a + m).$$
Now, since $m$ is even it follows that $a=1$ is a solution to the equation
$n a^2 + a + m = 0$ mod $2$, so Hensel's lemma implies that there is some odd
integer $\tilde{a}$ so that $n \tilde{a}^2+\tilde{a}+m = 0$ mod $2^j$.  Set $\vec{v} = (1,a)$.  Since $\lambda$ is
nondegenerate, there is some $\vec{w} \in H$ so that $\lambda(\vec{v},\vec{w}) = \frac{1}{2^j}$.
Let $k$ be so that $\lambda(\vec{w},\vec{w}) = 2k$.  The pair $\{\vec{v}, \vec{w} - k\vec{v}\}$ is then a new basis, 
and with respect to this basis $\lambda$ has the matrix $\frac{1}{2^j} \mattwotwo{0}{1}{1}{0}$, as desired.

Now assume that both $m$ and $n$ are odd.  By the same argument as in the previous paragraph (but solving
$n a^2+a+m=1$ mod $2^j$),
we can assume that $m = 1$, so the matrix for $\lambda$ is $\frac{1}{2^j}\mattwotwo{2}{1}{1}{2n}$.  Set
$\vec{v} = (1,0)$, and for an arbitrary integer $c$ define $\vec{w}_c = (c,1-2c)$.  Observe that
$$\lambda(\vec{v},\vec{w}_c) = 2c + (1-2c) = 1.$$
Moreover,
$$\lambda(\vec{w}_c,\vec{w}_c) = 2c^2 + 2c(1-2c) + 2n(1-2c)^2 = 2((4n-1)c^2 + (1-4n)c+n).$$
Now, since $n$ is odd it follows that $c = 1$ is a solution to
$$(4n-1)c^2 + (1-4n)c+n = 1\ \ ({\rm mod}\ 2),$$ 
so Hensel's lemma implies that there is some odd integer $\tilde{c}$ so that 
$$(4n-1)\tilde{c}^2 + (1-4n)\tilde{c}+n = 1\ \ ({\rm mod}\ 2^j).$$ 
It follows that $\lambda(\vec{w}_{\tilde{c}}, \vec{w}_{\tilde{c}}) = 2$.
Observe that since $\tilde{c}$ is odd the vectors $\vec{v}$ and $\vec{w}_{\tilde{c}}$ form
a basis, and with respect to this basis $\lambda$ has the matrix $\frac{1}{2^j} \mattwotwo{2}{1}{1}{2}$, as desired.
\end{proof}

We are ready to prove Theorem \ref{theorem:p2decomp}. 
For $x \in H$, define the {\em order} of $x$ to be the smallest nonnegative integer $n$ so that $2^n x = 0$.  Assume
first that there is some primitive $x \in H$ of order $j$ so that $\lambda(x,x) = \frac{a}{2^j}$ with $a$ odd.  Construct
a basis $x=x_1,x_2,\ldots,x_N$ for $H$.  Observe now that for any $2 \leq i \leq N$, we have $\lambda(x,x_i) = \frac{b_i}{2^j}$
for some integer $b_i$.  Since $a$ is odd, we can write $b_i = c_i a$ mod $2^j$ for some integer $c_i$.  Set
$x_i' = x_i - c_i x$.  Observe that $x=x_1,x_2',\ldots,x_N'$ is a new basis for $H$, and that moreover $(H,\lambda)$
has an orthogonal direct sum decomposition
$$\langle x \rangle \oplus \langle x_2',\ldots,x_N' \rangle.$$
Since $\lambda$ restricted to $\langle x \rangle$ is the basic form with matrix $\left(\frac{a}{2^j}\right)$, we are
done by induction.

We can therefore assume that for all $x \in H$, if $j$ is the order of $x$ then $\lambda(x,x) = \frac{2n}{2^j}$ for
some integer $n$.  Fixing some primitive $x \in H$ of order $j$ and letting $n$ be the integer with $\lambda(x,x) = \frac{2n}{2^j}$,
the fact that $\lambda$ is non-singular implies that there is some $y \in H$ so that $\lambda(x,y) = \frac{1}{2^j}$.  Necessarily
$y$ has order $j$, so we can write $\lambda(y,y) = \frac{2m}{2^j}$ for some integer $m$.  Construct a basis
$x=z_1,y=z_2,z_3,\ldots,z_N$ for $H$, and for $3 \leq i \leq N$ write $\lambda(x,z_i) = \frac{h_i}{2^j}$ and
$\lambda(y,z_i) = \frac{k_i}{2^j}$.  Since the matrix $\mattwotwo{2n}{1}{1}{2m}$ is invertible mod $2^j$, we can find
integers $a_i$ and $b_i$ so that
$$\bmattwotwo{2n}{1}{1}{2m} \cdot \bvectortwoone{a_i}{b_i} = \bvectortwoone{-h_i}{-k_i}$$
mod $2^j$.  For $3 \leq i \leq N$, set $z_i' = z_i + a_i x + b_i y$, and observe that $\lambda(x,z_i')=\lambda(y,z_i')=0$.  The linked
group $(H,\lambda)$ now has an orthogonal direct sum decomposition
$$\langle x,y \rangle \oplus \langle z_3',\ldots,z_N' \rangle.$$
Since $\lambda$ restricted to $\langle x,y \rangle$ is the linking form with matrix $\mattwotwo{2n}{1}{1}{2m}$,
Lemma \ref{lemma:binaryclassification} says that it is one of the basic forms, and we are done by induction.   This completes the proof of Wall's theorem.
\end{proof}

\subsubsection{Burger's numerical invariants}
\label{SSS:Burger's numerical invariants}

In \cite{Bu}, Burger constructs a complete set of numerical invariants of a linked abelian $p$-group $(H,\lambda)$.
The following formulation of Burger's result was suggested by Fox in \cite{FOXBURGERREVIEW}, and is
manifestly equivalent to Burger's original formulation.  Define the {\em degree} of an abelian $p$-group to be
the smallest nonnegative integer $n$ so that $p^n H = 0$.

\begin{definition}
Let $(H,\lambda)$ be a linked abelian $p$-group of degree $n$, and let $a = \frac{k}{p^n}$ for some $0 \leq k < p^n$.
We then define $\N_a(\lambda)$ to be the number of solutions $x \in H$ to the equation $\lambda(x,x)=a$. $\|$
\end{definition}

\begin{theorem}[Burger, \cite{Bu}]
Fix an abelian $p$-group $H$ of degree $n$.  Two linking forms $\lambda$ and $\lambda'$ on $H$ are then isomorphic
if and only if $\N_a(\lambda)=\N_a(\lambda')$ for all $a = \frac{k}{p^n}$ with $0 \leq k < p^n$. $\|$
\end{theorem}

We now confine ourselves to the case $p=2$.  Fix a linked abelian $2$-group $(H,\lambda)$ of degree $n$.  We shall use the notation 
$E(z)$ for $e^{2 \pi i z}$.  Also, to simplify our formulas we shall denote $\lambda(x,x)$ by
$x^2$.  Observe now that if $x \in H$, then $x^2 = \frac{k}{2^n}$ where $k$ is an integer which is
well-defined mod $2^n$.  Hence if $b \in \Z_{2^n}$, then the expression $b x^2 = \frac{b k}{2^n}$ is
a well-defined real number mod $1$, so the number $E(b x^2)$ is well-defined.  The
following definition therefore makes sense.

\begin{definition}
For $b \in \Z_{2^n}$, define $\overline{\Gamma}_b(\lambda) := \sum_{x \in H} E(b x^2)$. $\|$
\end{definition}

\noindent
Observe that we clearly have the identity
$$\overline{\Gamma}_b(\lambda) = \sum_{k=0}^{2^n-1} \N_{\frac{k}{p^n}}(\lambda) E(\frac{b k}{2^n}).$$
In other words, the numbers $\overline{\Gamma}_b(\lambda)$ are the result of applying the 
{\em discrete Fourier transform} (see \cite{DiscreteFourier}) to the numbers $\N_a(\lambda)$.
In particular, since the discrete Fourier transform is invertible, it follows that the
numbers $\overline{\Gamma}_b(\lambda)$ for $b \in \Z_{2^n}$ also form a complete set of invariants for linking
forms on $H$.  In fact, since $\overline{\Gamma}_0(\lambda) = \sum_{x \in H} 1 = |H|$, we only need
$\overline{\Gamma}_b(\lambda)$ for $b \neq 0$.

However, there is a certain amount of redundancy among the $\overline{\Gamma}_b(\lambda)$.  
Indeed, set $\theta = E(1/2^n)$.  Observe that $\theta$ is a $(2^n)^{\Th}$ root of unity and that the numbers
$\overline{\Gamma}_b(\lambda)$ lie in $\Q(\theta)$.  If $a$ is any odd integer, then $\theta^a$ is another
$(2^n)^{\Th}$ root of unity, and $\theta \mapsto \theta^a$ defines a Galois automorphism of
$\Q(\theta)$, which we will denote by $g_a$.  Consider a nonzero $b \in \Z_{2^n}$.  Write $b = 2^k b_0$, where
$b_0$ is odd and $k < n$.  Since $2^n x^2$ is an integer for all $x \in H$, we have
\begin{align*}
\overline{\Gamma}_b(\lambda) & = \sum_{x \in H} e^{\frac{2\pi i (2^k b_0) (2^n x^2)}{2^n}}
= \sum_{x \in H} {\left(e^{\frac{2\pi ib_0}{2^n}}\right)}^{2^{k+n}x^2}
= g_{b_0} \left(\sum_{x \in H} {\left(e^{\frac{2\pi i}{2^n}}\right)}^{2^{k+n}x^2}\right) \\
& = g_{b_0} \left(\sum_{x \in H} E(2^k x^2)\right)
\ = \ g_{b_0} (\overline{\Gamma}_{2^k}(\lambda)).
\end{align*}
We conclude that the set of all $\overline{\Gamma}_b(\lambda)$ for $b \in \Z_{2^n}$
can be calculated from the set of all $\overline{\Gamma}_{2^k}(\lambda)$ for $0 \leq k < n$.
This discussion is summarized in the following definition and lemma.

\begin{definition}
Let $\lambda$ be a linking form on $H$.  For $0 \leq k < n$, define $\Gamma_k(\lambda) := \overline{\Gamma}_{2^k}(\lambda)$. $\|$
\end{definition}

\begin{lemma}
\label{lemma:gammakcomplete}
The set of numbers $\Gamma_k(\lambda)$ for $0 \leq k < n$ form a complete set of invariants for linking forms on $H$.
\end{lemma}
 
\subsubsection{Calculating the numerical invariants for the basic forms}
\label{SSS:Calculating the numerical invariants for the basic forms}

By Theorem \ref{theorem:p2decomp}, we can decompose any linked abelian group
into a direct sum of basic linking forms.  The following lemma shows how this reduces
the calculation of the numbers $\Gamma_k(\lambda)$ to the knowledge of the $\Gamma_k(\cdot)$ for the basic
forms.
\begin{lemma}
\label{lemma:gammamult}
If $(H, \lambda)$ is the orthogonal direct sum $A_1 \oplus A_2$ with linkings $\lambda_i$ on $A_i$, then for $0 \leq k < n$
we have $\Gamma_k(\lambda) = \Gamma_k(\lambda_1) \cdot \Gamma_k(\lambda_2)$.
\end{lemma}
\begin{proof} For $x = (x_1,x_2)$:
\begin{align*}
\Gamma_k(\lambda) &= \sum_{x \in H} E(2^k x^2) = \sum_{x_1 \in A_1} \sum_{x_2 \in A_2} E(2^k(x_1^2 + x_2^2)) \\
                  &= \left(\sum_{x_1 \in A_1} E(2^k x_1^2)\right) \cdot \left(\sum_{x_2 \in A_2} E(2^k x_2^2)\right)
= \Gamma_k(\lambda_1) \cdot \Gamma_k(\lambda_2).
\end{align*}
\end{proof}
In the remainder of this section, we calculate the numbers $\Gamma_k(\cdot)$ for the basic forms.  Recall that 
$\cC = \mattwotwo{0}{1}{1}{0}$ and $\cD = \mattwotwo{2}{1}{1}{2}$.  The calculation
is summarized in the following proposition.

\begin{proposition}
\label{proposition:basiccalculation}
Define $\rho = E(1/8)$ and
$$
\varepsilon (a) =
\left\{ \begin{array}{rl}
1 \ \mod 8 & \text{ if } a \equiv 1 \ \mod 4 \\
-1 \ \mod 8 & \text{ if } a \equiv -1 \ \mod 4
\end{array}\right.
$$
Then for $j \geq 1$ and $k \geq 0$, we have
\begin{itemize}
\item $
\Gamma_k \left(\frac{a}{2^j} \right) =
\left\{ \begin{array}{ll}
2^j & \text{ if } k \geqslant j \\
0 & \text{ if } k = j-1 \\
2^{\frac{j+k+1}{2}}\rho^{\varepsilon(a)} & \text{ if $j-k \geqslant 2$ and is evenÊ} \\
2^{\frac{j+k+1}{2}}\rho^{a} & \text{ if $j-k \geqslant 2$ and is odd.}
\end{array}\right.$
\item $
 \Gamma_k \left(\frac{\cC}{2^j} \right) =
 \left\{ \begin{array}{ll}
 2^{2j} & \text{ if } k \geqslant j-1 \\
2^{j+k+1} & \text{ if } k < j-1.
\end{array}\right.$
\item $
  \Gamma_k \left( \frac{\cD}{2^{j}} \right) =
  \left\{ \begin{array}{ll}
 2^{2j} & \text{ if } k \geqslant j-1 \\
 (-1)^{j+k+1}2^{j+k+1} & \text{ if } k < j-1.
\end{array}\right.$
\end{itemize}
\end{proposition}

\begin{proof}[Proof of Proposition \ref{proposition:basiccalculation} for $\left(\frac{a}{2^j}\right)$]
We begin with the case $k \geq j$.  For these values of $k$, the number $2^k x^2$ is an
integer for all $x \in H$, so $E(2^k x^2)=1$ for all $x \in H$.  This implies that 
$$\Gamma_k\left(\frac{a}{2^j}\right) = |H| = 2^j,$$
as desired.

The next step is to prove the following formula, which reduces the remaining cases to the case $a=1$ and $k=0$:
\begin{equation}
\label{eqn:agammakreduction}
\Gamma_k \left(\frac{a}{2^j}\right) = 2^k g_a \left(\Gamma_0 \left(\frac{1}{2^{j-k}}\right)\right)\ \ \ \ \ \ \text{for $k<j$}
\end{equation}
We calculate:
$$
\Gamma_k \left(\frac{a}{2^{j}}\right) = \sum_{x \ \mod 2^j} E\left( \frac{2^k a x^2}{2^{j}}\right) = g_a \left(\sum_{x \ \mod 2^j} E\left( 
\frac{x^2}{2^{j-k}}\right)\right).
$$
Since $E(\frac{x^2}{2^{j-k}})$ depends only on $x \ \mod 2^{j-k}$, this equals
$$
g_a \left( 2^k \sum_{x \ \mod 2^{j-k}} E\left( \frac{x^2}{2^{j-k}}\right)\right) = 2^k g_a \left( \Gamma_0 \left
(\frac{1}{2^{j-k}}\right)\right),
$$
as desired.

To simplify our notation, define $\T_k = \Gamma_0 (\frac{1}{2^k})$.  We will prove that
\begin{equation}
\label{eqn:teqn}
\T_k = \left\{
\begin{array}{ll}
0                      & \text{if $k=1$}\\
2^{\frac{k+1}{2}} \rho & \text{if $k\geq2$}\\
\end{array} \right.
\end{equation}
First, however, we observe that the proposition follows from Equations (\ref{eqn:teqn}) and (\ref{eqn:agammakreduction}).
Indeed, for $k = j-1$ this is immediate.  For $k < j-1$ with $j-k$ odd, we have
$$\Gamma_k \left(\frac{a}{2^{j}}\right) = 2^k g_a(2^{\frac{j-k+1}{2}} \rho) = 2^{\frac{j+k+1}{2}} \rho^a.$$
Finally, for $k < j-1$ with $j-k$ even, using the fact that $\sqrt{2} \rho^{\varepsilon(\pm 1)} = 1 \pm i$ we have
$$\Gamma_k \left(\frac{a}{2^{j}}\right) = 2^k g_a(2^{\frac{j-k+1}{2}} \rho) = 2^{\frac{j+k}{2}} g_a(\sqrt{2} \rho)
                                        = 2^{\frac{j+k}{2}} g_a(1+i) = 2^{\frac{j+k}{2}} (1+i^a) 
                                        = 2^{\frac{j+k+1}{2}} \rho^{\varepsilon(a)},$$
as desired.

The proof of Equation (\ref{eqn:teqn}) will be by induction on $k$.  The base cases $1 \leq k \leq 3$
are calculated as follows:
\begin{align*}
\T_1 &= E(0)+E(\frac{1}{2}) = 1 - 1 = 0\\
\T_2 &= E(0)+E(\frac{1}{4})+E(\frac{4}{4})+E(\frac{9}{4}) = 1+i+1+i=2(1+i)=2 \sqrt{2} \rho\\
\T_3 &= 2(E(0)+E(\frac{1}{8})+E(\frac{4}{8})+E(\frac{1}{8}))=2(1+\rho-1+\rho)=4\rho
\end{align*}
Assume now that $k \geq 4$.  We must prove that $\T_k = 2\T_{k-2}$.  To see this, note first that 
$2^{k-1} + 1$ is a {\it square $\mod 2^k$}.  Indeed,
$${(2^{k-2}+1)}^2 = (2^{k-2})^2+ 2^{k-1}+1  \equiv 2^{k-1}+1 \ \mod 2^k$$ 
when $k \geqslant 4$. Hence $x^2 \mapsto (2^{k-1}+1)x^2$ defines a bijection of the set
of squares modulo $2^k$.  This implies that
\begin{align*}
\T_k &= \sum_{x \ \mod 2^k} E\left( \frac{x^2}{2^{k}}\right) = \sum_{x \ \mod 2^k} E\left( \frac{(2^{k-1} +1) x^2}{2^{k}}\right) \\
     &= \sum_{x \ \mod 2^k} E\left( \frac{x^2}{2^{k}}\right) E\left( \frac{x^2}{2}\right) 
      = \sum_{x \ \mod 2^k} (-1)^x E\left( \frac{x^2}{2^k}\right),
\end{align*}
and hence
\begin{align*}
2\T_k &= \sum_{x \ \mod 2^k} E\left( \frac{x^2}{2^{k}}\right) + \sum_{x \ \mod 2^k} (-1)^x E\left( \frac{x^2}{2^k}\right) \\
      &= 2 \sum_{\text{even } x \ \mod 2^k} E\left( \frac{x^2}{2^k}\right).
\end{align*}
We conclude that
\begin{align*}
\T_k &= \sum_{y \ \mod 2^{k-1}} E\left( \frac{(2y)^2}{2^k}\right) 
        = \sum_{y \ \mod 2^{k-1}} E\left( \frac{y^2}{2^{k-2}}\right)
     &= 2 \sum_{y \ \mod 2^{k-2}} E\left( \frac{y^2}{2^{k-2}}\right) = 2 \T_{k-2}.
\end{align*}
\end{proof}

\begin{proof}[Proof of Proposition \ref{proposition:basiccalculation} for $\left(\frac{\cC}{2^j}\right)$]
Here $H$ consists of all pairs $(x,y)$ with $x,y \in \Z_{2^j}$.  We begin with the case $k \geq j-1$.  
For these values of $k$, the number $2^k (x,y)^2 = 2^k \frac{2xy}{2^j}$ is an
integer for all $(x,y) \in H$, so $E(2^k (x,y)^2)=1$ for all $(x,y) \in H$.  This implies that
$$\Gamma_k\left(\frac{a}{2^j}\right) = |H| = 2^{2j},$$
as desired.

The next step is to prove the following formula:
\begin{equation}
\label{eqn:ckreduction}
\Gamma_k \left(\frac{\cC}{2^j}\right) = 2^{2k+2} \sum_{x,y \ \ \mod 2^{j-k-1}} E\left(\frac{xy}{2^{j-k-1}}\right)\ \ \ \ \ \ \text{for $k<j-1$}
\end{equation}
We calculate:
$$
\Gamma_k \left(\frac{\cC}{2^{j}}\right) = \sum_{x,y \ \mod 2^j} E\left( \frac{2^k (2xy)}{2^{j}}\right) = \sum_{x,y \ \mod 2^j} E\left(
\frac{xy}{2^{j-k-1}}\right).
$$
Since $E(\frac{xy}{2^{j-k-1}})$ depends only on $x$ and $y$ $\mod 2^{j-k-1}$, this equals
$$
2^{2k+2} \sum_{x,y \ \mod 2^{j-k-1}} E\left( \frac{xy}{2^{j-k-1}}\right),
$$
as desired.

Define
$$\U_k = \sum_{x,y \ \mod 2^{k}} E\left( \frac{xy}{2^k} \right).$$
By Formula (\ref{eqn:ckreduction}), to prove the proposition it is enough to prove that $\U_k = 2^k$ for
$k \geq 0$.  The proof of this will be by induction on $k$.  The base cases $k=0,1$ are as follows:
\begin{align*}
\U_0 &= E(0) = 1,\\
\U_1 &= E(0)+E(0)+E(0)+E(\frac{1}{2}) = 1+1+1-1 = 2.
\end{align*}
Assume now that $k \geq 2$.  We will show that $\U_k = 4 \U_{k-2}$.  First, we first fix some $y$ with $0 \leq y \leq 2^k$.
Write $y$ as $2^r y_0$ with $y_0$ odd, and suppose that $y$ is not equal to $0$ or $2^{k-1}$.  This implies that $r < k-1$ and that 
$2^{k-r+1} +1$ is odd.  Using the fact that in the following sum the numbers $x y_0$ are odd, we have
\begin{align*}
\sum_{\text{odd } x \ \mod 2^k} E \left( \frac{xy}{2^k} \right) &= \sum_{\text{odd } x} E \left( \frac{(2^{k-r-1}+1)xy}{2^k} \right)
   = \sum_{\text{odd } x} E \left( \frac{xy}{2^k} \right) E \left( \frac{2^{k-r-1}x \cdot 2^r y_0}{2^k} \right) \\
& = \sum_{\text{odd } x} E \left( \frac{xy}{2^k} \right) E \left( \frac{x y_0}{2} \right)
   = -\sum_{\text{odd } x} E \left( \frac{xy}{2^k} \right)
\end{align*}
Thus
$$\sum_{\text{odd }x} E\left( \frac{xy}{2^k} \right) = 0 \text{ for all $y \not= 0, 2^{k-1}$}.$$
In particular, since $k\geqslant 2$, the number $2^{k-1}$ is even and hence
$$\sum_{\text{odd } x, \text{odd } y}E \left( \frac{xy}{2^k} \right) = 0.$$
Also,
$$
\sum_{\text{odd } x, \text{even } y} E \left( \frac{xy}{2^k} \right)
= \sum_{\text{odd } x} \left( E \left( \frac{x\cdot 0}{2^k} \right) + E \left( \frac{x\cdot 2^{k-1}}{2^k} \right) \right)
= \sum_{\text{odd } x} (1-1) = 0;
$$
by symmetry,
$$
\sum_{\text{even $x$, odd $y$}}  E \left( \frac{xy}{2^k} \right) = 0.
$$
We have thus shown that $\U_k$ is equal to
\begin{align*}
\sum_{\text{even $x$, even $y$}}  E \left( \frac{xy}{2^k} \right)
&= \sum_{x, y \ \mod 2^{k-1}}  E \left( \frac{(2x)(2y)}{2^k} \right)
= \sum_{x, y \ \mod 2^{k-1}}  E \left( \frac{xy}{2^{k-2}} \right) \\
&= 4 \sum_{x, y \ \mod 2^{k-2}}  E \left( \frac{xy}{2^{k-2}} \right) = 4 \U_{k-2},
\end{align*}
as desired.
\end{proof}

\begin{proof}[Proof of Proposition \ref{proposition:basiccalculation} for $\left(\frac{\cD}{2^j}\right)$]
The cases $k \geq j-1$ are identical to the analogous cases for $\left(\frac{\cC}{2^j}\right)$.  For $k < j-1$,
we use Lemma \ref{lemma:gammamult} together with the isomorphism of Example \ref{example:p2problem} to conclude
that
$$\Gamma_k \left( \frac{1}{2^{j-1}} \right)  \Gamma_k \left( \frac{\cC}{2^{j}} \right)
=
\Gamma_k \left( \frac{-3}{2^{j-1}} \right)  \Gamma_k \left( \frac{\cD}{2^{j}} \right).$$
But by the previously proven cases of Proposition \ref{proposition:basiccalculation}, we have
$$
\frac{ \Gamma_k \left( \frac{1}{2^{j-1}} \right)}{ \Gamma_k \left( \frac{-3}{2^{j-1}} \right)}
=
\left\{ \begin{array}{ll}
1 & \text{ if $j-k$ is odd (since $\varepsilon (-3) \equiv 1 \ \mod 8$)} \\
\frac{\rho}{\rho^{-3}} = \rho^4 = -1 & \text{ if $j-k$ is even.}
\end{array}\right.$$
The proposition follows.
\end{proof}

\subsubsection{Reduction to the phase invariants}
\label{SSS:Reduction to the phase invariants}

By Theorem \ref{theorem:p2decomp}, Lemma \ref{lemma:gammamult}, and Proposition \ref{proposition:basiccalculation}, for any
linked abelian group $(H,\lambda)$ the invariants $\Gamma_k(\lambda)$ are either equal
to $0$ or to $(\sqrt{2})^m \rho^{\varphi}$ for some $m \geq 0$ and some $\varphi \in \Z_8$.  The point
of the following definition and theorem is that the $\sqrt{2}$-term is purely an invariant
of the abelian group $H$, and thus is unnecessary for the classification of linking forms (the
re-indexing is done to simplify the formulas in Theorem \ref{theorem:phaseadd}.

\begin{definition}
Let $(H,\lambda)$ be a linked abelian group of degree $n$.  Then for $1 \leq k \leq n$ the
$k^{\Th}$ {\em phase invariant} $\varphi_k(\lambda) \in \Z_8 \cup \{\infty\}$ of $(H,\lambda)$ is
defined to equal $\infty$ if $\Gamma_{k-1}(\lambda) = 0$ and to equal $\varphi \in \Z_8$ with 
$$\frac{\Gamma_{k-1}(\lambda)}{|\Gamma_{k-1}(\lambda)|} = \rho^{\varphi}$$
if $\Gamma_{k-1}(\lambda) \neq 0$.  The {\em phase vector} $\varphi(\lambda)$ 
of $\lambda$ is the vector $(\varphi_n(\lambda),\ldots,\varphi_1(\lambda))$.  $\|$
\end{definition}

\begin{theorem}
Fix a finite abelian group $H$ of degree $n$.  Then two linking forms $\lambda_1$ and $\lambda_2$ on $H$ are
isomorphic if and only if $\varphi_k(\lambda_1) = \varphi_k(\lambda_2)$ for all $1 \leq k \leq n$; i.e. if and
only if $\lambda_1$ and $\lambda_2$ have identical phase vectors.
\end{theorem}
\begin{proof}
Let $r_j$ be the ranks of the blocks $B_j$ of $H$, and consider a linking form $\lambda$ on $H$.
We must show that for $0 \leq k < n$ the number $\Gamma_k(\lambda)$ is determined by the $r_j$ and the
phase invariants of $\lambda$.  First, we have $\Gamma_k(\lambda) = 0$ if and only if $\varphi_k(\lambda)=0$.  We
can thus assume that $\Gamma_k(\lambda) \neq 0$, so 
$$\Gamma_k(\lambda) = (\sqrt{2})^{p_k} \rho^{\varphi_{k+1}(\lambda)}.$$
We must determine $p_k$.  By Theorem \ref{theorem:p2decomp}, we can write $(H,\lambda)$ as an orthogonal direct sum
of copies of the basic linking forms.  Fixing such a decomposition, Lemma \ref{lemma:gammamult} and 
Proposition \ref{proposition:basiccalculation} say that the orthogonal components of $(H,\lambda)$ make
the following contributions to $p_k$ :
\begin{itemize}
\item Every unary summand $\frac{a}{2^{j}}$ with $j \leqslant k$ contributes $2j$; binary summands contribute $4j$. 
Thus the block $B_j$ contributes $2jr_j$ for all $j \leqslant k$.
\item There are no unary summands when $j = k+1$, and each binary summand contributes 
$4 (k+1)$; thus $B_{k+1}$ contributes $2 (k+1) r_{k+1} = 2 j r_j$ also.
\item For $j > k+1$, each unary summand contributes $j+k+1$ and each binary summand $2(j+k+1)$, so $B_j$ contributes $r_j (j+k+1)$.
\end{itemize}
Adding these three sets of contributions gives us
$$
p_k = \sum_{j \leqslant k+1} 2 j r_j + \sum_{j > k+1} (j+k+1) r_j,
$$
so $p_k$ is a function of the ranks $r_j$ alone, as desired.
\end{proof}

The computation of the phase invariants is facilitated by the following result, whose
proof is immediate from Lemma \ref{lemma:gammamult} and Proposition \ref{proposition:basiccalculation}.

\begin{theorem}
\label{theorem:phaseadd}
The phase invariants are additive under direct sums, and have the following values on the
basic linking forms:
\begin{align*}
\varphi_k \left(\frac{a}{2^j} \right) &= \left\{ 
\begin{array}{ll}
0              & \text{ if $k > j$}\\
\infty         & \text{ if $k = j$}\\
\varepsilon(a) & \text{ if $k < j$ and $k-j$ is odd}\\
a              & \text{ if $k < j$ and $k-j$ is even}\\
\end{array}\right.\\ 
\varphi_k \left(\frac{\cC}{2^j} \right) &= 0, \text{ all } k,\\
\varphi_k \left( \frac{\cD}{2^{j}} \right) &= \left\{ 
\begin{array}{ll}
0         & \text{ if $k \geq j$} \\
4 (j+k)   & \text{ if $k < j$}
\end{array}\right.
\end{align*}
\end{theorem}

Before doing some examples, we record a definition we will need later.

\begin{definition}
\label{definition:evenodd}
A linking form $\lambda$ on a free $\Z_{2^n}$ module is {\em even} if $\lambda(x,x) = \frac{2k}{2^n}$ for
some integer $k$.  Otherwise, is {\em odd}.  We remark that $\lambda$ is even if and only if any
orthogonal decomposition of it into basic forms has no unary summands.  $\|$
\end{definition}

\begin{example}
Let us test the method on the linkings of Example~\ref{example:p2problem}.
\begin{itemize}
\item $\left(\frac{\cC}{2^n} \right) \oplus  \left(\frac{1}{2^{n-1}} \right)$ gives $\varphi = (0, \infty, 1, 1, \dots)$
since $a = \varepsilon (a) = 1$ and $\varphi  \left(\frac{\cC}{2^n} \right) = 0$.
\item $\left(\frac{\cD}{2^n} \right) \oplus  \left(\frac{-3}{2^{n-1}} \right)$ gives 
$$\varphi = (0, 4, 0, 4, \dots) + (0, \infty, 1, -3, \dots) = (0, \infty, 1, 1, \dots),$$
as desired.  $\|$
\end{itemize} 
\end{example}

\begin{example}
$\displaystyle \frac{1}{2^n} \oplus \frac{3}{2^{n-1}} \not\simeq \frac{3}{2^n} \oplus \frac{1}{2^{n-1}}$. The reason is: $\varphi$ of the left hand side is
$\varphi = (\infty, 1, 1, 1, \dots) + (0, \infty, -1, 3, \dots) = (\infty, \infty, 0, 4, \dots)$, and 
$\varphi$ of the right hand side is
$\varphi = (\infty, -1, 3, -1, \dots) + (0, \infty, 1, 1, \dots) = (\infty, \infty, 4, 0, \dots).$ $\|$
\end{example}

\begin{example}
\begin{eqnarray*} 
\varphi \left(  \frac{1}{2^n} \oplus \frac{3}{2^{n-1}} \oplus \frac{5}{2^{n-2}} \right) 
& = & (\infty, \infty, 0, 4, 0,
\dots) + (0, 0, \infty, 1, 5, \dots) = (\infty, \infty, \infty, 5, 5, \dots), \\
 \varphi \left(  \frac{3}{2^n} \oplus \frac{5}{2^{n-1}} \oplus \frac{1}{2^{n-2}} \right) 
 & = & (\infty, -1, 3, -1, 3,  \dots) + (0, \infty, 1, 5, 1, \dots) + (0, 0, \infty, 1, 1, \dots) \\
& = & (\infty, \infty, \infty, 5, 5, \dots)\\
\varphi \left(  \frac{5}{2^n} \oplus \frac{1}{2^{n-1}} \oplus \frac{3}{2^{n-2}} \right) 
& = & (\infty, \infty, \infty, 1, 1, \dots)
\end{eqnarray*}
Therefore  the first two forms are isomorphic, and the third form is different from the first two. $\|$
\end{example}

\begin{example}
If $B_j$ is odd for all $j$ from $1$ to $n$, then $\varphi (\lambda) = (\infty, \infty, \dots, \infty)$; in particular, the class of $\lambda$ does not depend at all on the particular form of the individual blocks $B_j$. $\|$
\end{example}

\subsection{Reidemeister's invariants}
\label{SS:Reidemeister's invariants}   As we have seen, the Seifert-Burger viewpoint gives us a family of topological invariants of a 3-manifold that are associated to $H = H_1(W,\mathbb Z)$, yet are not determined by $H$.  Moreover, their determination depends crucially on whether there is, or is not, 2-torsion in  $H$.  We have learned how to compute them from a symplectic Heegaard splitting. 

 In 1933, the year that \cite{Sei} was published, a paper by Reidemeister \cite{R2}  also appeared. Moreover, there are remarks at the end both \cite{Sei} and \cite{R2} pointing to the work of the other. In particular, Reidemeister notes in his paper that Seifert's invariants are more general than his, as they must be because a quick scan of Reidemeister's paper \cite{R2} does not reveal any dependence of his results on whether there is 2-torsion.   We describe Reidemeister's invariants briefly.

Referring to Theorem~\ref{T:fundamental theorem for f.g. abelian groups} and using the notation adopted there, Reidemeister defines the integers $\tau_{i,j} = \tau_j/\tau_i$, where $1\leq i<j\leq t$.  Let $\cQ^{(2)} = (q_{ij})$ be the $t\times t$ matrix   in (\ref{E:R2,P2,S2,Q2}), where we defined the partial normal form of Theorem~\ref{T:partial normal form}.  Let $p_{im}$ be a non-trivial prime factor of the greatest common divisor of $(\tau_{1,2},\tau_{2,3},\dots,\tau_{i,i+1})$.   His invariants are a set of  symbols  which he calls $\varepsilon_{im}$, defined as follows:
\begin{eqnarray} \label{Reidemeister's invariants}
\nonumber
\varepsilon_{im} & = & 0  \quad \quad \quad  {\rm if \ \ } p_{im} {\rm \ \ divides \ \ } q_{ii}, \\
& = & q_{ii}/p_{im} \ \ \ {\rm if \ \ } p_{im} \ \ {\rm does \ \ not \ \ divide \ \ } q_{ii}.   
 \end{eqnarray}
Thus $\varepsilon_{im}$ is defined for every non-trivial prime divisor of the greatest common divisor of $(\tau_{1,2},\tau_{2,3},\dots,\tau_{i,i+1})$, and for every $i=1,\dots,t-1$. 
 He proves that his symbols are well-defined, independent of the choice of the representative 
 $\cH^{(2)}$ within the double coset of $\cH^{(2)}$ in $\Gamma_t$, by proving that remain unaltered  under the changes which we  described in $\S$\ref{SS:uniqueness questions}.  The fact that they are undefined when the greatest common divisor of $(\tau_{1,2},\tau_{2,3},\dots,\tau_{i,i+1})$ is equal to 1 show that they do not change under stabilization. 
 
 \begin{remark}   
We remark that Reidemeister's invariants are invariants of both the Heegaard splitting and of the stabilized Heegaard splitting.  
Therefore, if one happened to be working with a manifold which admitted two inequivalent Heegaard splittings, it would turn out 
that their associated Reidemeister symbols would coincide.  This illustrates the very subtle nature of the Heegaard splitting invariants 
that are, ipso facto, encoded in the higher order representations of the mapping class group in the Johnson-Morita filtration.  
Any such Heegaard splitting invariant is either a topological invariant (as is the case for Reidemeister's invariant), or 
an invariant which vanishes after sufficiently many stabilizations.   We will uncover an example of the latter type in the next section.
 \end{remark}  

\newpage
  \section{The classification problem for minimal (unstabilized) symplectic Heegaard splittings.}
  \label{S:the classification problem for minimal symplectic Heegaard splittings}

Referring the reader back to Remark~\ref{R:inequivalent Heegaard splittings}, it is clear that knowledge of a complete set of  invariants of minimal symplectic Heegaard splittings, and the ability to compute them,  are of interest in their own right.   This was our motivation when we posed Problems 4, 5 and 6 in $\S$\ref{SS:6 problems}.  In this section we will solve these problems.

Given two minimal Heegaard pairs with isomorphic linked quotient groups $H_i$ and canonical volumes $\pm \theta_i$, 
Theorem \ref{T:the volume again} tells us that the pairs are isomorphic if and only if there is a volume preserving 
linking isomorphism $H_1 \to H_2$. By hypothesis there is a linking isomorphism $h$, but it may not be volume 
preserving --- $\det h$ may not equal $\pm 1$.  Suppose $f: H_1 \to H_1$ is a linking automorphism, that is, an 
{\it isometry}, of $H_1$.  Then $hf$ is still a linking isomorphism, and 
$\det (hf) = \det h \cdot \det f$.  Thus if $\det h \not= \pm1$, we may hope to change it to 
$\pm 1$ by composing it with some isometry of $H_1$. This will be our approach, and it will give us a complete set of 
invariants for {\it minimal} Heegaard pairs.

\subsection{Statement of Results. Solutions to Problems 4, 5, 6.}

Let $(V; B, \Bbar)$ be a minimal Heegaard pair with quotient $H_0$.  Choose any dual complement $A$ of $B$ and let
$\pi : A \rightarrow H_0$ be the projection.  Let $H < H_0$ be the torsion subgroup with associated
linking form $\lambda$, and set $\F = \pi^{-1}(H)$.  We thus have a minimal presentation $\pi : \F \rightarrow H$.
If $e_i$ ($i=1, \dots, r$) is a basis for $\F$, put $x_i = \pi (e_i)$. 
We define a {\it linking matrix} for $\lambda$ by choosing rational numbers
$\lambda_{ij}$ which are congruent $\mod ~ 1$ to $\lambda (x_i, x_j)$ for each $i, j$, subject to the symmetry condition
$\lambda_{ij} = \lambda_{ji}$; we call such a choice a {\it lifting} of $\lambda (x_i, x_j)$.

Though the linking matrix for $\lambda$ depends on choices, we can extract an invariant from it.  The
first step is the following theorem, which will be proven in \S \ref{section:invariance1}.  Let $\tau$ be the smallest
elementary divisor of $H$.

\begin{theorem}
\label{theorem:unstableinvariants1}
The number $| H | \det (\lambda_{ij})$ is an integer, and its reduction modulo $\tau$ is a unit in $\Z_{\tau}$ which 
depends only on the isomorphism class of $(H,\lambda)$ and the isomorphism class of the presentation $\pi : F \rightarrow H$.
\end{theorem}

This theorem was first proven (by different methods) in \cite{Birman1975}.  In some cases, we can do better.
We will need the following definition.

\begin{definition}
The linking on $H$ is {\em even} if for all $x \in H$ with $\tau x = 0$, we have $x^2 \in \leg{2}{\tau}$ (here we have 
abbreviated $x \cdot x$ to $x^2$ and $\leg{2}{\tau}$ is the subgroup of $\Q/\Z$ generated by $\frac{2}{\tau}$).
Otherwise the linking is {\em odd}.
\end{definition}

\begin{remark}
See Lemma \ref{lemma:8.2} below to relate this to the definition of an even linking form on a 2-group
defined in Definition \ref{definition:evenodd}.
\end{remark}

We will prove the following refinement of Theorem \ref{theorem:unstableinvariants1} in
\S \ref{section:invariance2}.

\begin{theorem}
\label{theorem:unstableinvariants2}
Let $\overline{\tau} = \tau$ if $\lambda$ is odd and $2 \tau$ if $\lambda$ is even.  Then 
the reduction modulo $\overline{\tau}$ of $| H | \det (\lambda_{ij})$ depends only
on the isomorphism class of $(H,\lambda)$ and the isomorphism class of the presentation $\pi : F \rightarrow H$.  
\end{theorem}

At the end of \S \ref{section:invariance2} we give an example which shows that we have indeed found
a stronger invariant than the one that was given in \cite{Birman1975}.

Corollary \ref{C:complete invariant of equivalence of minimal presentations} and Lemma \ref{L:the volume again} say 
that the reduction modulo $\overline{\tau}$ of $| H | \det (\lambda_{ij})$ is actually a well-defined
invariant of $(V; B, \Bbar)$, which we will denote by $\det(V; B, \Bbar)$.  
Our next theorem say that it is a complete invariant of minimal Heegaard pairs, solving Problem 4 of $\S$\ref{SS:6 problems}.

\begin{theorem} \label{T:unstabilizedHeegaardpairs}
Let $(V_i; B_i, \Bbar_i)$ ($i = 1, 2$) be minimal Heegaard pairs with linked quotients $(H_i, \lambda_i)$.
Then the pairs are isomorphic if and only if the linked quotients are isomorphic and 
$\det(V_1; B_1, \Bbar_1) = \det(V_2; B_2, \Bbar_2)$.
\end{theorem}

Theorem~\ref{T:unstabilizedHeegaardpairs} will be proven in \S \ref{section:completeness}.  With Theorem~\ref{T:unstabilizedHeegaardpairs} in hand, we will be able to count the number of isomorphism classes of minimal Heegaard pairs with linked quotients $(H,\lambda)$, solving Problem 5 of $\S$\ref{SS:6 problems}.
To make sense of that result, we will need the following lemma.

\begin{lemma}
\label{lemma:squaredef}
Consider an integer $n \in \Z_{\tau}$.  Then $n^2$ is well defined mod $\overline{\tau}$.
\end{lemma}
\begin{proof}
We may assume that our linking is even, so $\overline{\tau} = 2 \tau$.  Then for $a, b \in \Z$ we
have
$$(a + b \tau)^2 = a^2 + 2 a b \tau + b^2 \tau^2,$$
which equals $a^2$ modulo $2 \tau$ since $2 \tau$ divides $\tau^2$.  Hence knowledge of $a$ modulo $\tau$
sufficies to determine $a^2$ modulo $\overline{\tau}$, as desired.
\end{proof}

Denote the group of units in $\Z_{\tau}$ by $\units$.  In light of Lemma \ref{lemma:squaredef}, it makes
sense to define
$$\sqrt{1} = \{\text{$x \in \units$ $|$ $x^2 = 1$ modulo $\overline{\tau}$}\}.$$
Our result is the following; it will be proven in \S \ref{section:completeness} as a byproduct
of the proof of Theorem \ref{T:unstabilizedHeegaardpairs}.

\begin{theorem}
\label{theorem:isomorphismcount}
The number of isomorphism classes of distinct minimal Heegaard pairs with linked quotients $(H,\lambda)$ is
$\frac{| \units |}{| \sqrt{1} |}$.
\end{theorem}

We close this chapter in \S \ref{SS:remarks on problem 6} with a discussion of 
how our techniques give normal forms for symplectic gluing matrices, in certain situations.  This was the problem that was posed in Problem 6 of $\S$\ref{SS:6 problems}.

\subsection{Proof of Invariance 1}
\label{section:invariance1}

This section contains the proof of Theorem \ref{theorem:unstableinvariants1}.  
We will need several definitions.

Let $\F$ be a free abelian group and $R \subset \F$ be a subgroup of equal rank.  If $\varphi, \psi$ are
orientations of $\F, R$, then the inclusion map
of $R$ into $\F$ has an integral determinant, and this determinant's absolute value is well known to be $(\F : R) = | \F/RÊ|$ (for
example, choose a basis for $\F$ as in Proposition \ref{P:structure of abelian groups}).  Thus, given either $\varphi$ or $\psi$
there is a unique choice of the other so that this determinant is {\it positive}.  Orientations of $\F, R$ so chosen will be called
{\it compatible}.  Note that a change in the sign of one orientation necessitates a change in the other also to maintain
compatibility. An orientation $\varphi$ of $\F$ also induces a canonical orientation $\varphi^*$ on the dual group
$\F^* = \Hom (\F, \Z)$: choose any basis $\{ e_i \}$ of $\F$ with $e_1 \wedge \cdots \wedge e_r = \varphi$, then use the
dual basis of $\F^*$ to define $\varphi^*$.

If $H$ is any finite group, its {\it character group} is the additive group $H^* = \Hom (H, \Q/\Z)$.  It is well known that
$H^*$ is isomorphic to $H$, but not canonically so. This mirrors the relationship between a free (finitely generated) abelian
group $\F$ and its dual (in the following, the $^*$ on a finite group indicates its character group, but on a free group
indicates its dual). Also as in the free case, $H^{**}$ is {\it canonically} isomorphic to $H$.

Suppose that $\F \xrightarrow{\pi} H$ is a presentation of $H$ with kernel $R$; we construct from it a canonical presentation of
$H^*$ which we call the {\it dual} presentation.  The group $\F^*$ is a subgroup of $\F^* \otimes \Q = \Hom (\F, \Q)$, namely,
all maps $f: \F \to \Q$ such that $f(\F) \subset \Z$.  The fact that $H$ is finite and hence $\rk R = \rk \F$ implies that
$R^*$ is precisely the subgroup of maps $f \in \F^* \otimes \Q$ such that $f(R) \subset \Z$.  Note that
$R^* \supset \F^*$ in $\F^* \otimes \Q$.  If $f \in R^*$, then $f$ is a map of $\F$ into $\Q$ taking $R$ into $\Z$ and so
induces a map of $\F/R = H$ to $\Q/\Z$.  This element of $H^*$ we denote by $\pi^*(f)$; we have then that $\pi^*$ is a map
$R^* \to H^*$.  If $v \in H^*$, \ie $v: \F/R \to \Q/\Z$, we can lift $v$ to a map $f : \F \to \Q$ since $\F$ is free,
and clearly $f(R) \subset \Z$, so $f \in R^*$ and $\pi^*(f) = v$.  This shows that $R^* \xrightarrow{\pi^*} H^*$ is
a presentation of $H^*$.  The kernel of $\pi^*$ consists of all $f$ such that $f(\F) \subset \Z$, that is,
precisely $\F^*$.  Note thus that the index $(R^*: \F^*) = | H^* | = | H | = (\F:R)$.

Just as a choice of symmetric isomorphism $\F \to \F^*$ is the same as an ``inner product" on $\F$, so the choice of a
symmetric isomorphism $H \xrightarrow{\lambda} H^*$ is the same as a {\it linking} on $H$: if we write $x \cdot y$
for the linking of $x, y$, then $x \cdot y$ is defined to be $\lambda (x) (y) \in \Q/\Z$ and conversely.  The fact that
$\lambda$ is an isomorphism corresponds to the non-singularity of the linking.

If $H$ is a linked group and $\F \xrightarrow{\pi} H$ is a presentation of $H$, consider the diagram
$$
\begin{CD}
0 @>>> R @>>> \F @>>> H @>>> 0 \\
&& @VV{K}V @VV{L}V @VV{\lambda}V \\
0 @>>> \F^* @>>> R^* @>>> H^* @>>> 0.
\end{CD}
$$
The fact that $\F$ is free implies the existence of a map $L$ making the right square commute, and $L$ induces $K$ on $R$.  We
call $L$ a {\it lifting} of $\lambda$, and it is well defined up to the addition of a map $X : \F \to \F^*$.  If now we choose
an orientation $\varphi$ of $\F$, it induces $\varphi^*$ on $\F^*$ and compatible orientations $\psi, \psi^*$ on $R, R^*$; it is
easy to see that $\psi, \psi^*$ are also dual orientations.  Furthermore, if $\F \to H$ is {\it minimal}, then so is
$R^* \to H^*$ and we get induced orientations $\theta, \theta^*$ on $H, H^*$.

A change in the sign of $\varphi$ uniformly changes the sign of all the other orientations.  Thus the determinants
$\det K, \det L \in \Z$ and $\det \lambda \in \Z_\tau$ are all well defined and independent of the orientations.
Furthermore, $\det \lambda$ depends only on $\lambda$ and the presentation $\pi$.
The connection between these determinants and that of Theorem
\ref{theorem:unstableinvariants1} is given by the following lemma.

\begin{lemma} \label{lemma:8.1b}
\begin{enumerate}[a)]
\item $\det L \equiv \det \lambda \ \mod \tau$
\item $\det K = \det L$
\item If $(\lambda_{ij})$ is a linking matrix as in Theorem \ref{theorem:unstableinvariants1}, then it determines a lifting $L: \F \to R^*$ and $| H | \det (\lambda_{ij}) = \det K = \det L$.
\end{enumerate}
\end{lemma}

\begin{proof}
{\it a)} is proved in Lemma \ref{L:orientation-2}; {\it b)} follows from the commutativity of the left square and the fact
that the determinant of both (compatibly oriented) $R \to \F$, $\F^* \to R^*$ is $|H|$.  It remains to prove {\it c)}.  Now
the linking matrix $(\lambda_{ij})$ is clearly just the matrix of a map $L : \F \to \F^* \otimes \Q$ in terms of the basis
$\{ e_i \}$ of $\F$ used to define $(\lambda_{ij})$ and its dual basis in $\F^* \otimes \Q$, namely,
$L(e_i) = \sum_j \lambda_{ij} e_j^*$.  Put $x_i = \pi (e_i)$ and denote the linking in $H$ by the inner product dot; if then
$s = \sum \alpha_i e_i$ is in $R$, we find
\begin{eqnarray*}
L (e_i)(s)  & = & \sum_j \lambda_{ij} e_j^* (s) 
\ = \ \sum_j \lambda_{ij} \alpha_j 
\ \equiv_{\mod ~ 1} \ \sum (x_i \cdot x_j) \alpha_j \\
& = & x_i \cdot \left( \sum_j \alpha_j x_j \right)
\ = \ x_i \cdot \pi (s)
\ \equiv \ 0 \ \mod ~ 1.
\end{eqnarray*}
In other words, $L(e_i)$ takes $R$ into $\Z$ for all $i$, that is, $L(e_i) \in R^*$ for all $i$, which means that $L$ is
actually a map from $\F$ to $R^*$. By its very definition it is a lifting of $\lambda$.  Let now $s_1, \dots, s_r$ be a
basis of $R$ compatible with $e_1, \dots, e_r$ of $\F$, and let $s_i = \sum_j A_{ij} e_j$; thus $\det(A_{ij}) = |H|$.  We then
find that
$$
K (s_i) = L(s_i) = \sum_{j,k} A_{ij} \lambda_{jk} e_k^*,
$$
and since $K(s_i)$ is in $\F^*$, the matrix $A \cdot (\lambda_{ij})$ is integral and its determinant is
$\det K =$ $ \det A \det (\lambda_{ij}) $ $= |H| \det (\lambda_{ij})$.
\end{proof}

This lemma proves that $| H | \det (\lambda_{ij})$ is an integer whose mod $\tau$ reduction only
depends on the isomorphism class of $\pi : F \rightarrow H$ and the linking.  The fact that it is
a unit in $\Z_{\tau}$ follows from the fact that $\lambda : H \rightarrow H^*$ is an isomorphism.

\subsection{Proof of Invariance 2}
\label{section:invariance2}

We can assume that the linking form is even.  Let the notation be as in the previous section.  
The first step is to prove that the lifts $L$ which come from
the symmetric linking matrices are {\em symmetric} in the sense that the induced map $L^* : R \to F^*$ is the
map $K$.  Let the $e_i$, the $s_i$, and the matrix $A$ be as in the proof of Lemma \ref{lemma:8.1b}.  
We will determine the matrix of $L$ in terms of the bases
$\{ e_i\}$ of $\F$ and $\{ s_i^* \}$ of $R^*$.  Since $s_i = \sum_j A_{ij} e_j$, we have $e_j^* = \sum_i s_i^* A_{ij}$ so
$$
L(e_i) = \sum_j \lambda_{ij} e_j^* = \sum_{j, k} \lambda_{ij} A_{kj} s_k^*.
$$
Thus the matrix of $L$ is $\lambda A\t$, and in the dual bases $s_i$ and $e_i^*$ the operator $L^*$ has matrix $A \lambda\t$.
Since $(\lambda_{ij})$ was chosen to be symmetric, we have finally 
$L^* = A \lambda$, which is the matrix of $K$.  Note that two
symmetric liftings of $\lambda$ differ by a {\it symmetric} map $\F \to \F^*$.

Our goal is to prove that modulo $2 \tau$ the number $\det L$ is independent of the choice of a symmetric lifting
of $\lambda$.  By Theorem \ref{theorem:unstableinvariants1} and the fact that $\tau$ is even,
$(\det L, 2\tau) = 1$ and hence the matrix of $L$ has a $\mod 2\tau$ inverse, that is, there is an integral matrix
$L^{-1}$ such that $L L^{-1} \equiv \mathcal I \ \mod 2 \tau$.  Any other symmetric lifting of $L$ is of the form $L + X$
where $X$ is a symmetric map $\F \to \F^*$.  But note that $R \subset \tau \F$ and so $\F^* \subset \tau R^*$;
hence $X = \tau Y$ for some map $Y: \F \to R^*$.  Modulo $2 \tau$ we then have
$$\det (L + X) \equiv \det L \cdot \det (\mathcal I + L^{-1} X) = \det L \cdot \det (\mathcal I + \tau L^{-1} Y).$$

\begin{lemma} \label{lemma:8.4}
If $A$ is any square matrix and $\tau > 1$, then $\det (\cal I + \tau A) \equiv 1 + \tau \tr (A) \ \mod \tau^2$, where $\tr(A)$ is the trace of $A$.
\end{lemma}
\begin{proof}
In the expansion of $\det (1 + \tau A)$ only those monomials involving at most one off-diagonal factor are non-zero
$\mod \tau^2$.  A single off-diagonal factor cannot occur, however, in any monomial, and so
$\det (1+ \tau A) \equiv \prod_i (1 + \tau A_{ii} ) \ \mod \tau^2$.  The product is clearly
equal to $1 +\tau (\sum A_{ii}) \ \mod \tau^2$.
\end{proof}
This lemma shows that
$$\det (L + X) \equiv \det L + \tau \det L \cdot \tr (L^{-1} Y) \ \mod \tau^2.$$
But since $\det L \equiv 1 \ \mod 2$, it follows that $\tau \det L \equiv \tau \ \mod 2\tau$.  Hence since $2 \tau | \tau^2$,
we have
$$\det (L + X) \equiv \det L + \tau \tr (L^{-1} Y) \ \mod 2 \tau.$$
Thus to prove the desired result it suffices to show that
$\tr (L^{-1}Y) \equiv 0 \ \mod 2$. We must now take a closer look at the matrices $L$ and $X$.

We choose the basis of $\F$ as in Proposition \ref{P:structure of abelian groups}, so that $s_i = m_i e_i$ ($i = 1, \dots, r$) with
$m_1 = \tau$ and $m_i | m_{i+1}$.  Let $2^n$ be the highest power of 2 dividing $\tau$ (we notate this in the future by
$2^n \parallel \tau$) and suppose that the same is true for $m_1$ through $m_b$; that is, $\frac{m_i}{2^n}$ is odd for
$1 \leqslant i \leqslant b$ and even for $i>b$.  Let $(\lambda_{ij})$ be a (symmetric) linking matrix lifting $x_i \cdot x_j$ as before.
The matrix $A$ describing the basis $\{ s_i \}$ in terms of $\{ e_i\}$ is now the diagonal matrix $\diag(m_i)$, and so the
matrix of $L$ is of the form $(L_{ij}) = (\lambda_{ij}) A \t = (\lambda_{ij} m_j)$.  We claim that $L_{ij}$ is even for $i \leqslant b$
and $j > b$.  Indeed, since $x_i = \pi (e_i)$ has order $m_i < m_j$ we must have $\lambda_{ij} = \frac{N}{m_i}$ for some integer $N$,
and thus $L_{ij} = \frac{N}{m_i} m_j = N \frac{m_j}{m_i}$.  But $\frac{m_j}{m_i}$ is even whenever $i \leqslant b$ and $j>b$.  If we divide
the coordinate indices into two blocks with $1 \leqslant i \leqslant b$ in the first block and $i > b$ in the second, then
$\mod 2$, the matrix $L$ takes the form $\mattwotwo{B}{0}{C}{D}$.  The fact that $\det L \equiv 1 \ \mod 2$ implies that
$\det B \equiv \det D \equiv 1 \ \mod 2$.

\begin{lemma} \label{lemma:8.5a}
$B$ is symmetric $\mod 2$ and has zero diagonal $\mod 2$.
\end{lemma}
\begin{proof}
If $i < j$, then $B_{ij} = N \frac{m_j}{m_i}$.  But $2^n \parallel m_i$ and $2^n \parallel m_j$, so $\frac{m_j}{m_i}$ is odd and
$B_{ij} \equiv N \ \mod 2$.  On the other hand, $B_{ji} = \lambda_{ji} m_i = \frac{N}{m_i} m_i = N$.  Thus $B$ is symmetric
$\mod 2$.  Its diagonal term $B_{ii}$ is $\lambda_{ii} m_i$ where $\lambda_{ii}^2 \equiv x_i^2 \ \mod ~ 1$.  Now $x_i$ is of order
$m_i$ so $\frac{m_i}{\tau} x_i$ is of order $\tau$, and $\frac{m_i}{\tau}$ is {\it odd}.  Thus
${\leg{m_i}{\tau}}^2 \lambda_{ii} \equiv {\left( \frac{m_i}{\tau}x_i \right)}^2 \ \mod ~ 1$, and the latter is in
$\leg{2}{\tau}$ by the assumption that $\lambda$ is even, 
so we have $\leg{m_i}{\tau}^2 \lambda_{ii} = \frac{2N}{\tau}$ for some integer $N$.
Multiplying by $\tau$ we get $\frac{m_i}{\tau} (m_i \lambda_{ii}) = \frac{m_i}{\tau} B_{ii} \equiv 0 \ \mod 2$, which
by the oddness of $\frac{m_i}{\tau}$ implies that $B_{ii} \equiv 0 \ \mod 2$.
\end{proof}

\begin{lemma} \label{lemma:8.5b}
Let $X : \F \to \F^* \subset R^*$ be symmetric. Then its matrix, written in the bases $e_i, s_i^*$, is congruent $\mod 2 \tau$
to a matrix of the block form $\tau \mattwotwo{U}{0}{V}{0}$, where $U$ is symmetric.
\end{lemma}
\begin{proof}
In the bases $e_i, e_i^*$ the matrix of $X$ is $\mattwotwo{U}{V\t}{V}{W}$ where $U$ and $W$ are symmetric, but in the bases
$e_i, s_i^*$ it is $\mattwotwo{U}{V}{V\t}{W} \diag (m_i)$.  Since $\frac{m_i}{\tau}$ is odd for $i \leqslant b$ and even for $i> b$,
the block form of $\diag (m_i)$ is congruent $\mod 2 \tau$ to $\mattwotwo{\tau \mathcal I_b}{0}{0}{0}$ where $\mathcal I_b$ is the identity matrix.
Hence $X \equiv_{\mod 2 \tau} \mattwotwo{\tau U}{0}{\tau V}{0} =\tau \mattwotwo{U}{0}{V}{0}$.
\end{proof}

Recall the map $Y = \frac{1}{\tau} X$; by the above it is congruent $\mod 2$ to $\mattwotwo{U}{0}{V}{0}$.
We now calculate
$$
L^{-1} Y \equiv_{\mod 2} {\bmattwotwo{B}{0}{C}{D}}^{-1} \bmattwotwo{U}{0}{V}{0}
\equiv \bmattwotwo{B^{-1}}{0}{C'}{D^{-1}} \bmattwotwo{U}{0}{V}{0}
\equiv \bmattwotwo{B^{-1} U}{0}{C''}{0},
$$
where the precise calculation of the matrices $C', C''$ is unimportant to us.  We are interested only in
$\tr (L^{-1} Y) \equiv \tr (B^{-1} U) \ \mod 2$, where $B, U$ are symmetric and $B$ has zero diagonal.

\begin{lemma} \label{lemma:8.6}
Let $B$ be a nonsingular symmetric matrix over $\Z_2$ with zero diagonal; then its inverse has the same properties.
\end{lemma}
\begin{proof}
We may lift $B$ to an {\it integral} matrix $\tilde{B}$ which is {\it antisymmetric}, that is, $\tilde{B}\t = - \tilde{B}$.
Since $\det \tilde{B} \equiv 1 \ \mod 2$, the matrix $\tilde{B}^{-1}$ is rational and antisymmetric and
$\tilde{C} = \tilde{B}^{-1} \cdot \det \tilde{B}$ is {\it integral} and antisymmetric.  Reduced $\mod 2$, it is also an
inverse of $B$, since $\tilde{B} \cdot \tilde{C} = \det \tilde{B} \mathcal I \equiv \mathcal I \ \mod 2$. 
This proves the lemma.
\end{proof}

We can now see that the desired result follows from the above results and the following:

\begin{lemma} \label{lemma:8.7}
Let $C, U$ be symmetric matrices over $\Z_2$ such that $C$ has zero diagonal; then $\tr (CU) = 0$.
\end{lemma}
\begin{proof}
$\tr (CU) = \sum_{i, j} C_{ij} U_{ji}$. We split this sum into three parts: $\sum_{i<j} + \sum_{i>j} + \sum_{i=j}$. Now
\begin{eqnarray*}
\sum_{i<j} C_{ij} U_{ji} & = & \sum_{i<j} C_{ji} U_{ji} \quad \text{ by symmetry of $C$ and $U$} \\
& = & \sum_{j<i} C_{ij} U_{ji} \quad \text{ interchanging $i, j$;}
\end{eqnarray*}
thus $\sum_{i<j} + \sum_{i>j} $ cancel over $\mathbb Z_2$.   But the last summand is $\sum_i C_{ii} U_{ii} = 0$ since $C_{ii} = 0$ for all $i$.
\end{proof}

We close this section with an example showing that we have indeed found a stronger invariant.

\begin{example}
\label{Example:inequivalent unstabilized splittings}
Consider the matrices
$$
\cU = \bmattwotwo{\mattwotwo{0}{-15}{-15}{0}}{\mattwotwo{8}{0}{0}{8}}{\mattwotwo{-2}{0}{0}{-2}}{\mattwotwo{0}{1}{1}{0}}
\quad \text{ and } \quad
\cV = \bmattwotwo{\mattwotwo{0}{-5}{-5}{0}}{\mattwotwo{8}{0}{0}{8}}{\mattwotwo{-2}{0}{0}{-2}}{\mattwotwo{0}{3}{3}{0}}.
$$
They are easily seen to be symplectic and thus define Heegaard pairs as discussed in
\S \ref{S:symplectic spaces, Heegaard pairs and symplectic Heegaard splittings}, namely $(X_2 ; \F_2, \cU(\F_2) \text{ or } \cV (\F_2))$. The quotient groups are $\Z_8 \oplus \Z_8$ in both cases, with respective linking matrices $\frac{1}{8} \mattwotwo{0}{1}{1}{0}$ and $\frac{1}{8} \mattwotwo{0}{3}{3}{0}$, which are even, so $\bar{\tau} = 16$. Multiplication by 3 in $\Z_8^2$ gives an isomorphism between the two linkings, so the pairs are stably isomorphic; furthermore, their determinants are both $\equiv -1 \ \mod \tau (= 8)$. But these pairs are {\it not} isomorphic, since their respective determinants $\mod \bar{\tau}$ are $-1 \ \mod ~ 16$ and $-9 \ \mod ~ 16$. 
$\|$
\end{example}

\subsection{Proof of Completeness and a Count.}
\label{section:completeness}

We now prove Theorem \ref{T:unstabilizedHeegaardpairs}, which says that our mod $\overline{\tau}$ determinantal invariant is a
complete invariant of minimal Heegaard pairs.  A byproduct of our proof will be a proof of Theorem
\ref{theorem:isomorphismcount}.  As we observed in \S \ref{S:Heegaard pairs and their linked abelian groups},
we can assume that the the Heegaard pairs in question have finite quotients.  Now, we have shown that our
invariant is really an invariant of the linked quotient together with its induced volume, and it
is easy to see that all such volumes occur.  For a linked finite abelian group $(H,\lambda)$ with a
volume $\theta$, denote this determinantal invariant by $\det(\lambda, \theta)$.  Our first
order of business is determine which volumes on $(H,\lambda)$ have the same determinant, so fix 
a linked finite abelian group $(H,\lambda)$
together with a minimal presentation $F \rightarrow H$, and let $\tau$ and $\overline{\tau}$ be defined
as before.  We begin with a lemma.

\begin{lemma} \label{lemma:8.8}
In the commutative diagram
$$
\begin{CD}
\F @>\pi>> H \\
@VfVV @VVhV \\
\F' @>\pi'>> H'
\end{CD}
$$
let $\pi, \pi'$ be minimal presentations of the linked groups $(H, \lambda)$ and $(H', \lambda')$, and let $h$ 
be a linking isomorphism (we do not assume that $f$ is an isomorphism). 
Then if $\det h$ is measured with respect to the induced volumes on $H, H'$, we have
$\det (\lambda, \pi) = (\det h)^2 \det (\lambda', \pi') \ \mod \bar{\tau}$.
\end{lemma}
\begin{proof}
Lemma \ref{lemma:squaredef} shows that the statement is meaningful. 
Let now $\{ e_i \}$, $\{ e_i'\}$ be bases of $\F, \F'$ respectively, and $(\lambda'_{ij})$ be a
linking matrix for $H'$ in the basis $\{ e_i' \}$.  If $f$ has matrix $A$, then the fact that $h$ is a linking isomorphism
implies easily that $A \lambda' A\t$ is a linking matrix for $H$ in the basis $\{ e_i \}$.  Hence
$$
\det (\lambda, \pi) \equiv_{\mod \bar{\tau}} |ÊH | \det (\lambda_{ij}) = | HÊ| \det A^2 \det (\lambda_{ij}') \equiv_{\mod \bar{\tau}} (\det f)^2 \det (\lambda', \pi').
$$
But by lemma \ref{L:orientation-2} we have $\det f \equiv \pm \det h \ \mod \tau$, so
$(\det h)^2 \equiv (\det f)^2 \ \mod \bar{\tau}$.
\end{proof}

\begin{corollary} \label{corollary:8.9}
Let $h$ be an isometry of $H$; then $(\det h)^2 \equiv 1 \ \mod \bar{\tau}$.
\end{corollary}

This corollary restricts the determinant of an isometry to lie in $\sqrt{1}$.
The following is an immediate corollary of Lemma \ref{lemma:8.8}.

\begin{lemma} \label{lemma:8.12}
For any volume $\theta$ on $(H,\lambda)$, we have 
$\det (\lambda, m \theta) \equiv m^2 \det (\lambda, \theta) \ \mod \bar{\tau}$ 
for any $m$ in $\units$.
\end{lemma}

In particular, $\det (\lambda, m \theta) = \det(\lambda, \theta)$ if and only if $m \in \sqrt{1}$.  Observe
that this immediately implies Theorem \ref{theorem:isomorphismcount} : the set of volumes is
in bijection with $\units / \{\pm 1\}$, and two volumes define isomorphic Heegaard pairs if
and only if they are in the same coset of $\units / \sqrt{1}$. 

We conclude that to prove Theorem \ref{T:unstabilizedHeegaardpairs}, it is enough by Theorem \ref{T:the volume again} to prove that
everything in $\sqrt{1}$ is the determinant of an isometry.  We shall see that this
problem can be solved in each $p$-component independently and pieced together to get the general solution.  As
a technical tool, we will need the following result.

\begin{lemma} \label{lemma:8.2}
Let $H$ be a linked group whose smallest elementary divisor $\tau$ is even.  Then the following statements are equivalent:
\begin{enumerate}[a)]
\item Every $x \in H$ such that $\tau x = 0$ satisfies $x^2 \in \leg{2}{\tau}$ (here we have abbreviated $x \cdot x$ to $x^2$ and
$\leg{2}{\tau}$ is the subgroup of $\Q/\Z$ generated by $\frac{2}{\tau}$).
\item The lowest block of the $2$-component of $H$ is even (in the sense of Definition \ref{definition:evenodd}).
\end{enumerate}
\end{lemma}

\begin{proof}
Let $\{ p_i \}$ be the primes dividing $|H|$, with $p_1 = 2$, and let $\tau = \prod_i p_i^{n_i}$ (some of the $n_i$'s may be zero here,
but $n_1 >0$).  The group $H$ splits as an orthogonal direct sum of its $p_i$ components $H_i$, and every $x \in H$ can be written
uniquely as $x = \sum_i x_i$ with $x_i \in H_i$.  If $x_i$ has order $p_i^{r_i}$ then $x$ has order $\prod_i p_i^{r_i}$.  Thus
if $\tau x = 0$ we must have $r_i \leqslant n_i$ for all $i$.  Furthermore, $x^2 = \sum_i x_i^2$ and
$$
x_i^2 \in \leg{1}{p_i^{r_i}} \subset \leg{1}{p_i^{n_i}} = \leg{M_i}{\tau}
$$
where $M_i = \prod_{i \not= j} p_j^{n_j}$.  For odd primes (\ie $i>1$), the number $M_i$ is even and so $x_i^2 \in \leg{2}{\tau}$, but
for $i=1$ the number $M_i$ is odd.  Hence $x^2 \in \leg{2}{\tau}$ for all $x$ satisfying $\tau x = 0$ if and only if
$x_1^2 \in \leg{2}{\tau}$ for all $x_1 \in H_1$ satisfying $2^{n_1} x_1 = 0$.  The lowest block $B_1$ of $H_1$ consists of
elements all of which satisfy $2^{n_1} x_1 = 0$; their self linkings $x_1^2$ are in $\leg{1}{2^{n_1}}$, and are in $\leg{2}{\tau}$
if and only if they are in $\leg{2}{2^{n_1}}$.  Thus in this case $B_1$ must be even.  Conversely, suppose
$B_1$ is even. The group $H_1$ splits
orthogonally into blocks $B_1 \oplus B_2 \oplus \cdots \oplus B_k$ where each $B_i$ is a free $\Z_{2^{s_i}}$-module, with
$n_1 = s_1 < s_2 < \cdots < s_k$.  If $x_1 \in H_1$, write $x_1 = \sum_{i=1}^k y_i$, with $y_i \in B_i$.  We then have $2^{n_1} x_1 = 0$ if
and only if $2^{n_1} y_i = 0$ for all $i$, which is true if and only if $y_i \in 2^{s_i - n_1} B_i$ for all $i$.  Then
$x_1^2 = \sum y_i^2$ (by orthogonality) and
$$
y_i^2 \in \leg{2^{2 (s_i - n_1)}}{2^{n_1}} = \leg{2^{s_i - n_1}}{2^{n_1}} \subset \leg{2}{2^{n_1}} \text{ for all } i>1;
$$
but $y_1^2 \in \leg{2}{2^{n_1}}$ also since $B_1$ is even.  This concludes the proof.
\end{proof}

\begin{lemma} \label{lemma:8.15}
$\Lambda^r H$ is naturally isomorphic to the direct sum of $\Lambda^r H_p$ over all $p$ which divide $\tau$.
If $\theta$ is an
orientation (volume) on $H$ then its projection $\theta_p$ in $\Lambda^r H_p$ is an orientation (volume) on $H_p$.
\end{lemma}

\begin{proof}
The tensor power $H^r$ splits into the direct sum $H_p^r$ over all $p$ since $H_p \otimes H_q = 0$ if $p \not= q$.
Likewise, the kernel of $H^r \to \Lambda^r H$ splits into its $p$-component parts, and so we get a natural direct sum
$\Lambda^r H = \bigoplus_{\text{all } p} \Lambda^r H_p$.  But if $p \nmid \tau$ then $\rk H_p < r$ and so $\Lambda^r H_p = 0$, proving
the first statement.  Note that $\Lambda^r H_p$ is precisely the $p$-component of $\Lambda^r H$, which is
$\simeq \Z_{p^n}$ if $p^n$ is the largest power of $p$ dividing $\tau$.  If now $\theta \in \Lambda^r H$, we may write
$\theta = \sum_{p | \tau} \theta_p$.  Thus if $\theta$ generates $\Lambda^r H$, then $\theta_p$ must generate $\Lambda^r H_p$.
Finally, if $\theta$ is determined up to sign, so is $\theta_p$.
\end{proof}

Suppose now that $h : H \to H$ is an isometry.  Hence $h$ takes each $p$-component into itself, so $h$ splits into a 
direct sum of
maps $h_p$ on $H_p$; conversely the maps $h_p$ define $h$ on $H$.  Furthermore, the action of $h$ on $\Lambda^r H$ is
just multiplication by $\det h \ \mod \tau$ and hence the action of $h_p$ on 
$\Lambda^r H_p$ is also multiplication by $\det h$.
But it is also multiplication by $\det h_p \ \mod p^n$ (where $p^n \parallel \tau$), since $\Lambda^r H_p \simeq \Z_{p^n}$; 
in other words:

\begin{lemma} \label{lemma:8.16}
If $h$ is an endomorphism of $H$, then $\det h_p \equiv \det h \ \mod p^n$ for 
every $p$ dividing $\tau$, where $p^n \parallel \tau$.  Thus $\det h_p$ is determined by $\det h$.  
Conversely $\det h$ is determined by the values of $\det h_p$ (the proof of this is by the Chinese Remainder Theorem).
\end{lemma}

If $H$ has linking $\lambda$ then, by virtue of the orthogonality of distinct primary components, $\lambda$ splits into the
direct sum of linkings $\lambda_p$ on $H_p$.  Thus $h$ is an isometry of $H$ if and only if $h_p$ is so for each $p$.  Let now
$\tau = \prod_i p_i^{n_i}$.

\begin{lemma} \label{lemma:8.17}
$\sqrt{1} \ \mod \tau$ splits into the direct product of the groups $\sqrt{1} \ \mod p_i^{n_i}$.
\end{lemma}
\begin{proof}
Observe that $e \in \sqrt{1}\ \mod \tau$ means that $e \in \Z_\tau$ and $e^2 \equiv 1 \ \mod \bar{\tau}$.  This means that
$e^2 \equiv 1 \ \mod p_i^{\bar{n}_i}$, where $\bar{n}_i = n_i$ if $p_i$ is odd or if $p_i = 2$ and $\lambda$ is odd, but
$\bar{n}_i = n_i +1$ if $p_i = 2$ and $\lambda$ is even, which by Lemma \ref{lemma:8.2} is true if and only if $\lambda_2$ is
even.  Hence $e$ reduced $\mod p_i^{n_i}$ is in $\sqrt{1} \ \mod p_i^{n_i}$.  Conversely let
$e_i \in \sqrt{1} \ \mod p_i^{\bar{n}_i}$, that is $e_i \in \Z_{p_i^{n_i}}$ be such that 
$e_i^2 \equiv 1 \ \mod p_i^{\bar{n}_i}$;
by the Chinese Remainder Theorem again, there is a unique $e \in \Z_\tau$ such that $e \equiv e_i \ \mod p_i^{n_i}$, and we
find $e^2 \equiv e_i^2 \equiv 1 \ \mod p_i^{\bar{n}_i}$ which implies $e^2 \equiv 1 \ \mod \bar{\tau}$.
\end{proof}

\begin{corollary}
An element $a \in \sqrt{1} \ \mod \tau$ is the determinant of an isometry of $(H, \lambda)$ if and only if its
reduction $\mod p_i^{n_i}$ is the determinant of an isometry of $(H_{p_i},\lambda_{p_i})$.
\end{corollary}

We have now reduced the proof of Theorem \ref{T:unstabilizedHeegaardpairs} to the proof of:

\begin{lemma} \label{lemma:5.18}
Let $H$ be a linked $p$-group with $\tau = p^n$ and $e \in \sqrt{1} \ \mod p^n$. Then there is an isometry of $H$ with determinant $e$.
\end{lemma}
\begin{proof}
If $\lambda$ is odd, by \cite{Bu} the linked group $H$ has an orthogonal splitting of the form $(x) \oplus H_0$, where
$(x)$ is the cyclic subgroup generated by an element $x$ of order $p^n$ satisfying $x^2 = \frac{u}{p^n}$ with $p \nmid u$.  
The map $h$ which takes $x$ to $e x$ and which is the identity on $H_0$ is then an {\it isometry} and its determinant is 
clearly $e$.
So now let $p=2$ and $\lambda$ be even.  In this case, $\sqrt{1}$ consists of all $e$ mod $2^n$ such that 
$e^2 \equiv 1 \ \mod 2^{n+1}$.
There are four square roots of $1 \ \mod 2^{n+1}$ (when $n \geqslant 2$), namely $\pm 1$ and $2^n \pm 1$, but $\mod 2^n$ 
these give
only two distinct elements $\pm 1$ in $\sqrt{1}$.  Thus we must simply exhibit an isometry with determinant $-1 \ \mod 2^n$.
By Lemma \ref{lemma:8.2}, the even nature of $\lambda$ implies that the $2^n$-block of $H$ is even.  By the classification 
of linked $2$-groups in \S \ref{SS:classifying linked p-groups, p=2}, the linked group $H$ has 
an orthogonal splitting of the form
$Q \oplus H_0$, where $Q \cong \Z_{2^n} \oplus \Z_{2^n}$, generated by let us say $x, y$ of order $2^n$, 
and where $Q$ has a linking matrix
equal to one of $\frac{1}{2^n} \mattwotwo{0}{1}{1}{0}$ or $\frac{1}{2^n} \mattwotwo{2}{1}{1}{2} \ \mod ~ 1$.  Clearly
interchanging $x$ and $y$ is an isometry on either form, and extending by the identity on $H_0$ gives an isometry of $H$
with determinant $\equiv -1 \ \mod 2^n$.  This proves the lemma and concludes the proof of Theorem \ref{T:unstabilizedHeegaardpairs}.
\end{proof}

\subsection{Problem 6}
\label{SS:remarks on problem 6}

In Theorem~\ref{T:partial normal form} we found a partial normal form for the double coset associated to a symplectic matrix $\cH$, and in $\S$\ref{SS:uniqueness questions} we investigated its non-uniqueness.  We raised the question of whether the submatrix $\cQ^{(2)}$ could be diagonalized.  In fact, the following is true:
\begin{proposition}
\label{P:linking form can be diagonalized}
Let $W$ be a 3-manifold which is defined by a Heegaard splitting. Let $H=H_1(W;\mathbb Z)$ and let $T$ be the torsion subgroup of $H$. Let $t$ be the rank of $T$, so that $T$ is a direct sum of cyclic groups of order $\tau_1,\dots,\tau_t$, where each $\tau_i$ divides $\tau_{i+1}$.  Assume that every $\tau_i$ is odd. Then $T$ is an abelian group with a linking, and there is a choice of basis for $T$ such that the linking form for $T$ is represented by a $t\times t$ diagonal matrix. 
\end{proposition}   
\begin{proof} 
Consider, initially, a fixed $p$-primary component $T(p)$ of $T$ and its splitting $T(p) = T_1 \oplus \cdots \oplus T_\nu$ into 
cyclic groups $T_j$ of prime power order $p^{e_j}$.  The $T_j$'s may be collected into subsets consisting of groups of like order. 
Keeping notation adopted earlier, consider a typical such subset $T_{\rho +1}, \dots, T_{\rho + k}$ containing all cyclic summands of 
$T(p)$ of order $p^{\mu + 1}$.  Let $g_{\rho+1}, \dots, g_{\rho + k}$ generate these summands.  Define a new linking $\lambda'$ on 
$T_{\rho + j}$ ($j = 1, \dots, k$) by the rule:
\begin{equation} \label{equation:normalform18}
\lambda' (g_{\rho+i}, g_{\rho+j}) = 
\left\{ 
\begin{array}{ll}
\frac{| \mathcal A_{\rho+1} |}{p^{e_{\rho+1}}} \ \mod ~ 1 & \text{ if } i = j = 1 \\
\frac{1}{p^{e_{\rho+1}}} \ \mod ~ 1 & \text{ if } i = j = 2, \dots, k \\
0 & \text{ if } i \neq j.
\end{array}
\right.
\end{equation}
There is an induced linking $\lambda'$ on $T(p)$ obtained by direct summing the linkings on all of the cyclic summands, and thus
an induced linking on $T$ obtained by taking the orthogonal direct sum of all the $p$-primary summands.  We will also denote this
by $\lambda'$.  By Theorem~\ref{T:Seifert's theorem on the linking form}, the linking on $T$ is determined entirely by the 
quadratic residue characters of linkings on the $p$-primary summands of $T$.  It follows that $(T, \lambda')$ is 
equivalent as a linked group to $(T, \lambda)$. 

To complete the proof we need only note that by 
Theorem~\ref{T:fundamental theorem for f.g. abelian groups} the generators $g_{ij}$ of the cyclic summands of prime power order determine 
the generators $y_i$ of the cyclic summands of order $\tau_1,\dots,\tau_t$.  This follows from (\ref{E:gid from y_i}) of 
Theorem~\ref{T:fundamental theorem for f.g. abelian groups}.  Therefore there is a $t\times t$ matrix which also defines 
$\lambda'$, and $\lambda'$ is equivalent to $\lambda$.  The proof is complete.
\end{proof}

\begin{remark} \label{R:diagonalize all blocks?}
One might be tempted to think that Proposition~\ref{P:linking form can be diagonalized} implies that 
there is a matrix in the same double coset as the matrix $\cH'$ in (\ref{E:partial normal form for cH}) of 
Theorem~\ref{T:partial normal form} in which the blocks $\cP^{(2)}, \cQ^{(2)}$ are both diagonal. 
Suppose we could prove that.  Then for each $j = 1, \dots, t$ choose $r_j, s_j$ so that $r_j q_j - \tau_j s_j = 1$. 
Define $\cR^{(3)} = \diag (r_1, \dots, r_t)$ and $\cS^{(3)} = \diag (s_1, \dots, s_t)$.  With these choices it is easy 
to verify that $\mattwotwo{\cR^{(3)}}{\cP^{(2)}}{\cS^{(3)}}{\cQ^{(3)}}$ is symplectic.  If so, that would imply that 
$\cR^{(2)}$ and $\cS^{(2)}$ also are diagonal.   
However, while we have learned that there is a change in basis for $T$ in which $\cQ^{(2)}$ is diagonal, we do not know 
whether this change in basis  preserves the diagonal form of the matrix $\cP^{(2)}$.  Therefore we do not know whether 
it is possible to find a representative of the double coset in which all four blocks are diagonal.  
Proposition~\ref{P:linking form can be diagonalized} tells us that there is no reason to rule this out.  The discussion in 
$\S$\ref{SS:uniqueness questions} also tells us that it might be possible.  On the other hand, the fact that such a 
diagonalization cannot always be achieved when there is 2-torsion tells us that the proof would have to be deeper than the work we 
have already done. $\|$
\end{remark}

\begin{example} In spite of the difficulties noted in Remark~\ref{R:diagonalize all blocks?}, we are able to construct a very large class of  examples for which all four blocks are diagonal.  We construct our examples in stages:  
\begin{itemize}
\item First, consider the case where our 3-manifold $W(h_{p,q})$ is a lens space of type $(p,q)$.  Then $W(h_{p,q})$ admits a genus 1 Heegaard splitting with gluing map that we call  $h_{p,q}$, where $p$ is the order of $\pi_1(W(h_{p,q}))$.   The symplectic image of $h_{p,q}$ will be 
$\mattwotwo{r}{p}{s}{q}$, where $rq - ps = \pm 1$.  
\item Next, consider the case when our 3-manifold $W(\tilde{h})$ is the connect sum of $g$ lens spaces of types $(p_1,q_1),\dots, (p_g,q_g)$, so that it admits a Heegaard splitting of genus $g$. Think of the Heegaard surface as the connect sum of $g$ tori. The restriction of the gluing map $\tilde{h}$ to the $i^{\rm th}$ handle will be $h_{p_i,q_i}$, so that the symplectic image of the gluing map will be $M = \mattwotwo{R}{P}{S}{Q}$, where $P={\rm diag}(p_1,\dots,p_g), \ \ R = {\rm diag}(r_1,\dots,r_g), \ \ S = {\rm diag}(s_1,\dots,s_g),$\\ $Q= {\rm diag}(q_1,\dots,q_g)$.  
\item Finally, consider the class of 3-manifolds $W(\tilde{f}\tilde{h})$ of Heegaard genus $g$ which are defined by the gluing map $\tilde{f}\tilde{h}$, where $\tilde{f}$ is any element in the kernel of the natural homomorphism $\tilde{\Gamma}_g\to Sp(2g,\Z)$.   The fact that $\tilde{f}$ has trivial image in $Sp(2g,\Z)$ shows that the symplectic image of the gluing map for $W(\tilde{f}\tilde{h})$ will still be $M$. Thus we obtain an example for every element in the Torelli group, for every choice of  integers $(p_1,q_1),\dots, (p_g,q_g)$. $\|$
 \end{itemize}
\end{example}

\newpage
\section{Postscript : Remarks on higher invariants}
\label{S:a remark on higher invariants}

In this section, we make a few comments about the search for invariants
of Heegaard splittings coming from the action of the mapping class group
on the higher nilpotent quotients of the surface group (\ie the higher
terms in the Johnson-Morita filtration).  In this paper,
our invariants have come from 3 sources:
\begin{enumerate}
\item [(1)] The abelian group $H_1(W)$ of the 3-manifold $W$.
\item [(2)] The linking form on the torsion subgroup of $H_1(W)$.
\item [(3)] The presentation of $H_1(W)$ arising from the Heegaard splitting.
\end{enumerate}

With regard to (1), It is easy to see that there is a natural generalization. The classical Van Kampen Theorem shows  that the Heegaard gluing map $\tilde{h}$ determines a canonical presentation for $G = \pi_1(W)$ which arises via the action of $h$ on $\pi$.  This action determines in a natural way a presentation for $G/G^{(k)}$,  the $k^{\rm th}$ quotient group in the lower central series for $G$.  We do not know of systematic studies of these invariants of the fundamental groups of closed, orientable 3-manifolds. 

With regard to (2) and (3), if $\pi_1$ is the fundamental group of the Heegaard surface (which in
this section we will consider to be a surface with 1 boundary component corresponding
to a disc fixed by the gluing map -- this will make $\pi_1$ a free group), then
$H_1(W)$ is the quotient of the abelian group $\pi_1 / \pi_1^{(2)}$ by the
two lagrangians arising from the handlebodies.  The obvious generalization
of this is a quotient of the free nilpotent group $\pi_1 / \pi_1^{(k)}$.  Since
it is unclear what the appropriate generalization of the linking form to this
situation would be, one's first impulse might be to search for presentation invariants.  

Now, it is easy to see that the quotient of $\pi_1 / \pi_1^{(k)}$ by one of the ``nilpotent
lagrangians'' is another free nilpotent group.  Our presentation is thus a surjection
$\pi : N_1 \rightarrow N_2$, where $N_1$ is a free nilpotent group.  The invariants
of presentations of abelian groups arise from the fact that automorphisms of the presented
group may not lift to automorphisms of the free abelian group.  Unfortunately, the
following theorem says that no further obstructions exist:

\begin{theorem}
\label{theorem:liftingautos}
Let $\pi : N_1 \rightarrow N_2$ be a surjection between finitely generated nilpotent groups, where $N_1$ is a free
nilpotent group.  Also, let $\phi$ be an automorphism of $N_2$.  Then $\phi$ may be lifted to an automorphism of $N_1$
if and only if the induced automorphism $\phi_{\ast}$ of $N_2^{\Ab}$ can be lifted to
an automorphism of $N_1^{\Ab}$.
\end{theorem}

The key to proving Theorem \ref{theorem:liftingautos} is the following criterion
for an endomorphism of a nilpotent group to be an automorphism.  It is
surely known to the experts, but we were unable to find an appropriate reference.

\begin{theorem}
\label{theorem:autocriterion}
Let $N$ be a finitely generated nilpotent group and let $\psi : N \rightarrow N$ be an
endomorphism.  Then $\psi$ is an isomorphism if and only if the induced
map $\psi_{\ast} : N^{\Ab} \rightarrow N^{\Ab}$ is an isomorphism.
\end{theorem}
\begin{proof}
The forward implication being trivial, we prove the backward implication.  The proof
will be by induction on the degree $n$ of nilpotency.  If $n=1$, then $N$ is 
abelian and there is nothing to prove.  Assume, therefore, that $n>1$ and that
the theorem is true for all smaller $n$.  We begin by observing that since finitely generated nilpotent groups
are Hopfian, it is enough to prove that $\psi$ is surjective.  Letting
$$N = N^{(1)} \vartriangleright N^{(2)} \vartriangleright \cdots \vartriangleright N^{(n)} \vartriangleright N^{(n+1)} = 1$$
be the lower central series of $N$, we have an induced commutative diagram
$$\begin{CD}
1 @>>> N^{(n)}    @>>> N          @>>> N / N^{(n)} @>>> 1\\
@.     @VV{\psi}V      @VV{\psi}V      @VVV             @.\\
1 @>>> N^{(n)}    @>>> N          @>>> N / N^{(n)} @>>> 1
\end{CD}$$
Since $N / N^{(n)}$ is an $(n-1)$-step nilpotent group, the inductive hypothesis implies that
the induced endomorphism of $N / N^{(n)}$ is an isomorphism.  The five lemma therefore
says that to prove that the map
$$\psi : N \longrightarrow N$$
is surjective, it is enough to prove that the map
$$\psi : N^{(n)} \longrightarrow N^{(n)}$$
is surjective.  Now, $N^{(n)}$ is generated by commutators of weight $n$ in the
elements of $N$.  Let $\beta$ be a bracket arrangement of weight $n$ and let
$\beta(g_1,\ldots,g_n) \in N^{(n)}$ with $g_i \in N$ be some commutator 
of weight $n$.  Since $\psi$ induces an isomorphism of $N / N^{(n)}$, we can
find some $\tilde{g}_1,\ldots,\tilde{g}_n \in N$ and $h_1,\ldots,h_n \in N^{(n)}$ so that 
$$\psi(\tilde{g}_i) = g_i h_i$$ 
for all $i$.  Hence $\psi$ maps $\beta(\tilde{g}_1,\ldots,\tilde{g}_n)$
to $\beta(g_1 h_1,\ldots,g_n h_n)$.  However, since $N^{(n)}$ is central we have
that
$$\beta(g_1 h_1,\ldots,g_n h_n) = \beta(g_1,\ldots,g_n),$$
so we conclude that $\beta(g_1,\ldots,g_n)$ is in $\psi(N^{(n)})$, as desired.
\end{proof}

\noindent
We now prove Theorem \ref{theorem:liftingautos}.

\begin{proof}[Proof of Theorem \ref{theorem:liftingautos}]
Let $\{g_1,\ldots,g_k\}$ be a free nilpotent generating set for $N_1$, and let $\rho$ be an automorphism
of $N_1^{\Ab}$ lifting $\phi_{\ast}$.  Also, let $\overline{g}_i \in N_1^{\Ab}$ be the image of $g_i$.  Now,
pick any lift $h_i \in N_1$ of $\rho(\overline{g}_i)$.  Observe that by assumption $\pi(h_i)$ and
$\phi(\pi(g_i))$ are equal modulo $[N_2,N_2]$.  Since the restricted map $\pi : [N_1,N_1] \rightarrow [N_2,N_2]$
is easily seen to be surjective, we can find some $k_i \in [N_1,N_1]$ so that $\pi(h_i k_i) = \phi(\pi(g_i))$.  Since
$N_1$ is a free nilpotent group, the mapping
$$g_i \mapsto h_i k_i$$
induces an endomorphism $\tilde{\phi}$ of $N_1$ which by construction lifts $\phi$.  Moreover,
Theorem \ref{theorem:autocriterion} implies that $\tilde{\phi}$ is actually an automorphism, as desired.
\end{proof}

\begin{remark} \rm Theorem~\ref{theorem:liftingautos} does not destroy all hope for finding invariants of presentations, as there may be obstructions
to lifting automorphisms to automorphisms which arise ``geometrically''.
However, it makes the search for obstructions much more subtle. Moreover, we note that in \cite{L-M} Y. Moriah and M. Lustig used the presentation of $\pi_1(W)$ arising from a Heegaard splitting to prove that certain Heegaard splittings of Seifert fibered spaces are in fact inequivalent. Their subsequent efforts to generalize what they did (\cite{L-M-1993}) show that the problem is difficult, and the final word has not been said on invariants of Heegaard splittings that arise from the associated presentation of $\pi_1(W)$. $\|$
\end{remark}



\end{document}